\documentclass[a4paper,12pt]{scrartcl}
\usepackage[utf8]{inputenc}
\usepackage[english]{babel}
\usepackage[fixlanguage]{babelbib}
\usepackage[noadjust]{cite}
\usepackage{amssymb, 
amstext, amsopn, amsthm, amscd, amsxtra, amsfonts,amsbsy}

\usepackage{colonequals}
\usepackage{enumerate}
\usepackage{graphicx}%
\usepackage[all]{xy}
\usepackage{bm}
\usepackage[normalem]{ulem}
\usepackage{cancel}
\usepackage[shortlabels]{enumitem}
\usepackage{comment}

\usepackage{verbatim,shellesc}
\ShellEscape{cp /usr/local/share/latexmk/LatexMk ./LatexMk}

\usepackage{amsthm,amssymb,mathrsfs,mathtools,empheq}
\mathtoolsset{showonlyrefs}
\usepackage{color,stmaryrd,bbm}
\usepackage{fullpage}

\usepackage{todonotes}

\newcommand{\embed}{\hookrightarrow }
\newcommand{\loc}{\mathrm{loc}}
\newcommand{\step}{\mathrm{step}}

\newcommand\eps{\varepsilon}

\newcommand{\toup}{\nearrow}
\newcommand{\then}{\Rightarrow}

\usepackage{scalerel}
\usepackage{stackengine,wasysym}

\newcommand\reallywidetilde[1]{\ThisStyle{%
  \setbox0=\hbox{$\SavedStyle#1$}%
  \stackengine{-.1\LMpt}{$\SavedStyle#1$}{%
    \stretchto{\scaleto{\SavedStyle\mkern.2mu\AC}{.5150\wd0}}{.6\ht0}%
  }{O}{c}{F}{T}{S}%
}}

\def\test#1{$%
  \reallywidetilde{#1}\,
$\par}
\newcommand\newtilde[1]{\mbox{\test{#1}}}

\usepackage{hyperref}%
\hypersetup{
	urlcolor = red,
	colorlinks = true,
	linkcolor = blue,
	citecolor = blue,
	linktocpage = true,
	pdftitle = {Stochastic calculus for the Ebin-Marsden approach},
	pdfauthor = {Zdzis{\l}aw, Brze{\'z}niak, Mario Maurelli, Alexander Schmeding},
	bookmarksopen = true,
	bookmarksopenlevel = 1,
	unicode = true,
	hypertexnames =false
}%
\usepackage{mathrsfs}
\usepackage{dsfont}

\newcommand\etazero{\bm{\eta}_0} 

\newcommand\dela[1]{}

\numberwithin{equation}{section}

\swapnumbers
\newtheoremstyle{standard}
{16pt} 
{16pt} 
{} 
{} 
{\bfseries}
{} 
{ } 
{{\thmname{#1~}}{\thmnumber{#2.}}\thmnote{~(#3)}} 

\newtheoremstyle{kursiv}
{16pt} 
{16pt} 
{\itshape} 
{} 
{\bfseries}
{} 
{ } 
{{\thmname{#1~}}{\thmnumber{#2.}}\thmnote{~(#3)}} 

\theoremstyle{standard}
\newtheorem{defn} [subsection]{Definition}

\newtheorem{rem} [subsection]{Remark}
\newtheorem{remark} [subsection]{Remark}

\newtheorem{setup} [subsection]{}

\newtheorem{hh}[subsection]{Hypothesis}
\theoremstyle{kursiv}
\newtheorem{thm}[subsection]{Theorem}

\newtheorem{prop} [subsection]{Proposition}
\newtheorem{cor} [subsection]{Corollary}
\newtheorem{lem} [subsection]{Lemma}
\newtheorem{lemma}[subsection]{Lemma}

\newcommand{\cA}{\ensuremath{\mathscr{A}}}

\newcommand{\cF}{\ensuremath{\mathscr{F}}}
\newcommand{\cG}{\ensuremath{\mathscr{G}}}

\newcommand{\cL}{\ensuremath{\mathscr{L}}}

\newcommand{\cP}{\ensuremath{\mathscr{P}}}

\newcommand{\E}{\ensuremath{\mathbb{E}}}

\newcommand{\R}{\ensuremath{\mathbb{R}}}

\newcommand{\N}{\ensuremath{\mathbb{N}}}

\DeclareMathOperator{\id}{id}
\DeclareMathOperator{\Diff}{Diff}

\newcommand{\Dmu}[1][s]{\mathrm{Diff}^{#1}_\mu(K)}
\newcommand{\tr}{\mathrm{tr}}
\newcommand{\rK}{\mathscr{H}}

\newcommand{\VFsmu}[1][s]{\mathfrak{X}^{#1}_\mu (K)}

\newcommand{\Frechet}{Fr\'{e}chet }
\newcommand{\coloneq}{\colonequals}

\date{}
\title{The Ebin-Marsden toolbox for stochastic PDEs:\\stochastic Euler equations}
\author{Zdzis{\l}aw Brze{\'z}niak\footnote{Department of Mathematics, University of York, YO10 5DD Heslington, York, United Kingdom: \href{mailto:zdzislaw.brzezniak@york.ac.uk}{zdzislaw.brzezniak@york.ac.uk}}, Mario Maurelli\footnote{Dipartimento di Matematica, Universit\`a di Pisa, Largo Bruno Pontecorvo 5, 56127 Pisa, Italy: \href{mailto:mario.maurelli@unipi.it}{mario.maurelli@unipi.it}} and  Alexander Schmeding\footnote{Institut for matematiske fag, NTNU Trondheim, Alfred Getz' vei 1
Gl{\o}shaugen, 7034 Trondheim, Norway: \href{mailto:alexander.schmeding@ntnu.no}{alexander.schmeding@ntnu.no}}}

\begin{document}

\maketitle

\begin{abstract}
The Ebin-Marsden theory is a powerful geometric framework for many PDEs from fluid dynamics. In this paper we provide a toolbox to apply the Ebin-Marsden approach to stochastic PDEs, combining tools from infinite-dimensional geometry and stochastic analysis. We showcase our approach in the context of incompressible Euler equation for an ideal fluid with additive noise. Among our main results there are: (i) local well-posedness of maximal solutions by using the Ebin-Marsden framework; (ii) a stochastic version of the celebrated no-loss-no-gain theorem.
\end{abstract}

\\[2.3mm]


\tableofcontents

\section{Introduction}

In recent years, there has been a renewed interest in stochastic or more generally, rough path versions of partial differential equations (PDE) for fluid dynamics, see e.g. \cite{Glatt-Holtz-Vicol-2014, CFH17,Crisan+Holm+Leahy+Nilssen_2022a,Crisan+Holm+Leahy+Nilssen_2022b,BrzFlaMau2016,LaC23}. Most prominently, stochastic versions of  Euler's equations for an ideal fluid have been investigated. Euler`s equations are a prime example of an equation which can be treated with geometric methods. The geometric approach due to Arnold \cite{Ar1966} seeks to rewrite the PDE as geodesic equations (whence as second order ordinary differential equations (ODE)) on an infinite-dimensional manifold. Following Arnold, Ebin and Marsden made these ideas rigorous and establish a basic solution theory for Euler's equations via geometric methods, cf. e.g. the introduction \cite{Ebin15} .
This circle of ideas is the foundation of geometric hydrodynamics. Among the important classical results of geometric hydrodynamics is the celebrated no-loss-no-gain theorem from \cite{EM70}. Roughly put, the no-loss-no-gain theorem states that the solution does not essentially depend on the Sobolev index chosen for the infinite-dimensional manifold on which the equation is solved, as long as the equation makes sense on a manifold of $H^{r}$-Sobolev mappings, a solution to it also solves the equation on the $H^s$-Sobolev maps  (with $s<r$ large enough to make sense) and there is neither loss nor gain in the existence time of the solution. Hence the geometric analysis is independent of the choice of manifold.

A natural question is whether a similar programme can be carried out in the stochastic setting, and more specifically for stochastic PDEs. Let us point out that, in the case of stochastic flows of diffeomorphisms (not stochastic PDEs), the Ebin-Marsden programme was realized by Elworthy \cite{Elw82} and Brze{\'z}niak and Elworthy in \cite{Brz+Elw_1996}. 
In \cite{MMS19} this Ebin-Marsden approach to stochastic PDEs has been applied to a stochastic version of the Euler equations. While \cite{MMS19} is a proof of concept, it does not investigate the finer points of stochastic analysis on infinite-dimensional manifolds. For example, a stochastic analogue of the no-loss-no-gain theorem is still missing. We remedy this in the current paper and expand the analytic toolbox for stochastic differential equations on infinite-dimensional (Hilbert) manifolds. As a foundation we give a self-contained presentation of foundational results for SDEs on Hilbert manifolds, such as differentiability in probability with respect to initial conditions. Since the manifolds of interest are infinite dimensional, special care has to be taken in the stochastic analysis (for example to avoid compactness arguments). Then we apply these results in the Ebin-Marsden framework and obtain well-posedness of maximal solution and a stochastic version of the celebrated no-loss-no-gain results. This Ebin-Marsden toolbox has several advantages, some of which will be explored in future papers: it yields SDEs on infinite-dimensional manifolds with \textit{smooth} coefficients; it can be applied to several other stochastic PDEs and with other types of noises; it works for any choice of the underlying manifold where the equation evolves.

Our first main result is the existence and uniqueness of a local \textit{maximal} solution to the Euler equations in Lagrangian form. To understand the statement let us fix notation: Let $K$ be a compact manifold of dimension $d>0$ together with a Riemannian volume form $\mu$ on $K$. We consider a stochastic version of the Euler equation on $K$ for an ideal fluid. Formulated in Lagrangian form, this is the equation for position $\Phi$ and velocity $V$ of fluid particles contained in $K$. Since the motion of fluid particles can be modelled as a flow of diffeomorphisms on $K$, the Euler equation in Lagrangian form is a stochastic differential equation in the infinite dimensional manifold of $H^s$-Sobolev diffeomorphisms and reads
\begin{equation}\label{stoch_euler1}
 \mathrm{d}\eta_t = B(\eta_t)\mathrm{d}t + \Sigma (\eta_t) \bullet \mathrm{d} W_t
\end{equation}
where $\eta=(\Phi,V)$ is the couple position and velocity, $\bullet \mathrm{d} W_t$ denotes Stratonovich noise, $B$ is the Ebin-Marsden drift, i.e.~the spray of the right invariant $L^2$-metric (possibly including a deterministic forcing term) and $\Sigma$ a suitable diffusion field; the noise $\mathrm{W}$ is white in time and smooth in space. For the precise assumptions on the fields and the noise $W$ we refer the reader to Section \ref{sec:noloss_nogain}. An advantage of this Ebin-Marsden formulation, as noted in \cite{EM70} in the deterministic case and in \cite{MMS19} in the stochastic case, is that $B$ is \emph{smooth} and the smoothness of $\Sigma$ on the manifold of $H^s$-mappings can be controlled by spatial regularity of the noise. Our results subsume the following informal statement (see Theorem \ref{thm:wellposed_Lagr} for the precise statement):\medskip

\textbf{Theorem A (Maximal solutions)}
\emph{Take $s > d/2 + 1$ and a suitable noise together with an initial value $\etazero$ in the Sobolev class $H^s$. Then there exists a (probabilistically strong) local maximal solution $\eta=(\eta_t)_{t\in [0,\tau)}$ to the Euler equations in Lagrangian form \eqref{stoch_euler1} in a space of $H^s$ maps.} 
\smallskip

To our knowledge this foundational result has not previously been established for stochastic Euler equations in the Ebin-Marsden setting. Concerning the known results in the literature we remark that \cite[Theorem 3.5]{MMS19} does not show maximality and other results as in \cite{Glatt-Holtz-Vicol-2014} take a different approach (they use a PDE setting and work at the Eulerian level, that is with equation \eqref{stoch_euler2}).

We then prove as our second main result a stochastic version of the no-loss-no-gain result from \cite{EM70} (see also \cite{BaM18}). 
Our results subsume the following informal version of the no-loss-no-gain result (see Theorem \ref{thm:noloss_nogain} for the precise statement): 
\medskip

\textbf{Theorem B (Stochastic no-loss-no-gain)} \emph{Take $s > d/2 + 1$ and a suitable noise together with an initial value $\etazero$. Let 
$\eta = (\eta_t)_{t \in [0,\tau)}$ be the corresponding maximal solution in a space of $H^s$ maps to \eqref{stoch_euler1}. Assume that $\etazero$ is of Sobolev class $H^{s+1}$, then, $\mathbb{P}$-a.s., the map $t \mapsto \eta_t , t \in [0, \tau )$, is well-defined and continuous with values in a space of $H^{s+1}$ Sobolev maps. It coincides with the (unique) maximal solution to \eqref{stoch_euler1} in the $H^{s+1}$ maps.} 
\smallskip

Note that while Theorem B is formulated for the Euler equation, its proof does not depend on the fact that we are dealing with the Euler equation. Indeed, minor modifications of Theorem B will provide the no-loss-no-gain theorem for stochastic versions of other PDE amenable to the Ebin-Marsden approach. 
Our No-Loss-No-Gain Theorem  (Theorem B, cf. Theorem \ref{thm:noloss_nogain}) can be explained in the following way.
Suppose  a certain stochastic differential  equation  (SDE) is locally well posed on a scale  $\mathscr{E}$ of Banach manifolds $M$, e.g. Banach spaces (in our case, the manifolds $M=H^s(K,N)$ of Sobolev functions are indexed by the integer regularity index $s$).
In particular, for every $M \in \mathscr{E}$ and every initial data $x \in M$,  there exists a local maximal solution defined on a random time interval  $[0,\tau(M,x) )$. Theorem B says that the maximal existence time $\tau(M,x)$ does not depend on $M$ (in our case, on the integer regularity index $s$). This result is of a similar nature to the celebrated Beale-Kato-Majda (BKM) condition, see \cite{Beale+Kato+Majda_1984}. In the stochastic case, \cite{Glatt-Holtz-Vicol-2014} has shown the BKM condition (though not explicitly, but in the form of an inequality) for Euler equations with additive noise in 2D, but the same argument works in 3D, and with linear multiplicative noise in 3D; \cite{CFH17} has proved the BKM condition for 3D Euler equations with transport noise; see also \cite{Manna+Panda_2019} for an extension to stochastic Boussinesq equations.
In the 2D case, the BKM condition
implies that $\tau(\etazero,M)=\infty$, for every manifold $M=T\Diff^s_\mu(K)$ with $s>2$, see \cite{Glatt-Holtz-Vicol-2014,BrzFlaMau2016,Crisan+Holm+Leahy+Nilssen_2022b,LaC23} and also Corollary \ref{cor-global for d=2}.

To prove Theorems A and B above, we need several tools from the theory of stochastic differential equations on Hilbert manifolds. In particular, for Theorem A, we need a result on local well-posedness of maximal solutions for SDEs on Hilbert manifolds; for Theorem B, following the deterministic argument in \cite{EM70}, we need a result on differentiability of solutions with respect to the initial conditions. These results are morally known, especially from Elworthy's book \cite{Elw82} on SDEs on Hilbert manifolds, which we have used as the main source for our statements and proofs on this topic. However we have chosen to give a complete and mostly self-contained treatment, also due to several small technical differences with respect to \cite{Elw82}\footnote{For example, we have slightly weaker regularity assumptions in the coefficients for It\^o formula on manifolds (It\^o formula \cite[Chapter VII, Lemma 9B]{Elw82} requires $C^2$ coefficients); in the statement of differentiability with respect to the initial conditions Theorem \ref{thm-Strat_diff_manifold}, we consider differentiability as $L^0(\Omega,C([0,T],H))$-valued maps rather than $C([0,T],L^0(\Omega,H))$ as morally in \cite[Chapter VII, Theorem 8E]{Elw82}.}. It should also be stressed that the proof of Theorem B requires care in the stochastic case. As we have mentioned, the proof of Theorem B uses, as in the deterministic case, the differentiability of solutions with respect to the initial conditions. However, a pathwise result on differentiability for SDEs is highly nontrivial in infinite dimensional spaces; indeed, even the existence of a flow solution to an infinite-dimensional SDE is trickier than in the finite-dimensional case. Luckily, it is enough for our argument to have differentiability in probability, which is available also in the infinite-dimensional setting.

Traditionally, Euler equations are formulated in the Eulerian form, asking for a vector field $u$ as a solution. In the case at hand, the Euler equations in Eulerian form read
\begin{equation}\label{stoch_euler2}
\mathrm{d}u_t +\Pi[\nabla_{u_t}u_t]\mathrm{d}t = \mathrm{d} W_t
\end{equation}
where $u_t$ is the (unknown) velocity field in $H^s$ and $\Pi$ is the Leray projection. As the Ebin-Marsden approach reformulates the equation in Lagrangian form, it is crucial in our setting to pass between the Eulerian formulation \eqref{stoch_euler2} and the Lagrangian formulation \eqref{stoch_euler1} of the differential equation at hand. However, in \cite{MMS19} the passage from Lagrangian to Eulerian form incurred a hefty loss in regularity of the solution. This unnatural  behaviour was due to the non-smoothness of certain mappings used in the proof for the equivalence of both formulations. In the present paper we remedy this problem and show that there is indeed no decay of regularity.
Our result subsume the following informal theorem (see Theorem \ref{thm:equivalence_flow_pde} and the results thereafter for the precise statement).
\medskip

\textbf{Theorem C} 
\emph{Fix $s>d/2+1$ and a noise $W$ with suitable assumptions. Then the Eulerian form \eqref{stoch_euler2} and the Lagrangian form \eqref{stoch_euler1} of the stochastic Euler equation are equivalent, in the following sense: $u(t)$ is a solution to the Eulerian form \eqref{stoch_euler2} if and only if $\eta(t)=(\Phi(t),u(t)\circ \Phi(t))$ is a solution to the Lagrangian form \eqref{stoch_euler1}, where $\Phi$ is the flow of $H^s$ diffeomorphisms driven by $u$, that is
\begin{align*}
\dot{\Phi}(t) = u(t,\Phi(t))\, dt,\quad \Phi(0)=\id.
\end{align*}
In particular, Theorem A and Theorem B hold also for the Eulerian form \eqref{stoch_euler2}.}
\smallskip

Again, while Theorem C is formulated for the Euler equation, the proof can be adapted with minor changes to other equations which are amenable to the Ebin-Marsden approach. 

We would like to stress that the purpose of the current work is to provide a general toolbox for the Ebin-Marsden approach to stochastic PDEs via geometry. Our results are agnostic of the geometry to the underlying manifold on which the equation evolves. This is one of the strengths of the Ebin-Marsden approach.
 
 In stochastic analysis, the Ebin-Marsden approach to PDEs was exploited to discuss variational principles in (at least) two groups of works. In the first group, deterministic, viscous PDEs as the Navier-Stokes equations have been studied by variational principles for stochastic processes, for example \cite{Cip+Cru07} for the 2D Navier-Stokes equation; see also \cite[Introduction]{MMS19} for further references. The second group of works, including our paper and the former paper \cite{MMS19}, deals with stochastic PDEs by variational principles. The paper \cite{MR1747615} deals with Euler equations with additive space-independent white-noise; for this noise, a random translation on the position allows to reduce these equations to the deterministic Euler equations, getting well-posedness and smooth dependence on the initial conditions. In \cite{Crisan+Holm+Leahy+Nilssen_2022a} certain PDEs with rough transport noise are treated in the regime of smooth solutions. Note that loc.cit. does neither treat local well-posedness of these equations, nor encompasses the results we are after in the stochastic setting. In the follow-up work  \cite{Crisan+Holm+Leahy+Nilssen_2022b} Euler equations with rough transport noise on a torus are treated in the Sobolev regime. Here local well-posedness is established together with a Beale-Kato-Majda (BKM) blow up criterion in the rough setting are developed; a continuity (but not differentiability) result with respect to initial condition and noise is also established. While somewhat similar to our work, the results of \cite{Crisan+Holm+Leahy+Nilssen_2022b} are not directly comparable to ours. The main reason is that they do not use the Ebin-Marsden setting (that is they do not study directly the Lagrangian equation \eqref{stoch_euler1} on an infinite-dimensional manifold of $H^s$ maps) to show local well-posedness or BMK criterion; they rather apply PDE-type arguments (together with tools for rough transport noise) and show that the Lagrangian motion of position and velocity is a critical point for a certain action functional. Moreover, they work with rough transport noise, not additive noise as here, and on the torus, not on a general compact manifold $K$.
 Some works deal with both deterministic viscous PDEs and stochastic PDEs. For example, \cite{Chen+Cruzeiro+Ratiu_2023} deals also with infinite-dimensional manifolds and shows the equivalence between deterministic and stochastic PDEs, such as compressible Euler and Navier-Stokes equations, and certain variational principles. However, also in this work no direct analysis of the Lagrangian equation \eqref{stoch_euler1} is performed, nor there is a no-loss-no-gain result; moreover, the noise in \cite{Chen+Cruzeiro+Ratiu_2023} is of transport-type with a Wiener process with no spatial dependence and the manifolds are not allowed to have boundary.
 Beside the Ebin-Marsden approach, there are several results in the literature on well-posedness and Beale-Kato-Majda criterion on stochastic Euler equations, with other (usually PDE-like) approach. We mention again \cite{Glatt-Holtz-Vicol-2014} for the case of additive and multiplicative noise and \cite{Crisan+Holm+Leahy+Nilssen_2022a,Crisan+Holm+Leahy+Nilssen_2022b,LaC23,CFH17} for the case of transport noise. While our results are restriced to the case of additive noise, it is worth pointing out that  the general stochastic toolbox developed in the present article is not dependent on the choice of noise (up to some assumptions). So the results in the present paper for stochastic equations on Hilbert manifolds are transferable to other type of noises which arise in the Ebin-Marsden approach. The details of this will be given elsewhere in upcoming work by the authors. 
 
We would like to mention that we believe that many of our results can be generalized to a framework of Banach manifolds modeled on fractional Sobolev spaces, see e.g. \cite{Brz+Elw_1996}, \cite{Brz+Elw_2000}, \cite{Brz+Carr_2003} and \cite{Brz+Raza_2016}  where this framework has been used. The construction of the required manifold of mappings has been investigated in the book \cite{IKT13}. However, to our knowledge there is no complete account of manifolds of fractional Sobolev mappings on manifolds with (smooth) boundary available in the literature.

The paper is organized as follows: Section \ref{sec:manifold_emb} provides essentials on infinite dimensional manifolds of Sobolev morphisms. Here we establish some useful fact on embeddings of manifolds of diffeomorphisms which seem to be new and are of independent interest. In Section \ref{sec:maximal_existence}, after recalling some facts on stopping times and SDEs on Hilbert manifolds, we establish a general local well-posedness and maximal existence result. Then Section \ref{sec-Strat_diff_manifold}, establishes differentiability in probability of the solution to the SDE with respect to the initial condition. In Section \ref{sec:noloss_nogain}, we show the local well-posedness of maximal solutions and the no-loss-no-gain result for the Lagrangian equation \eqref{stoch_euler1}, applying the previous differentiability result. Finally, in Section \ref{sec:Euler_Lagr} we transfer these results to the Eulerain equation \eqref{stoch_euler2}.
\smallskip

\textbf{Acknowledgements} We would like to thank M.~Bauer and K.~Modin for enlightening discussions on aspects of the present work and for pointing us towards several results in the literature we were unaware of. Further, M.M. thanks the Mathematical Institute at NTNU Trondheim for hospitality while part of this work was conducted and acknowledges support from the Istituto Nazionale di
Alta Matematica, group GNAMPA, through the project GNAMPA 2020 `SPDE in fluidodinamica', and from the Hausdorff Research Institute for Mathematics in Bonn, under the Junior Trimester Program `Randomness, PDEs and Nonlinear fluctuations'. A.S. acknowledges support from the Research Council
of Norway through project 302831 "Computational Dynamics and Stochastics on Manifolds" (CODYSMA). Part of this work was undertaken when M.M. was at Universit\`a degli Studi di Milano, Italy and when A.S. was at Nord University in Levanger, Norway.

\subsection*{Notation and conventions}
\addcontentsline{toc}{subsection}{Notation and conventions}
By $\mathbb{N}$ we will denote the set of natural numbers which for us is equal to $\{1,2,\cdots \}$ and write $\N_0 = \N \cup \{0\}$. 

\begin{setup}[Probability space and Stratonovich integration]
The probability space $(\Omega,\mathscr{A}, \mathbb{P})$ used throughout the whole paper will be fixed for once. Hence we will simply write a.s. or almost surely, without mentioning the probability measure $\mathbb{P}$. The notation $\mathbb{F}=(\mathscr{F}_t)_{t\in [0,\infty)}$ is used for a filtration satisfying the usual assumptions. The notation 
\begin{equation*}
\bullet\, dW
\end{equation*}
is used for Stratonovich integration, while $\circ$ denotes composition between maps.
\end{setup}

\begin{setup}[Bochner-Lebesgue spaces]
Here and in the rest of the article we fix the following standard notation for the Bochner-Lebesgue spaces (cf. \cite[22.7]{Sche97}):
Let $I$ be an interval and $E$ a separable Banach space. We endow both with their respective Borel $\sigma$-algebras. Then we write $L^p(I,E)$ for $p \in \N$ for the space of (equivalence classes) of measurable mappings  which are $L^p$-integrable with respect to Lebesgue measure. 
\end{setup}
\begin{setup}[Linear operators and bundles]
For Banach spaces $E,F$ we let $\cL (E,F)$ be the space of continuous linear mappings with the operator norm. More generally, we denote for two vector bundles $\mathbb{A},\mathbb{B}$ the associated vector bundle  $\mathcal{L}(\mathbb{A},\mathbb{B})$ of linear maps (cf.\ e.g.\ \cite[Lemma 1.2.12]{MR1330918}). If $\mathbb{A} = M \times E$ is a trivial bundle, we write shorter $\mathcal{L}(E,\mathbb{B})$ for the bundle of linear maps.
\end{setup}

\begin{setup}[Differentiable mappings]
With $C(X,Y)$ or $C^0(X,Y)$, we denote the set of continuous mappings between two topological spaces.  
We shall use the usual notation of $C^k, k \in \N_0$ to denote $k$-times continuously \Frechet differentiable mappings on open sets between Banach spaces. For such a mapping $f \colon U \rightarrow F$ on an open subset of the Banach space $E$ we write 
\begin{equation}\label{eq:Df_notation}
Df \colon U \rightarrow \cL (E,F), \quad u \mapsto D_u f 
\end{equation}
for its \Frechet derivative. We shall also write $Df(u) \coloneq Df(u;\cdot) \coloneq D_uf$ for the \Frechet derivative. 
If a map is $C^k$ for every $k \in \N$, we say it is smooth. 
\end{setup}

\begin{setup}[Manifolds]All manifolds will be assumed to be Hausdorff. If $M, N$ are smooth manifolds (modelled on Banach spaces), then $C^k (M,N)$ denotes the set of $C^k$-mappings from $M$ to $N$. If $f \colon M \rightarrow N$ is a $C^1$ map,  we denote by  $Tf \colon TM \rightarrow TN$  its tangent map. Further, if $x \in M$, we let $T_xf \colon T_x M \rightarrow T_{f(c)}N$ be the tangent map at $x$. In local coordinates (which we suppress in the following formula)  we put 
\begin{equation*}
Tf = (f,Df) \quad \text{ with } Tf(x,v)=(f(x), T_xf(v))=(f(x),D_xf(v)).
\end{equation*}
If $N=E$ is a Banach space, we shall abuse notation and write $Df$ for the second component of $Tf$ (coherently with formula \eqref{eq:Df_notation}). We will always use a symbol $K$ to denote a generic Riemannian compact manifold, possibly with smooth boundary $\partial K$. Further we assume that $K$ is orientable, i.e. that it admits an associated volume form $\mu$.
\end{setup}

\section{Preliminaries on infinite-dimensional manifolds}\label{sec:manifold_emb}

 We first recall basic information on infinite dimensional manifolds such as manifolds of (Sobolev type) mappings. Though we will mainly be interested in Hilbert manifolds (i.e.\ manifolds modelled on Hilbert spaces), most results can easily be formulated in a more general setting. The main result of this section will be an embedding of the groups of Sobolev diffeomorphisms as a split submanifold of a separable Hilbert space. For this we shall first provide some general results from infinite-dimensional geometry, which, to our knowledge, have not yet been recorded in the literature.

\addcontentsline{toc}{subsection}{Extending mappings from split submanifolds}
\subsection*{Extending mappings from split submanifolds}

In the present section we will assume that the spaces we are working with are Banach spaces and the manifolds are modelled on Banach spaces. See \cite{MR1666820} for an exposition of the basic theory of calculus and manifolds in the Banach setting. Most results in the current section remain valid for manifolds modelled on more general spaces.\footnote{Then the term differentiable map has to be understood in the sense of Bastiani calculus, see Definition \ref{defn:Bast} below. We refer to \cite{schmeding_2022} for more information and a detailed introduction to Bastiani calculus.} For some of our arguments boundedness assumptions are essential.
\begin{defn}
Let $E,F$ be Banach spaces and $U \subseteq E$ open. A $C^k$ map $f\colon U \to F$ 
\begin{enumerate}
\item \emph{has bounded $k$th derivative} (or alternatively is of class $C^k_b$) if 
$$\exists B >0 \text{ such that }\lVert D^kf(x)\rVert_{\text{op}} \leq B ,\quad  \forall x \in U$$
where $D^kf$ is the iterated $k$th \Frechet derivative and $\lVert \cdot \rVert_{\text{op}}$ denotes the operator norm on the respective space of $k$-linear map. 
\item is of class $C^k_{\text{glob}}$ if it is $C^\ell_b$ for all $0\leq \ell \leq k$.  
\item is of class $C^{k,1}$ (or $C^{k,1}_{\text{loc}}$) if its $k$th derivative is (locally) Lipschitz continuous.
\end{enumerate}
\end{defn}

Note that $C^0_b$ means that $f$ is bounded. In general, a function which is $C^k_b$ will not automatically be $C^\ell_b$ for $\ell < k$ as we are asking for global bounds.

In infinite-dimensional differential geometry, new classes of submanifolds become important. We recall them now for the readers convenience:

\begin{defn}
Let $M$ be a $C^r$-manifold and assume that for every chart $\varphi \colon M \supseteq U_\varphi \to V_\varphi \subseteq E_\varphi$ in an atlas of $M$ there is a sequentially
closed vector subspace $F_\varphi \subseteq E_\varphi$.
A $(C^r$-)submanifold of $M$ is a subset $N \subseteq M$ such that for each $x \in N$, there exists a chart $(U_\varphi, \varphi)$ with $x \in U_\varphi$ such that $\varphi (U_\varphi \cap N) = V_\varphi \cap F_\varphi$. We then call $(U_\varphi,\varphi)$ a submanifold chart. 
If all of the subspaces $F_\varphi$ are complemented subspaces of $E_\varphi$ (i.e.\ $E_\varphi \cong F_\varphi \times H$ as locally convex spaces, we call $N$ a split submanifold of $M$.
\end{defn}

Recall that if $M$ is a Hilbert manifold (i.e.\ all the modelling spaces are Hilbert spaces), then every submanifold is automatically split. Moreover, every submanifold of a finite dimensional manifold is already a split submanifold. See \cite{schmeding_2022} for more information. We shall now construct bounded extensions for maps defined on split submanifolds. For this we need to recall some terminology.

\begin{defn}
	Consider a manifold $M$ and a subset $\mathscr{S} \subseteq C^0 (M,\R)$ of continuous functions. We say that $M$
	\begin{enumerate}
		\item \emph{admits $\mathscr{S}$-bump functions} or is ($\mathscr{S}$-regular) if for every $x \in M$ and every open $x$-neighborhood $U \subseteq M$ there exists $\chi \in \mathscr{S}$ such that $\chi (M) \subseteq [0,1]$, $U = \chi^{-1} (]0,1])$ and $\chi(x)=1$.
		\item \emph{is $\mathscr{S}$-paracompact} if every open cover $(U_i)_{i\in I}$ of $M$ admits a subordinate partition of unity $(\chi_i)_{i \in I}$ of functions $\chi_i \in \mathscr{S}, i\in I$.
	\end{enumerate}
If $\mathscr{S} = C^k(M,\R)$ for some $k$, we shorten the notation and say that $M$ admits $C^k$-bump functions or is $C^k$-paracompact. Similar notation is in use for $\mathscr{S}=C^k_b (E,\R)$, where $E$ is a Banach space (the boundedness condition makes no sense on a manifold without additional structure). Note that $C^0$-paracompactness is equivalent to the usual topological notion of paracompactness.
\end{defn}

More information on $\mathscr{S}$-regularity and paracompactness can be found for example in \cite[Chapter 16]{MR1471480}.
For us the most important observation will be:

\begin{lem}\label{lem:Hilbert_paracompact}
	Every Hilbert space $H$ is $C^\infty$-paracompact and if $H$ is in addition separable, then $H$ is also $C^\ell_{b}$-paracompact for $\ell \in \{1,2\}$.
\end{lem}
\begin{proof}
Hilbert spaces are smooth paracompact by \cite[16.16 Corollary]{MR1471480}. For the other properties recall from \cite[Remark after 16.10 Theorem]{MR1471480} that every separable Hilbert space is $\mathcal{L}ip^2_{glob}$-paracompact. Here $\mathcal{L}ip^2_{glob}$ denotes $C^2$-functions whose second derivative satisfies a global Lipschitz condition. So $H$ is $C^2_b$-paracompact. For $\ell=1$, \cite[15.9]{MR1471480} shows that every closed subset of $H$ is the zero set of a $\mathcal{L}ip^\ell_{glob}$ function, so in particular of a $C^1_b$ function. Then \cite[15.3 Proposition]{MR1471480} implies that $H$ is $C^1_b$-regular and by \cite[16.10 Theorem]{MR1471480}, the space $H$ is $C^1_b$-paracompact. 
\end{proof}

\begin{remark}\label{rem:boundedness_bump}
To untangle the statement of Lemma \ref{lem:Hilbert_paracompact}, $C^2_b$-functions will not automatically be $C^1_b$. However, as $C^k$-bump functions have (by construction) $C^k$-flat points, the usual Taylor estimates, see \cite[p.162-163]{MR1471480}, imply that all elements of a $C^k_b$-partition of unity whose carrier is bounded are already contained in $C^k_{\text{glob}}$. In particular, these $C^2_b$ bump functions are $C^1_b$. This is the case we will use. Since there is no general inheritance of the bounds we chose to give Lemma \ref{lem:Hilbert_paracompact} in the somewhat technical way. Note that the statement of Lemma \ref{lem:Hilbert_paracompact} becomes false for non-separable Hilbert spaces.
\end{remark}

For finite dimensional manifolds it is well known that one can extend functions from a submanifold to a neighborhood, see e.g.\ \cite[Lemma 5.34]{MR2954043}. With the help of $C^k$-bump functions we can obtain similar results for split submanifolds. This is the content of the next lemmata. 

\begin{lem}\label{technical:Lemma}
	Let $M$ be a manifold modelled on a space $E$ which is $C^k$-regular. Fix a split submanifold $N \subseteq M$ together with an open subset $O \subseteq N$. Let $f \in C^k (N,F)$ for some $k \in \N_0 \cup \{\infty\}$ and a Banach space $F$.
	\begin{enumerate}
		\item for every $x \in N$, there is an open $x$-neighborhood $U_x \subseteq M$ and $f_x \in C^k (M,F)$ such that $f_x|_{U_x \cap N} = f|_{U_x \cap N}$.
		\item assume in addition that $M$ is metrisable, $C^k$-paracompact and $N$ is a closed subset of $M$. Then there exists an open neighborhood $U \subseteq M$ with $U \cap N = O$ and $g \in C^k (U,F)$ such that $g|_O = f|_{O}$
	\end{enumerate}
\end{lem}

\begin{proof}
	1. Pick a submanifold chart $(U_\varphi, \varphi)$ of $N$ such that $x \in U_\varphi$. Since $N$ is a split submanifold, $E=F_\varphi \times Y$ for suitable subspaces and the canonical projection $\pi_\varphi \colon E \to F_\varphi$ is continuous linear, whence smooth. After shrinking the chart domain, we may assume that $\varphi (U_\varphi) = V_\varphi = V_X \times V_Y \subseteq F_\varphi\times Y = E$ and $\varphi (U_\varphi\cap N ) = V_X \times \{0\}$.  
    As $E$ admits $C^k$ bump functions, we can pick a $\chi \in C^k (E,\R)$ such that there is a neighborhood $U_x  \subseteq U_\varphi$ of $x$ with $\chi\circ \varphi (U_x) \equiv 1$ and $\chi^{-1}( (0,1])=V_\varphi$. Now construct the function
	$$f_x \colon M \to F,\quad f_x(y) =
	\begin{cases}
		\chi (\varphi (y)) f(\varphi^{-1} (\pi_\varphi (\varphi (y)))& \text{ if } y \in U_\varphi \\
		0 & \text{ if } y \in M \setminus U_\varphi
	\end{cases}
	$$
	and observe that $f_x$ is a $C^k$-function with $f_x|_{U_x \cap N} = f|_{U_x\cap N}$.
	
	2. Since $O$ is open in the subspace $N\subseteq M$, we can pick $U \subseteq M$ open such that $U \cap N = O$. Note that $O$ is closed in $U$ as $N$ is closed in $M$. Let us construct an extension of $f|_O$ to $U$. Since $M$ is metrisable and $C^k$-paracompact, the topology of $M$ has a basis which is a countable union of locally finite families of carriers of $C^k$ functions, \cite[Theorem 16.15]{MR1471480}. As a subset of $M$ also $U$ is metrisable and admits a topological base with the same properties (which can be seen by just intersecting the base of $M$ with the open set $U$). Hence \cite[Theorem 16.15]{MR1471480} allows us to deduce that also $U$ is $C^k$ paracompact.
    Let now $\{(U_{\varphi_i},\varphi_i)\}_{i \in I}$ be a family of charts such that
	\begin{itemize}
		\item $O \subseteq \bigcup_{i \in I} U_{\varphi_i} \subseteq U$,
		\item every $\varphi_i \colon U_{\varphi_i} \to V_{X_i} \times V_{Y_i} \subseteq F_{\varphi_i} \times Y_i = E$ satisfies $\varphi_i (U_{\varphi_i}\cap O) = V_{X_i} \times \{0\}$.
	\end{itemize}
Together with the open set $U\setminus O$, we obtain an open cover of $U$ and pick a $C^k$-partition of unity $\{\chi_i\}_{i\in J}$ subordinate to this cover. We are only interested in the elements of the partition which are completely supported in one of the $U_{\varphi_i}$. Removing all $i$ which are not the domain of a partition element, we may assume that $J=I$. The resulting family of $C^k$-functions is not a partition of unity, but satisfies $\sum_{i\in I} \chi_i (x) = 1$ if $x \in O$.
	
	We will now first construct the necessary functions locally in the submanifold chart $(U_{\varphi_i},\varphi_i), i\in I$. Applying the construction from the proof of 1., we define a smooth function $g_i \colon U_{\varphi_i} \to F, x \mapsto \chi_i (\varphi_i(x))\cdot f (\varphi^{-1}(\pi_{\varphi_i}(\varphi(x)))$. Continue every $g_i$ trivially by $0$ to a $C^k$-function on all of $U$. Then by construction, the $C^k$-function
	$$g \colon U \to F,\quad x\mapsto \sum_{i \in I} g_i (x)$$
	satisfies $g|_{U\cap O}=f|_{O}$.
\end{proof}

\begin{rem}
The hypotheses in Lemma \ref{technical:Lemma} 2. are essential to obtain an extension of functions beyond the boundary points. For example, the function $f \colon (0,\infty) \to \R, f(x)=1/x$ can not be extended smoothly beyond $0$ (but since the submanifold is not closed it is of course not defined on that boundary point). For this function the techniques of Lemma \ref{technical:Lemma} 2. yield the trivial extension $g=f$. See also \cite[Problem 5-18]{MR2954043} for a more in depth discussion of the finite-dimensional case.
\end{rem}

We establish now a variant of Lemma \ref{technical:Lemma}, point 2, by asking for a bounded extension around a given compact set.

\begin{lem}\label{lem:bounded_extension}
	Let $E$ be a Banach space which is $C^k_{b}$-paracompact and $M \subseteq E$ a split submanifold of $E$. Fix $K \subseteq M$ compact and $f \in C^{k} (M,F)$, where $F$ is another Banach space. Then there exists an open $K$-neighborhood $O \subseteq M$ and $g \in C^k_{\text{glob}} (E,F)$ such that $g|_O = f|_O$.
\end{lem}

\begin{proof}
	For every $x\in K$ we apply Lemma \ref{technical:Lemma} 1.~and find an open and bounded $x$-neighborhood $U_x \subseteq E$ and $f_x \in C^k (N,F)$ such that $f_x|_{U_x \cap M} = f|_{U_x \cap M}$. Shrinking the neighborhoods we may assume without loss of generality that $f_x|_{U_x} \in C^k_{\text{glob}} (U_x,F)$. Moreover, we choose and fix for every $x \in K$ an open $x$-neighborhood $V_x$ such that its closure satisfies $\overline{V}_x \subseteq U_x$.
	By compactness of $K \in E$ we may pick finitely many $x_i, i \in \{1,\ldots , n\}$, such that $K \subseteq V \coloneq \bigcup_{1\leq i \leq n} V_{x_i}$. Then $$\{U_{x_i}\}_{1\leq i \leq n} \cup \left\{E\setminus \overline{V}\right\}$$ is a finite (whence locally finite) open cover of $E$. Hence by $C^k_{b}$-paracompactness of $E$, we may choose a $C^k_{b}$-partition of unity $\{\chi_i\}_{1\leq i \leq n} \cup \{\chi_{\text{out}}\}$ subordinate to this open cover.
	Define now a mapping
	$$g \colon E \to F ,\quad x \mapsto \sum_{1\leq i \leq n} \chi_i (x) f_{x_i} (x).$$
    Note that $\chi_i \equiv 0$ outside of $U_{x_i}$ and $f_{x_i}|_{U_{x_i}} \in C^k_{\text{glob}}(U_{x_i},F)$ for every $1\leq i \leq n$. Now $U_{x_i}$ is bounded and thus Remark \ref{rem:boundedness_bump} shows that every $\chi_i$ is also in $C^k_{\text{glob}}$ (not only in $C^k_b$!). Then the Leibniz identity shows that $g \in C^k_{\text{glob}}(E,F)$. Set now $O \coloneq M \cap V$. As $\chi_{\text{out}}|_{\overline{V}} \equiv 0$ we have $\sum_{1\leq i \leq n} \chi_i (x) = 1$ for all $x \in V$. Thus if $x \in O$ the definitions yield
	$$g(x) =  \sum_{1\leq i \leq n} \chi_i (x) f_{x_i} (x) =\sum_{1\leq i \leq n} \chi_i (x) f(x) = f(x).$$
	This concludes the proof.
\end{proof}

\addcontentsline{toc}{subsection}{Spaces and manifolds of Sobolev morphisms on manifolds}
\subsection*{Spaces and manifolds of Sobolev morphisms on manifolds}

Let us briefly recall the construction of manifolds of Sobolev type morphisms. The crucial point here will be that we need spaces and manifolds of Sobolev functions defined on manifolds with values in manifolds (while the literature often only treats the vector valued case, cf.\ e.g.\ \cite{Heb96}).
Throughout the article and in particular in this section we use the following conventions 
\begin{setup}[Conventions]
$K$ will always denote a compact (whence finite dimensional), manifold, possibly with non-empty smooth boundary $\partial K$. We shall denote by $d$ the dimension of the manifold $K$.\footnote{Compactness of $K$ is crucial for us as it is a necessary ingredient for the manifold of mappings structure on Sobolev morphisms. Further, we wish to employ the Sobolev embedding theorem, which fails on general non-compact Riemannian manifolds (see \cite{Heb96} for a discussion).}
Moreover, we assume that $K$ is endowed with a Riemannian metric $g$ and oritentable. Hence we can always choose a volume form $\mu$ associated to the Riemannian metric. 
We shall use the standard multiindex notation for distributional derivatives (see e.g. \cite[Section 1]{Tem01}). 
\end{setup}

\begin{defn}\label{defn:Sobolev_fun}
Let $U$ be any open set in $\R^d$, $d \in \N$ and $\ell \in \N_0$. Let $f \colon U \rightarrow \R^m$ be a function. 
\begin{enumerate}
\item If all distributional derivatives $D^\alpha f$ of $f$ exist for $|\alpha| \leq \ell$ and are square integrable with respect to Lebesgue measure, we say that $f$ is an $H^s$-function. 
\item Denote by $H^\ell(U,\R^m)$ the Hilbert space of all $H^\ell$-functions endowed with
\begin{align} \label{Sob_inner_prod}
\langle f, g\rangle_{H^\ell} \coloneq 
\sum_{ \alpha\in \mathbb{N}_0^d: |\alpha| \leq  \ell} 
\int_U \langle D^\alpha f, D^\alpha g \rangle \mathrm{d}\lambda
\end{align}
where $\lambda$ is the Lebesgue measure.
\item If $f$ is square integrable, we say that $f$ is \emph{locally $H^\ell$}, i.e. $f \in H^\ell_{\mathrm{loc}}(U,\R^m)$ if and only if for every $x \in U$, there exists an open subset $V \subset U$ containing $x$, such that the restriction of $f$ to $V$ belongs to $H^\ell(U,\R^m)$.
\end{enumerate}
\end{defn}

\begin{rem}\label{rem:ident_Sobolev_product}
Note that $f \colon U \rightarrow \R^m$ is $H^\ell$ if and only if every component $f_i \colon U \rightarrow \R, 1 \leq i \leq d$ of $f$ is an $H^\ell$-function. Thus the mapping 
$$H^\ell(U,\R^m) \rightarrow H^\ell(U,\R)^{m}, \quad f \mapsto (f_i)_{1\leq i \leq d}$$
is a linear isomorphism with inverse $(f_i)_{1 \leq i \leq d} \mapsto \sum_{1 \leq i \leq d}(0, \ldots ,0,f_i,0,\ldots,0)$ where $f_i$ is inserted in the $i$th component. Continuity of these mappings is easily established by a direct calculation, whence we obtain a canonically identification of Hilbert spaces $H^\ell(U,\R^m) \cong H^\ell(U,\R)^m$ for every $s,m \in \N_0$.
\end{rem}

Employing Remark \ref{rem:ident_Sobolev_product} we can apply \cite[Theorem 3.22]{MR2424078} to the spaces $H^\ell(U,\R^m)$ as follows: If  $U$ is an open bounded subset of $\R^d$ whose boundary satisfies the segment condition, for instance $U$ has Lipschitz or piecewise smooth boundary with pieces intersecting transversally, the space $H^\ell (U,\R^n)$ coincides with the completion of the (compactly supported) smooth functions on $U$ with respect to the Sobolev $H^s$-norm \eqref{Sob_inner_prod}). Hence we can extend the definition of spaces of Sobolev functions as follows:

\begin{defn}\label{defn:Sobolev_funII}
Let $U\subseteq \R^d$ be a set with dense interior. Assume that the boundary of $U$ satisfies the segment condition and $\ell \in \mathbb{N}_0$. Define the $H^\ell$-functions as the completion of the (compactly supported) smooth functions with respect to the $H^\ell$-inner product,
\begin{align}\label{Hell_product}
\langle f, g\rangle_{H^\ell} \coloneq 
\sum_{ \alpha\in \mathbb{N}_0^d: |\alpha| \leq  \ell} 
\int_U \langle D^\alpha f, D^\alpha g \rangle \mathrm{d}\lambda
\end{align}
where $D^\alpha$ are partial derivatives indexed by the multiindex $\alpha$, see e.g. \cite[Appendix B]{MMS19}.  
We write $H^\ell(U,\R^m)$ for the Hilbert space of $H^s$-functions endowed with the inner product \eqref{Hell_product}.
\end{defn}

Due to our assumptions on the boundary of $U$, the $H^s$-functions defined for open subsets (with nice boundary) via Definition \ref{defn:Sobolev_funII} coincide with the one defined in Definition \ref{defn:Sobolev_fun}. Further, the argument in Remark \ref{rem:ident_Sobolev_product} extends without any changes also to $H^s(U,\R^m)$ for $U$ non-open as in Definition \ref{defn:Sobolev_funII}. In practice, we will always deal with non-open sets with piecewise smooth boundary where the pieces intersect transversally. Note that these sets satisfy the segment condition and also the cone condition \cite[4.6]{MR2424078}, which allows us to apply the Sobolev embedding theorem, cf.\ \cite[Theorem 4.12]{MR2424078}, to the spaces $H^s(U,\R^m)$ if $s$ is large enough.
Moreover, recall that for smooth boundary sets the Calderon extension theorem \cite{MR0324726} asserts that all Sobolev $H^s$-functions on a non-open domain $U$ with smooth boundary can be extended to $H^s$-Sobolev functions on an open subset of $\R^d$ containing $U$.  

\subsubsection*{Sobolev morphisms on manifolds}

We first review definitions of Sobolev type mappings between manifolds. 

\begin{defn}\label{defn:Sobolevmaps}
Pick now an integer $s > d/2$ and let $M$ be a $d$-dimensional manifold (again possibly with smooth boundary). A continuous map $f \colon K \to M$ is an $H^s$-map if for every $x\in K$, there are charts $(U,\varphi)$ of $K$ and $(V,\psi)$ of $M$ such that $x\in U$, $f(U)\subseteq V$ and $\varphi\circ f \circ \psi$ is an $H^s$-map. 
\end{defn}

\begin{setup} The $H^s$-property is only required to hold in some charts. As the Sobolev property involves a boundedness condition, it is in general not stable under change of charts (see e.g. \cite[3.1]{IKT13} for an example). This problem can be remedied by requiring the charts in an atlas to satisfy additional conditions. As worked out in \cite{IKT13}, the Sobolev property is stable under change of charts, if one requires the atlas to consist only of pre-compact subsets whose boundary satisfies certain Lipschitz conditions.
\end{setup}

\begin{setup}
Denote by $H^s(K,M)$ the set of all Sobolev $H^s$-morphisms from $K$ to $M$. If $s > d/2+1$ the set $\Diff^s (K) \subseteq H^s (K,K)$ of $H^s$-diffeomorphisms becomes a group with respect to composition of maps.

One can now endow $H^s(K,M)$ with the $H^s$-topology (see \cite[Section 3]{IKT13} and compare \cite[Appendix B]{MMS19} for the case of $K$ having non-empty boundary $\partial K$). In the $H^s$ topology for $s>d/2+1$, the subset $\Diff^s (K)$ is open in $H^s(K,K)$.
\end{setup}

A version of the Sobolev embedding theorem holds for the spaces of $H^s$-mappings. Note that this implies in particular, that the Sobolev $H^s$-topology is stronger than the compact open topology for $s$ above the critical Sobolev index $d/2$. We can leverage that to obtain regularity of the evaluation mapping. 

\begin{lem}\label{lem:eval}
Let $s>d/2+k$ for some $k \in \N$, then the map
\begin{align*}
\mathrm{eval}\colon H^{s}(K,\R^n) \times K \to \R^n, \quad (X,k)\mapsto X(k)
\end{align*}
is $C^k$. Moreover, its tangent map is given by $T\mathrm{eval}((X,Y),v) = TX(v) +(0,Y\circ \pi_K(v))$.
\end{lem}

\begin{proof}
Since $s>d/2+k$ we can invoke the Sobolev embedding theorem \cite[Lemma B.8]{MMS19} and obtain a smooth, linear inclusion $\iota \colon H^{s} (K,\R^n) \to C^k (K,\R^n)$.
Now we have $\mathrm{eval} = \mathrm{eval}_{C^k} \circ (\iota \times \id_K)$, where $\mathrm{eval}_{C^k} \colon C^k(K,\R^n) \times K \to \R^n$ is the evaluation map of the $C^k$-functions.
However, for $\mathrm{eval}_{C^k}$ it is known (cf.\ \cite[Proposition 3.20 and Lemma 3.15]{AaS15} or \cite[Lemma 1.19]{AaGaS20}) that $\mathrm{eval}_{C^k}$ is a $C^k$-map. The formula for $T\mathrm{eval}$ follows from the the usual calculation in the natural identification of the tangent bundle \cite[Appendix A]{AaGaS20}. Observe that the first term in $T\mathrm{eval}$ is the reason for $\mathrm{eval}$ being only a $C^k$-mapping.
\end{proof}

Recall from \cite[Example 9.32]{MR2954043}, that every compact manifold with smooth boundary $K$ can be embedded in a smooth compact manifold $DK$ without boundary, the \emph{double of $K$}. Smooth sections of a vector bundle on the double $DK$ of a manifold with boundary, restrict to smooth sections of the restriction of the bundle to $K$.
It is easy to see that the restriction (to suitable subsets of $K$) of smooth functions is continuous in the $H^s$-topology. Hence, we immediately obtain a continuous restriction operator on $H^s$-maps (cf.~also \cite[X: Theorem 6]{MR0198494}) 

\begin{lem}\label{lem:res:cont}
Let $K$ be a compact manifold with smooth boundary and $DK$ the double of $K$. For every $s > d/2$ and $n\in \N$, the restriction 
$$\mathrm{res}_K\colon H^s (DK,\R^n) \to H^s(K,\R^n),\quad F \mapsto F|_K$$ makes sense and is a continuous linear map.
\end{lem}

Later on, in particular for the proof of Theorem \ref{thm:equivalence_flow_pde}, we need a suitable regularization operator. The following Lemma provides it.

\begin{lem}\label{lem-regularization}
Let $K$ be a compact manifold (possibly with boundary) of dimension $d$ and $s>d/2$, $n\in \N$. Then there exists for every $\delta>0$ a regularization operator $R_\delta\colon H^{s}(K,\R^n)\to H^{s}(K,\R^n)$, such that:
\begin{itemize}
\item $R_\delta$ is continuous linear and its image is contained in $H^{s+2}(K,\R^n)$ and there exists $C\coloneq C(s)>0$ such that for every $\delta$,  
\[ \Vert R_\delta v\Vert_{H^s} \le C\| v\|_{H^{s}},\quad v\in H^{s}(K,\R^n).
\]
\item for every $v \in H^{s}(K,\R^n)$,  $R_\delta v \to v$ convergence in $H^{s}$ as $\delta\to 0$, 
\end{itemize}
\end{lem}

\begin{proof}
 \textbf{Step 1: Reduction to scalar functions.}
 Recall from Remark \ref{rem:ident_Sobolev_product} that there is a canonical continuous linear identification $H^s(K,\R^n) \cong H^s (K,\R)^n$. Hence it suffices to construct a regularization operator for $H^s(K,\R)$ and apply the $n$-fold product of these operators in the identification. We may thus assume without loss of generality that $n=1$.\smallskip

 \textbf{Step 2: Reduction to manifolds without boundary.}
 If $\partial K \neq \emptyset$ we embed $K$ in its double $DK$. By the Calderon Extension Theorem \cite{MR0324726} there exists a continuous linear extension mapping
 $
 \text{ext}\colon H^{s} (K,\R^n) \to H^{s} (DK,\R^n), 
 $  such that for the restriction $\text{res}_K$ from Lemma \ref{lem:res:cont} we have
 \begin{equation}\label{eqn-extension Calderon-2}
 \text{res}_K \circ \text{ext}(f)=\text{ext}(f)=\text{ext} (f)|_K =f \mbox{ for every }f \in H^{s}(K,\R),
 \end{equation}
 Assume for a moment that there were a family $\tilde{R}_\delta\colon H^s (DK,\R) \rightarrow H^s (DK,\R), \delta >0$ of regularization operators with the claimed properties for $H^{s} (DK,\R)$. The inclusion $H^{s+2}(DK,\R^n) \subseteq H^s (DK,\R^n)$ is continuous linear and we deduce from Lemma \ref{lem:res:cont} that for $\delta > 0$ the map $R_\delta \coloneq \text{res}_K \circ \tilde{R}_\delta \circ \text{ext}$ is continuous linear and bounded by some constant only depending on $s$ (but not on $\delta$). In addition, convergence of $\tilde{R}_\delta \text{ext}(v) \rightarrow \text{ext}(v)$ for $\delta \rightarrow 0$ in $H^s$ yields convergence of $R_\delta v \rightarrow v$. Thus without loss of generality, we may assume that $K$ is a compact manifold without boundary.\smallskip

 \textbf{Step 3: Regularization operator on $H^s(K,\R)$ for $K$ without boundary.}
 Denote by $\Delta$ the Laplace-Beltrami operator with respect to the Riemannian structure on $K$, cf.\ \cite[3.2.3]{Lab15}. 
Now, as is well known, $H^{s}(K,\R)$ for a compact Riemannian manifold $K$ admits a spectral decomposition with respect to $-\Delta$ such that each eigenspace is finite dimensional (see for example~\cite[Thm.~4.3.1]{Lab15}).
 Thus, $H^{s}(K,\R)$ can be equipped with an orthonormal spectral basis for $-\Delta$. 

 Thus for $\delta >0$, $R_\delta\coloneq (\id-\delta\Delta)^{-1}\colon H^{s}(K,\R)\to H^{s}(K,\R)$ is a continuous linear operator which is diagonal with respect to the spectral basis. Since the eigenvalues $\lambda_k$ of $-\Delta$ are positive, the eigenvalues of $R_\delta$ satisfy $1/(1+\delta\lambda_k) \leq 1$.
 Moreover, by construction the map $R_\delta\colon H^{s}(K,\R) \to H^{s}(K,\R)$ is continuous linear and takes its image in $H^{s+2}(K,\R)$. Checking at the level of the spectral decomposition, it is easy to see that $\lim_{\delta \to 0} R_\delta u = u \in H^{s}(K,\R)$, for every $u \in H^{s}(K,\R)$. Finally, we have due to the eigenvalue estimate $\| R_\delta \|_{op} \leq 1$ independent of $\delta$ as required.
\end{proof}

\subsubsection*{Manifolds of Sobolev mappings}
In the last section we studied topological spaces of (manifold valued) Sobolev functions. We shall now recall, cf.\ \cite[Appendix B]{MMS19}, that the topological structure turns $H^s(K,M)$ into an infinite-dimensional manifold modelled on Hilbert spaces of Sobolev sections. Since we shall use them in some proofs, we shall recall the construction of canonical charts for these manifolds in broad strokes now. For more information on manifolds of Sobolev type morphisms we refer to \cite{IKT13} (for an introduction to manifolds of (smooth) mappings, see \cite{schmeding_2022}).

\begin{setup}[Local addition and canonical charts]
Let $N$ be a manifold and $O \subseteq TN$ be an open subset such that for every $x \in M$ the zero element $0_x\in T_x N$ is contained in $O$, Then a smooth mapping $\Theta \colon O \to N$ is called \emph{local addition} if 
\begin{itemize}
\item $\Theta (0_x)=x$ for all $x \in N$,
\item $E_\Theta =(\Theta , \pi_N) \colon O \to N \times N$ induces a diffeomorphism from $O$ to an open neighborhood of the diagonal in $N \times N$.
\end{itemize}
An example of a local addition is the Riemannian exponential map $\exp_g$ of the Riemannian manifold $(K,g)$.
Local additions can be used to construct charts for the manifolds of mappings. These canonical manifold charts are constructed as follows: Let $(K,g)$ be a Riemannian manifold and $(N,\Theta)$ be a manifold with a local addition $\Theta$. Shrinking $O$, we may assume that $E_\Theta (O)$ is symmetric with respect to interchanging the omponents of $N\times N$. Define for $f \in H^s(K,N)$ the set 
$$\mathcal{U}_f \coloneq \{g \in H^s (K,N) \mid (f,g)(x) \in E_\Theta (O) , \forall x \in N\},$$
together with a map 
\begin{equation}\label{defn:can-charts}
\varphi_f \colon \mathcal{U}_f \to H^s_f (K,TN) = \{F \in H^s (K,TN)\mid \pi_N \circ F =f\}, g \mapsto E_\Theta^{-1}\circ (f,g).
\end{equation}
We also refer to Appendix \ref{App:Sobsect} for more information on the spaces $H^s_f (K,TN)$. Clearly the inverse of the map $\varphi_f$ is
$$\varphi_f^{-1}\colon H^s_f(K,TN) \supseteq O_f \to H^s (K,N) , \quad F \mapsto E_\Theta \circ F,$$
where $O_f = \{F \in H^s_f (K,TN)\mid F(K) \subseteq O\}.$ One can then show that the family $\{(\mathcal{U}_f,\varphi_f)\}_{f \in H^s (K,N)}$ forms a manifold atlas for $H^s (K,N)$. We call the charts $\varphi_f$ canonical charts of the manifold of mappings $H^s(K,N)$ and refer to \cite[B.7]{MMS19} for more information and the omitted proofs.
\end{setup}

\begin{setup}
For $s > d/2$ There is a canonical identification of the tangent bundle taking equivalence classes $[\gamma]$  of smooth curves with values in $H^s$ to 
\begin{align}\label{tangent_ident}
TH^s (K,N) \rightarrow H^s (K,TN), [\gamma] \mapsto \left(k \mapsto \left.\frac{d}{dt}\right|_{t=0}\gamma(t)(k)\right)
\end{align}

In particular, the map \eqref{tangent_ident} is an isomorphism of Hilbert vector bundles.
For a proof we refer to \cite[Theorem 2.4]{BaHaM19}, where the case of $K$ being a manifold without boundary is treated. Note  however that the proof carries over verbatim to the case of $K$ having smooth boundary, as the canonical charts in this generalised setting are the same, cf.\ \cite[Appendix B]{MMS19}.
\end{setup}

\begin{setup}[The group $\Diff^s (K)$]\label{setup:diffgp}
The set $\Diff^s (K)$ of $H^s$-diffeomorphisms is open in $H^s (K,K)$ for $s>d/2+1$. Hence it is an open submanifold and a group under composition of mappings. However, it is only a topological group, but not a Lie group. One can show (cf.\ \cite[Appendix B]{MMS19}), that the right multiplication
$$R_\Phi \colon \Diff^s (K) \to \Diff^s (K), \quad \xi \mapsto \xi \circ \Phi$$
is smooth, while left composition with $\zeta \in H^{s+\ell} (K,N),\ell \in \N_0$
$$L_\zeta \colon \Diff^s(K)\to H^s (K,N),\quad \Phi \mapsto \zeta \circ \Phi$$
is only a $C^\ell$-map. Recall that this is due to the fact, that the identification \eqref{tangent_ident} yielding  $TH^s (K,N) \cong H^s (K,TN)$ allows us to identify the derivatives of these mappings as
$$TR_\Phi (X)=X\circ \Phi, \qquad TL_\zeta (\eta) = T\zeta \circ \eta = L_{T\zeta}(\eta).$$
Similarly, the inversion formula yields for a $C^1$-curve $\gamma \colon \mathbb{R} \rightarrow \Diff^{s+\ell}(K)$ the following
\begin{align}\label{inversion_nonsmooth}
\frac{d}{dt} \gamma(t)^{-1} = -T(\gamma(t)^{-1})\circ \left(\frac{d}{dt} \gamma(t) \right)\circ \gamma(t)^{-1} \in T_{\gamma(t)^{-1}}\Diff^s(K)
\end{align}
which only makes sense if $\ell >0$.
Summing up, $\Diff^s (K)$ is only a half-Lie group (as in only "half" of the multiplication is smooth, while the other half is in general only continuous). For later we remark that \eqref{tangent_ident} implies that
\begin{equation}\label{eq:TDiff_identification}
T_\phi \Diff^s (K) \cong\{X \circ \phi \mid X \in \mathfrak{X}^s(K)\} \subseteq H^s(K,TK),
\end{equation}
Thus we may think of tangent vectors to diffeomorphisms as $H^s$-vector fields composed with diffeomorphisms. For more information on the spaces of $H^s$-vector fields, the reader may consult Appendix \ref{App:Sobsect}. 
\end{setup}

 Let $M$ be either $K$ if $K$ has no boundary or $M=\newtilde{K}$, the double of $K$, if $K$ has boundary. Hence, we obtain a canonical embedding $J\colon K\to M$ of $K$ as a submanifold of $M$. Applying the Whitney embedding theorem \cite[Theorem 1.3.4]{Hir76} we have further an smooth embedding $\iota \colon M \to \R^m$ for some positive integer $m$. Let $i \colon K\to \R^m$ be the embedding $i=\iota \circ J$. We will now prove that the pushforward $i_*(\varphi)\coloneq i\circ \varphi$ induces an embedding of $\Diff^s_\mu(K)$ into $H^s(K,\R^{m})$.

\begin{prop}\label{prop:embedding_Diff}
For $s > d/2+1$ the unit component $\Diff_0^s (K)$ and also the group $\Diff^s(K)$ can be embedded through $I_*$ as split submanifolds of a Hilbert space $H^s (K,\R^m)$.
\end{prop}

\begin{proof}
 As $\Diff_0^s (K) \subseteq \Diff^s (K) \subseteq H^s(K,K)$ is a chain of inclusions of open subsets, these subsets from open (whence split) submanifolds of $H^s(K,K)$. If $K$ has boundary, recall from \cite[B.10]{MMS19} that $\Diff^s (K)$ is a closed, whence split, submanifold of $H^s (K,DK)$, where $DK$ is the double of $K$. By transitivity of split submanifolds \cite[Lemma 1.4]{Glo16} it suffices thus to prove that $H^s(K,M)$ is a split submanifold of $H^s (K,\R^m)$ for some $m>0$, where $M$ is either $K$ (no boundary case) or the double of $K$ (if $K$ has smooth boundary).

 Since $M$ is a compact manifold without boundary it embeds into some $\R^m$. The idea is to lift this embedding to the infinite-dimensional manifold of mappings. For manifolds of smooth maps, this was done in \cite[Proposition 10.8]{MR583436}. We follow loc.cit.~closely and need only deal with some additional complications due to the Sobolev type mappings.
	
 Corestricting the embedding $\iota \colon M \to \R^m$ we obtain a diffeomorphism $I\colon M \to \iota(M)$. The pushforward with $I$ induces an isomorphism $H^s(K,M) \cong H^s (K,\iota (M))$ (this can be seen by endowing $\iota(M)$ with the metric induced by $I$, whence the pushforward relates the canonical charts of $M$ and $\iota (M)$). Hence we may assume without loss of generality that $M \subseteq \R^m$.
 Now the proof can be completed as in \cite[Proposition 10.8]{MR583436}. We repeat it for the readers convenience:
 Pick a tubular neighborhood $M \subseteq U \subseteq \R^m$ of $M \in \R^m$ and work with the vector bundle $(U,p,M)$. By \cite[10.6]{MR583436} we can choose a local addition $\tau \colon TU \to U$ such that $M$ (identified with the $0$-section of $E$ and all fibres $p^{-1}(m), m\in M$ are additively closed in $U$.\footnote{Recall that the submanifold $X$ of a manifold $Y$ is called additively closed with respect to the local addition $\tau\colon TY\to Y$ if $\tau (TX)\subseteq X$ holds.} In particular this entails that the canonical chart $(\mathscr{U}_g,\varphi_g)$ of $H^s (K,\R^m)$ around $g\in C^\infty(K,M)$ (see \eqref{defn:can-charts}) satisfies 
 $$\mathscr{U}_g \cap H^s(K,M) = \varphi^{-1}_g (H^s_g(K,TM)) \subseteq \varphi^{-1}_g (H^s_g(K,T\R^m)).$$ 
 Via the vertical bundle $V(U)$ of $(U,p,M)$ we split $TU|_M = TM \oplus V(U)$. Since $g$ takes its values in $M$ we have $H^s_g(K,T\R^m)\cong H^d(K,g^*T\R^m) = H^s(K,g^*TU)=H^s(K,g^*TU|_M)$, whence 
	\begin{align*}
	 H^s_g (K,T\R^m) \cong& H^s (K,g^*TU|_M) \stackrel{\text{\cite[Theorem 14.5]{MR0248880}}}{=}& H^s (K,g^*TM) \oplus H^s (K,g^*V(U))\\ \cong& H^s_g (K,TM) \oplus H^s (K,g^*V(U)).
	\end{align*}
 We see that $H^s_g (K,TM)$ is a direct summand of $H^s_g (K,T\R^m)$. Smooth functions are dense in $H^s(K,M)$ such that every $H^s$-function is contained in $\mathscr{U}_g$ for some $g \in C^\infty (K,M)$. The canonical charts centered at such $g$ yield submanifold charts realising $H^s(K,M)$ as a split submanifold of $H^s (K,T\R^m)$.
\end{proof}

\begin{cor}\label{cor:embedding_TDiff}
For $s > d/2+1$, $T\Diff_0^s (K)$ and $T\Diff^s (K)$ can be embedded through (the restriction of) $Ti_*  = Ti\circ \cdot$ as split submanifolds of the Hilbert space $H^s (K,\R^{2m})$.
\end{cor}

\begin{proof}
This result follows immediately from the Proposition \ref{prop:embedding_Diff}: The tangent map of a split submanifold embedding is a split submanifold embedding (this is a consequence of \cite[Lemma 1.13]{Glo16}).
\end{proof}

\begin{cor}
 Since the groups $\Diff^s_\mu (K)$ and $(\Diff^s\mu)_0(K)$ are (split) submanifolds of $\Diff^s(K)$, $(T)i_*$ embedds $(T)\Diff^s_\mu (K)$ and $(T)(\Diff^s_\mu)_0(K)$ as a split submanifold of a Hilbert space.
\end{cor}

The notation $(\Diff^s_\mu)_0(K)$ is quite cumbersome. However, in the Ebin-Marsden approach we work often only with the unit component of the groups involved. Thus we allow ourselves the abuse of notation that later on we shall often suppress the unit component in the notation. In particular, it will be convenient to allow the abuse of notation and later write $(T)\Diff^s_\mu(K)$ and $(T)i_*$ to denote $(T)(\Diff_\mu^s)_0(K)$ and the restriction of the embedding $(T)i_*$.

\addcontentsline{toc}{subsection}{Weaker concepts of differentiability}
\subsection*{Weaker concepts of differentiability}
Later we will investigate differentiable dependence of the solution of a stochastic differential equation on its initial values. For this it is useful to have the following weakened notions of differentiability at hand.
\begin{defn}\label{Gateaux:diff}
Let $E,F$ be topological vector spaces. Consider a map $f\colon U \to F$ defined on an open subset $U \subseteq E$. Assume that $a \in U$.  We say that  $f$ is  Gateux differentiable at $a$  if and only if  for every $v \in E$ the limit
$$df(a;v)\coloneq \lim_{h \to 0} h^{-1}(f(a+hv)-f(a))$$ exists.
the map $df(a;\cdot) \colon E\to F$ is then called the Gateaux derivative of $f$ at $ a$. 

Following Elworthy \cite[p. 139]{Elw82}, we say that $f$ is \emph{strongly Gateaux differentiable} in $a \in U$ if for every $C^1$ curve $\gamma\colon (-\varepsilon,\varepsilon)\to U \subseteq E$ such that $\gamma(0)=a$, the function 
$
(-\varepsilon,\varepsilon)\ni s\mapsto f \circ \gamma(s)\in F
$
is differentiable at $s=0$.
\end{defn}

Clearly, if $f$ is strongly Gateaux differentiable then $f$ is Gateaux differentiable and
\begin{align}\label{eq:strongGateaux}
\left.\frac{d}{ds}\right|_{s=0} f(\gamma(s)) := \lim_{h\to 0}h^{-1}(f(\gamma(h))-f(\gamma(0))) = df(\gamma(0), \dot{\gamma}(0))
\end{align}
holds for the Gateaux derivative of $f$ at $a$. The converse is not true, but one can prove the following variant of \cite[Lemma 8D]{Elw82}:

\begin{lemma}\label{la:Gateaux}
Let $f \colon E \supseteq U \to F$ be a Gateaux differentiable in $x$, $U \subseteq E$ open in a normed space $E$ and $F$ a normed space. If $f$ is Lipschitz continuous, then $f$ is strongly Gateaux differentiable in $x$.
\end{lemma}

\begin{proof}
Assume that $\gamma \colon (-a,a)\to U$ is a $C^1$-curve with $\gamma(0)=x$ and $\gamma'(0)=h$. Let $L$ be a Lipschitz constant for $f$. We compute now for $s\neq 0$ an estimate 
\begin{align*}
 &\lVert s^{-1}(f(\gamma(s))-f(\gamma(0)))-df(x;h)\rVert_F\\
 \leq& \lVert s^{-1}(f(\gamma(s))-f(x+sh))\rVert_F + \lVert s^{-1}(f(x+sh)-f(x))-df(x;h)\rVert_F \\
\leq& L\lVert s^{-1} (\gamma(s)-x)-h\rVert_E + \lVert s^{-1}(f(x+sh)-f(x))-df(x;h)\rVert_F
\end{align*}
By construction, both terms in the last line tend to $0$ as $s\to 0$. Thus the limit \eqref{eq:strongGateaux} exists and tends to the Gateaux derivative of $f$ at $x$. In conclusion, $f$ is strongly Gateaux differentiable at $x$. 
\end{proof}

Note that pointwise Gateaux derivatives may behave wildly. Since we have not even asked for continuity of the derivative, one may not even assume that the derivative is linear in the vector component. Moreover, if the vector spaces are not locally convex, even continuously Gateaux differentiable mappings may exhibit pathological behaviour (see \cite[Section 1.1]{schmeding_2022} for a discussion.

\begin{defn}\label{defn:Bast}
A mapping $f\colon U \to F$ is said to be Bastiani continuous differentiable or $C_{\text{Bas}}^1$, if it is Gateaux differentiable for every $x \in U$ and the Gateaux derivatives glue together to a continuous map $df \colon U \times E \to F$. Iteratively, we can then define $C_{\text{Bas}}^k$-mappings for every $k \in \mathbb{N}$. If $f$ is $C_{\text{Bas}}^k$ for every $k \in \mathbb{N}$, we say $f$ is smooth. 
Note that for Banach spaces, $C_{\text{Bas}}^k$-mappings are $k-1$times differentiable in the usual sense. Thus in particular Bastiani smooth and smooth maps coincide on these spaces. 
\end{defn}

\begin{remark}
 Results such as Lemma \ref{technical:Lemma} remain true in the more general setting of Bastiani differentiable mappings between manifolds modelled on locally convex spaces. See \cite{schmeding_2022} for more information on this generalised setting.
\end{remark}

\section{SDEs on Hilbert manifolds: local well-posedness and maximal existence}\label{sec:maximal_existence}

In this section, we show the existence of a maximal solution for SDEs on Hilbert manifolds, as well as other minor extensions of results in \cite{MMS19}. All the results are classical up to minor modifications, see e.g. \cite{Elw82}, but we prefer to state them for the readers convenience to make the article self-completeness.

\subsection*{Stopping times and the It\^o integral}
\addcontentsline{toc}{subsection}{Stopping times and the It\^o integral}
In what follows, we will use integral and SDEs with random initial time. Although in the main results on stochastic Euler equations we put $0$ as the initial time, random initial times are needed at least in the proof of existence of a maximal solution. As in \cite{MMS19}, we recall some preliminary facts, adding stopping times as initial times.

We are given a probability space $(\Omega,\mathscr{A},\mathbb{P})$ and a filtration $\mathbb{F}=(\mathscr{F}_t)_{t\in [0,\infty)}$ satisfying the usual assumptions, that is,
\begin{enumerate}
\item[(i)] $\mathbb{F}$  is  right-continuous,
\item[(ii)]  and all null sets of $\mathscr{A}$ are elements of $\mathscr{F}_0$.
\end{enumerate}
 We call $\mathscr{F}_\infty$ the $\sigma$-algebra generated by all $\cF_t$. \\
  A stopping time is a function   $\tau:\Omega\to [0,\infty]$ such that, for every $t\ge0$, the event $\{\tau\le t\}$ belongs to $\cF_t$, see 
  \cite[Definition I.2.1]{Kar-Shr-96}, \cite[Definition 4.1]{Metivier_1982} and \cite[section III.5]{Elw82}.
  The family 
\begin{align*}
\cF_\tau= \{A\in \cF_\infty \mid A\cap \{\tau\le t\}\in \cF_t,\,\forall t\}
\end{align*}
is called the $\sigma$-algebra associated with, or generated by,  $\tau$. \\
A  stopping time $\tau$ is called a finite stopping time if and only if  a.s. $\tau<\infty $.

Let us make the following remark based on  Appendix A in \cite{Brz+H+Raza_2021}.
\begin{remark}\label{rem-stopping time-progressive measurability}
If $\tau$ is a stopping time, the indicator stochastic process  $\mathds{1}_{[0,\tau)}$ defined by 
\begin{equation}\label{eqn-indicator process}
\mathds{1}_{[0,\tau)}(s) = \begin{cases} 1, & \mbox{ if }s\in [0,\tau),\\
0, &\mbox{ otherwise},
\end{cases}
\end{equation}
 is well measurable, see \cite[Definition 3.1 and Proposition 4.2]{Metivier_1982}. Hence, the stochastic process $\mathds{1}_{[0,\tau)}$,  is progressively measurable. 
 Indeed,  according to   \cite[Theorem 1.6]{Metivier_1982}, the $\sigma$-field of well measurable sets is a subset of  the $\sigma$-field of all progressively measurable sets. \\
 Hence, if   $\rho$ and $ \tau$  are stopping time     satisfying Hypothesis 
\ref{hh-stopping times}, then the process $ \mathds{1}_{[\rho,\tau)}$ is also progressively measurable. Indeed, we have 
 \[
 \mathds{1}_{[\rho,\tau)}(s)=\mathds{1}_{[0,\tau)}(s)-\mathds{1}_{[0,\rho)}(s), \;\; s \geq 0.
 \]
\end{remark}

\begin{remark}\label{rem-filtration}
Since the filtration $\mathbb{F}$ is fixed, by writing, for example, adapted or progressively measurable, we will always have in mind that these refer to the filtration $\mathbb{F}$.
\end{remark}

\begin{defn}\label{def-accessible stopping time}
If $\rho$  is a finite  stopping time and  $A \in \cF_\rho$, then a stopping time $\tau$ such that  $\tau > \rho$ a.s.  on $A$,   is called $\rho$-accessible on $A$ if and only if 
there exists a sequence  $(\tau_n)$ of stopping times such that the following two conditions are satisfied a.s. on the set $A$, 
\begin{enumerate} 
\item[(i)] $\rho\le \tau_n<\tau$   for every $n$, 
\item[(ii)]  $\tau_n\nearrow \tau$.
\end{enumerate}
In particular, we can take $\tau_n=\tau$ on $A^c$ without loss of generality. 
Such a sequence is called an \emph{announcing sequence} for $\tau$ relative to $\rho$ on $A$.\\
If $A=\Omega$, we simply say that $\tau$ is $\rho$-accessible and the sequence $(\tau_n)$ is called announcing sequence for $\tau$ relative to $\rho$. \\
If $A=\Omega$ and $\rho=0$, we simply say that $\tau$ is accessible
and that  $(\tau_n)$ is an  announcing sequence for $\tau$. 
\end{defn}

Some additional properties of accessible stopping times are listed in Remark 3.4 of paper \cite{Brz+H+Raza_2021}.

Given an admissible  process $\xi$ with values in a Hilbert space $H$ and 
a finite stopping time $\rho$,  the first exit time from an open set $U \in H$ after $\rho$ is defined by 
\begin{align}\label{eqn-tau_U}
\tau_U = \tau_{U,\rho} := \inf\{t\ge \rho \mid \xi_t\notin U\}
\end{align}
It is known that $\tau_{U,\rho}$ is an $\rho$-accessible stopping time on $\{\xi_\rho\in U\}$, with announcing sequence $(\tau_n)$ defined by  $\tau_n= \inf\{t\ge \rho \mid \xi_t\notin U_n\}$, where $U_n = \{x\in H \mid \text{dist}(x,U^c)>1/n\}$; see e.g. \cite[Proposition 3.7]{Ba2017}.

\begin{defn}[Stochastic intervals  and stopping times]\label{stoch_iv_notation}
Given an event $B \in \cA$ and two stopping times $\rho$ and $\tau$ such that  $\rho\le \tau$ a.s. on $B$, we define the following two stochastic intervals
\begin{align}\label{eqn-stochastic intervals}
[\rho,\tau)\times B := \{(t,\omega) \in [0,\infty)\times B  \mid t\in [\rho(\omega),\tau(\omega))\},
\\
[\rho,\tau]\times B := \{(t,\omega) \in [0,\infty)\times B \mid t\in [\rho(\omega),\tau(\omega)]\}
\end{align}
When $B=\Omega$, we also use the notation 
\begin{equation}
[[\rho,\tau))=[\rho,\tau)\times \Omega \mbox{ and }  [[\rho,\tau]]=[\rho,\tau]\times \Omega    
\end{equation}
 Finally, we denote by
\begin{equation} \Omega_t(\rho,\tau):=
	\{ \omega \in \Omega : \rho(\omega) \leq t < \tau(\omega)\},
	\end{equation}
and put $\Omega_t(\tau)=\Omega_t(0,\tau)$.
\end{defn}
If  $(G,\cG)$ is a measurable space,  a $G$-valued process $\xi$ on $[\rho,\tau)$ is a map $\xi:[\rho,\tau)\times\Omega\to G$, such that, for every $t\in [0,\infty)$, the map
\begin{equation}\label{eqn-xi-bar}
\bar\xi_t: \Omega \ni \omega \mapsto 
\begin{cases}
\xi_t(\omega), &\mbox{ if } \mathds{1}_{[\rho(\omega),\tau(\omega))}(t)=1, 
\\
y_0,&\mbox{ if } \mathds{1}_{[\rho(\omega),\tau(\omega))}(t)=0.
\end{cases}
\end{equation}
is $\mathscr{A}/\cG$-measurable, for some (equivalently every)  $y_0 \in G$.  A process $\xi$ is called adapted, resp. progressively measurable, if $\bar\xi$ is adapted, resp. progressively measurable, for some (equivalently every)  $y_0 \in G$. A process $\xi$ is called \emph{admissible} if and only if  it is adapted and 
for almost all $\omega \in \Omega$, the function \begin{equation}\label{eqn-local trajectory} 
[\rho(\omega),
			\tau(\omega))\ni t \mapsto \xi(t, \omega) \in X
\end{equation}   
   is continuous. According to \cite[Proposition 1.1.13]{Kar-Shr-96}, every  admissible process is  progressively measurable. Analogous  definitions can be formulated  for  process defined on $[\rho,\tau]$.

The following is a modification of \cite[Definition 4.1]{daPaZ} motivated by the remarks following that Definition.
\begin{defn}\label{def-Wiener process in Banach space}
Let us assume that $E$ is a separable Banach space endowed with the Borel $\sigma$-field $\mathscr{B}(E)$.
An  $E$-valued  $W=(W_t)_{t\geq 0}$ is called  $\mathbb{F}$-Wiener process  if and only if 
 and   is , i.e. a continuous
with independent increments  $E$-valued $\mathbb{F}$-adapted stochastic process such that
\begin{enumerate}
\item[(w1)] $W$ has continuous trajectories and $W_0=0$,
\item[(w2)] for all $t>s\geq 0$, the increment $W_t-W_s$ is independent of $\sigma$-field $\cF_s$,
 and
 \begin{equation}\label{eqn-equality of laws}
\mbox{  the law $\mathscr{L}(W_t-W_s)$ is equal to the law $\mathscr{L}(W_{t-s})$,}
 \end{equation}
\item[(w3)]   and, for every $t \geq 0$,
 $
 \mbox{
the law $\mathscr{L}(W_t)$ is equal to
$\mathscr{L}(-W_t)$.}
$
 \end{enumerate}
\end{defn}
It follows that for every  $t>0$, the law $\mathscr{L}(W_t)$ is a gaussian measure on $R$ with mean $0$, see \cite[Section 4.1.1]{daPaZ}. Hence,
there exists a separable Hilbert space $\rK$, called  the Reproducing Kernel Hilbert Space (RKHS) of  $\mathscr{L}(W_1)$  on $E$, see  in \cite[Section 2.2.2 ]{daPaZ}.
 Moreover,  by \cite[Section 2.3.1]{daPaZ}, $\mathscr{L}(W_1)$ has a covariance operator $Q$ which is of trace class and "it completely characterizes of $W$".
 Let us denote by $i\colon \rK \embed E$ the natural embedding which is known to be $\gamma$-radonifying, e.g.\ $i$ is of Hilbert-Schmidt class if $E$ is a Hilbert space.

The following definition is based on \cite{Brz_1997}, which in turn was based on \cite{Neidhardt_1978} and \cite{Brz+Elw_2000}, see also \cite{Brz_1995}. For various different  notions of
radonifying, see \cite{Vakhania+T+C_1987}.
 \begin{defn}
 In what follows all the spaces are
 assumed to be separable, $H$ is a Hilbert space while $E$ and $X$
are Banach spaces. Let us recall that a bounded linear map $L\colon H \to X$
is  \emph{$\gamma$-radonifying},  if and only if  the image $L(\gamma_H)$ by $L$ of  the canonical
Gaussian distribution  $\gamma_H$ on $H$ extends  to a Borel probability
measure  on  $X$  (necessary  Gaussian)  which  will be denoted by
$\nu_L$. We set
\begin{equation}
\gamma(H,X)\coloneq \left\{  L:H \to  X \mbox{  s.th. }  L \mbox{  is $\gamma$-radonifying  }
\right\}.
\label{3.1}
\end{equation}
Because  of  the  Fernique  theorem  (which  asserts  in  the framework
described  above that  for any $L \in \gamma(H,X)$ there is  $c>0$ such that
$\int_{X} e^{c |x|^2}\, \nu_L(dx)  < \infty$)
the following quantity (for $L\in \gamma(H,X)$) is finite
\begin{equation}
\Vert    L\Vert_{\gamma(H,X)}:=   \left\{    \int_X   |x|^2    \,   \nu_L(dx)
\right\}^{\frac12}.
\label{3.2}
 \end{equation}
We say that $i\colon H \to  E$ is an \emph{Abstract Wiener Space (AWS)} iff
$i$ is linear one-to-one map with range dense in $E$ and
 $i \in \gamma(H,,E)$. If $i\colon H\to E$ is an
AWS, then  the Gaussian measure $\nu_i$ on $E$ will  be denoted by $\mu$
and called the canonical Gaussian measure on $E$.
 \end{defn}

The following result, see  \cite[Theorem 2.1]{Brz_1997}, is  taken from \cite{Neidhardt_1978} and \cite{Baxendale_1976}.
\begin{thm}\label{Th:3.1} Let $H$ be a separable Hilbert space and $X,E$ be separable Banach spaces.
\begin{enumerate}
\item[(i) ] $\gamma(H,X)$  with the norm  $\Vert \cdot\Vert_{\gamma(H,X)}$ from \eqref{3.2} is a separable
Banach space.
\item[(ii)] If $i:H \to E$ is an AWS then the induced map
\begin{equation}
{\tilde i} :\mathscr{ L}(E,X) \ni A \mapsto A\circ i \in \gamma(H,X)
\label{3.3}
\end{equation}
is  well defined  linear and  bounded. The  range ${\tilde i}\left(\mathcal{
L}(E,X)\right)$ of ${\tilde i}$ is a dense subspace of $\gamma(H,X)$. \\
Moreover, for $A \in \mathscr{ L}(E,X)$ one has
\begin{equation}
\nu_{A\circ i}=A(\mu).
\label{3.4}
\end{equation}
\item[(iii)] If  $\{ e_n\}_n$ is an  orthonormal basis of $H$,  $H_n$ is
the  linear  span  of  $e_1,   \ldots,  e_n$,  $P_n$  is  the  orthogonal
projection in $H$ onto $H_n$ and if  $A \in \gamma(H,X)$ then
\begin{eqnarray*}
\Vert AP_n\Vert_{\gamma(H,X)} &\le& \Vert A\Vert_{\gamma(H,X)} ,
\\
 \Vert AP_n- A\Vert_{\gamma(H,X)}  &\to & 0 \; \mbox{ as } n \to \infty.
\end{eqnarray*}
\end{enumerate}
\end{thm}

Since  the natural embedding  $i\colon \rK \embed E$  is  $\gamma$-radonifying, 
by the Kwapie\'n-Szyma\'nski Theorem \cite{Kwapien+Szym_1980} there exists an  orthonormal  basis  $\{e_j\}_{j \in \mathbb{N}}$ of $\rK$ and a sequence of independent standard Gaussian random variables $\{\beta_j\}_{j \in \mathbb{N}}$ such   that
the series $\sum_{j=1}^\infty \beta_j e_j$ is almost surely convergent in $E$, its law is equal to the law of $W_1$ and
\begin{equation}\label{eqn-Kwapien}
\Vert i \Vert_{\gamma(\rK,E)}^2 \sim
\sum_{j=1}^{\infty} \vert e_j \vert_E^2 <\infty.
\end{equation}

The above property \eqref{eqn-Kwapien} implies that
for every  bounded bilinear map $\Lambda \in \mathcal{L}_2(E\times E,H)$ we can define
\begin{equation}\label{eqn-tr K}
\tr_{\rK}(\Lambda):= \sum_{i=1}^{\infty} \Lambda(e_j,e_j) \in E.
\end{equation}

One can deduce, see e.g.
\cite[Remark 4.1.]{BrzPesz2001}, that $W$ generates a unique $K$-cylindrical Wiener process, for which we do not introduce a special notation.

We will indicate below the advantage  of this approach. Before we formulate the next definition let us formulate the following hypothesis.

\begin{hh}\label{hh-stopping times}\;\;\;
\begin{trivlist}
\item[(H1a)] Assume that $\rho$ is a finite stopping time   and $ \tau$  is 
a $\rho$-accessible  stopping time. 
\item[(H1b)] More generally,  we assume that $\rho$ is a finite stopping time, $A \in \mathscr{F}_\rho$   and $ \tau$  is 
a $\rho$-accessible  stopping time on $A$.
\end{trivlist}
We will say that stopping times $\rho$ and $ \tau$  satisfy Hypothesis \ref{hh-stopping times} if condition (H1a) is satisfied. 
We will say that stopping times $\rho$ and $ \tau$ satisfy Hypothesis \ref{hh-stopping times} on event $A$ if condition (H1b) is satisfied.
Obviously, stopping times $\rho$ and $ \tau$  satisfy Hypothesis \ref{hh-stopping times} iff they satisfy   Hypothesis \ref{hh-stopping times} on event $\Omega$.
\end{hh}
By $(\tau_n)_{n=1}^\infty $   we will denote  an announcing sequence for $\tau$  relative to $\rho$ on $B$.  In particular,
 the following two conditions hold a.s.
$\rho \le \tau_n <  \tau$ and $\lim_{n}\tau_n= \tau$. If $\tau=\infty$, then obviously we can take $\tau_n=n$.

Such an  announcing sequence  $(\tau_n)_{n=1}^\infty $  will be fixed in the definitions below.

\begin{defn}\label{def-spaces of processes}
Assume that $p \in [1,\infty)$ and   $S$ is  a separable  normed   vector space  endowed with  a sigma-field  $\mathscr{S}$.
Assume that   stopping times $\rho$   and $ \tau$   satisfy Hypothesis \ref{hh-stopping times}.

In the following definitions we will use a symbol with $\tilde{\;}$ to denote sets of stochastic processes, while the same symbol without 
$\tilde{\;}$ will be used to denote the set of the corresponding equivalence classes. To be precise, 
by  $\tilde{\mathscr{N}}([\rho,\tau),\,S)$ we denote
the set of all progressively measurable  processes
$\eta\colon [\rho,\tau)\times \Omega \to S$. We also set
\begin{align}
\tilde{\mathscr{M}}^p([\rho,\tau),S) &:= \{\xi\in \tilde{\mathscr{N}}([\rho,\tau),S)
:\mathbb{ E}\int_{\rho}^{\tau}\vert\xi _{t}\vert^p_S\;dt<\infty \},
\label{eqn-2.4}
\\
\tilde{\mathscr{N}}^p([\rho,\tau),S) &:= \{\xi\in \tilde{\mathscr{N}}([\rho,\tau),S)
:\int_{\rho}^{\tau}\vert\xi _{t}\vert^p_S\;dt<\infty  \mbox{ a.s.} \},
\label{eqn-2.4'}
\\
\tilde{\mathscr{N}}_{\step}([\rho,\tau),S)
&:=
\left\{\xi\in \tilde{\mathscr{N}}
([\rho,\tau),S): \substack{\exists \mbox{ stopping times } \rho=\tau_0<\tau_1<\ldots <\tau_n= \tau  \\ \text{such that }
 \xi _{t} =\xi (\tau_k),\; t\in [\tau_k,t_{k+1}),\;k=0,1,\ldots ,n-1}\right\}.
\label{eqn-2.5}
\end{align}
By $\mathscr{M}^p([\rho,\tau),S)$, respectively  $\mathscr{N}^p([\rho,\tau),S)$, we denote the space of equivalence classes of
 all  processes from $\tilde{\mathscr{M}}^p([\rho,\tau),S)$, respectively from $\tilde{\mathscr{N}}^p([\rho,\tau),S)$.

  The space 
of all equivalence classes of progressively measurable processes
$\xi\colon  [\rho,\tau)\times \Omega \to S$ such that  for every $\omega \in \Omega$, the trajectory
\begin{equation}\label{eqn-trajectory}
\xi(\cdot,\omega) \colon [\rho,\tau) \to S
\end{equation}
is continuous will be denoted by
$\tilde{\mathfrak{C}}([\rho,\tau) ,S)$.
The space of all equivalence classes of processes belonging to  $\tilde{\mathfrak{C}}([\rho,\tau) ,S)$
 will be denoted by $\mathfrak{C}([\rho,\tau) ,S)$.

We denote by $\widetilde{\mathcal{M}}^p_{\loc}([\rho,\tau ,S)$ the space of all progressively measurable processes
$\xi: [\rho,\tau)\times \Omega \to S$ such that,
\begin{equation}\label{eqn-M^p_loc}
\mathbb{ E} \int_{\rho}^{\tau_n} \vert \xi _{t} \vert^p \, dt <
\infty,\quad \forall n.
\end{equation}
Finally, we denote by $\mathscr{M}^p_{\loc}([\rho,\tau) ,S)$ the space of equivalence  classes of all
 processes $\xi$ which belong to $\widetilde{\mathcal{M}}^p_{\loc}(\rho,\tau ,S)$.
\end{defn}

\begin{remark}\label{rem-spaces of processes}
It is easy to prove that a progressively measurable process
$\xi: [\rho,\tau)\times \Omega \to S$
 belongs to $\widetilde{\mathcal{M}}^p_{\loc}(\rho,\tau ,S)$ if and only if 
\begin{equation}\label{eqn-xi in N}
 \mathbb{P} \Bigl( \bigr\{ \omega \in \Omega: \int_{\rho}^{t} \vert \xi (s,\omega) \vert^p \, ds < \infty \mbox{ for each } t\in [\rho(\omega),\tau(\omega))  \Bigr\}\Bigr).
 \end{equation}
Thus, the definition of the space $\mathscr{M}^p_{\loc}([\rho,\tau) ,S)$ is independent of the announcing sequence. Note also that $\mathscr{N}_{\step}([\rho,\tau),S) \subset \mathscr{N}^p_{\loc}([\rho,\tau) ,S) $.
 \end{remark}

\begin{remark}\label{rem-topologies}
The space $\mathfrak{C}([0,\infty),H)$ is metrizable in the following way. First of all,  we notice that
\[
\mathfrak{C}([0,\infty),H)=L^{0}\bigl(\Omega,C([0,\infty),H)\bigr)
\]
where $X\coloneq C([0,\infty),H)$ is a Polish space endowed with the classical metric
\[
d_X(\omega_1,\omega_2):= \sum_{n=1}^\infty 2^{-n}\frac{\sup_{t \in [0,n]} \vert \omega_{1}(t)-\omega_{2}(t)\vert  }{1+ \sup_{t \in [0,n]} \vert \omega_{1}(t)-\omega_{2}(t)\vert}
\]
and
$L^{0} \bigl(\Omega,X\bigr)$ is the space of equivalence of all $X$-valued random variables (cf.\ Definition \ref{defn:L0})
endowed with the Ky Fan metric, see
 \cite[Section 1.1.5 Convergence of random variables]{Applebaum_2005} and/or  \cite[p.268]{Billingsley_2008},
\begin{align}\label{Ky-Fan}
d_{KF}(\xi_1,\xi_2):=\inf \left\{ \eps>0: \mathbb{P}\left\{ d_X(\xi_1,\xi_2)>\eps\right\} \leq \eps\right\} .
\end{align}
Note that on the space $\mathscr{L}^{0}\bigl(\Omega,X\bigr)$ of all $X$-valued random variables the function $d_{KF}$ defined above is only a pseudo-metric.

The space $\mathscr{N}^p_{\loc}([0,\infty),S)$ can be metrized  in a similar way. Indeed,   we notice that
\[
\mathscr{N}^p_{\loc}([0,\infty),S)=L^{0}\bigl(\Omega, L^p_{\loc}([0,\infty),S) \bigr)
\]
where $Y:=L^p_{\loc}([0,\infty),S)$ is a Polish space endowed with the classical metric
\[
d_Y(\omega_1,\omega_2)\coloneq  \sum_{n=1}^\infty 2^{-n}\frac{ \bigl( \int_{0}^n \vert \omega_{1}(t)-\omega_{2}(t)\vert^p \, dt \bigr)^{1/p} }{1+  \bigl( \int_{0}^n \vert \omega_{1}(t)-\omega_{2}(t)\vert^p \, dt \bigr)^{1/p}}.
\]
\end{remark}

We give the following standard result (see e.g. \cite[Proposition 4.20 and Theorem 4.36]{daPaZ}):

\begin{thm}\label{thm-Ito integral}
Assume that $H$ is a separable Hilbert space. Then there exists a unique linear continuous map $I \colon \mathscr{N}^2_{\loc}([0,\infty),\gamma(\rK,H)) \rightarrow \mathfrak{C}([0,\infty),H)$  such that for every $\xi \in \mathscr{N}_{\step}([0,\infty),\gamma(\rK,H))$ represented as
\begin{align}
\xi _{t} =\sum_{k=0}^{n-1} \xi (t_k) \mathds{1}_{[t_k,t_{k+1})}( t), \;\; t\in [0,\infty),
\label{eqn-2.5'}
\end{align}
where $0=t_0< t_1< \cdots <t_{n-1}<t_{n}<\infty$ and $\xi(t_n)=0$, the following holds
\begin{equation}\label{eqn-Ito-integral-def}
    I(\xi)(t)=\sum_{k=0}^{n-1} \xi(t_k)\bigl(W_{t_{k+1}\wedge t}-W_{t_k\wedge t}\bigr), \quad t \geq 0.
\end{equation}
We will denote the process $I(\xi)$ by $\int_0^{\cdot} \xi_s dW_s$.

If $\mathscr{N}^2([0,\infty),\gamma(\rK,H))$, then a.s. the limit $\lim_{t \to \infty} [I(\xi)](t)$ exists and it will be denoted by
\begin{equation*}
\int_0^{\infty} \xi_s dW_s.
\end{equation*}

Moreover, if $\xi \in \mathscr{M}^2_{\loc}([0,\infty),\gamma(\rK,H))$, then the process $\int_0^{\cdot} \xi_s dW_s$ is an $H$-valued $\mathbb{F}$-martingale and
for every $T>0$, the following It\^o isometry formula the Burkholder inequality, i.e. for every $p \in [1,\infty)$ there exists $C_p>0$,   holds
\begin{equation}\label{eqn-Ito}
 \mathbb{E} \vert  \int_0^T \xi_s \,dW_s \vert^2=
 \mathbb{E}   \int_0^T \Vert \xi_s  \Vert_{\gamma(\rK,H)}^2\,ds.
\end{equation}
\begin{equation}\label{eqn-Burkholder}
 \mathbb{E} \bigl( \sup_{t \in [0,T]} \vert  \int_0^t \xi_s \,dW_s \vert \bigr)^p \leq C_p
 \mathbb{E}\Bigl[  \bigl(  \int_0^T \Vert \xi_s  \Vert_{\gamma(\rK,H)}^2\,ds \bigr)^{p/2} \Bigr]
\end{equation}
If $\xi \in \mathscr{M}^2([0,\infty),\gamma(\rK,H))$, then the  following  generalizations of \eqref{eqn-Ito}  and  \eqref{eqn-Burkholder} hold,
\begin{equation}\label{eqn-Ito-infty}
 \mathbb{E} \vert  \int_0^\infty \xi_s \,dW_s \vert^2=
 \mathbb{E}   \int_0^\infty  \Vert \xi_s  \Vert_{\gamma(\rK,H)}^2\,ds.
\end{equation}
\begin{equation}\label{eqn-Burkholder-infty}
 \mathbb{E} \bigl( \sup_{t \in [0, \infty)} \vert  \int_0^t \xi_s \,dW_s \vert \bigr)^p \leq C_p
 \mathbb{E}\Bigl[  \bigl(  \int_0^\infty \Vert \xi_s  \Vert_{\gamma(\rK,H)}^2\,ds \bigr)^{p/2} \Bigr]
\end{equation}
\end{thm}

In order to be able to rigorously define the It\^o integral on random intervals we need the following result called "the Localization Theorem", see \cite[Proposition 2.12]{Brz+Elw_2000}  which origins can be tracked to
\cite[Theorem 2 in Chapter II, p. 44]{Gih-Sko-72}. One can also consult \cite[Theorem 5, p. 106]{Elw82}, \cite[Theorem 4.4.7, p. 77]{Friedman_2006} and  \cite[Theorems 6.3 and 6.4, pp. 87-89]{Steele_2001} for the case of space $\mathscr{N}^2_{\loc}$ and \cite[Proposition 7.5, p. 98]{Steele_2001} for the general case.

\begin{thm}
\label{thm-localization}  Assume  that $\xi_k\in
\mathscr{N}^2_{\loc}([0,\infty),\gamma(\rK,H))$, $k=1,2$. Assume
that $\tau
$ is a stopping time and for Lebesgue almost every $s\ge 0$,  
\begin{equation}
\xi_{1}(s)=\xi_{2}(s) \mbox{ a.s.  on } \Omega_s(\tau).    
\end{equation}
 Then,  a.s.
\begin{equation}\label{eq-localization}
\int_0^t \xi_1(s)\, dW_s=
\int_0^t \xi_2(s)\, dW_s \mbox{ for } t\in [0,\tau(\omega)] \cap [0,\infty).
\end{equation}
\end{thm}

Next we are going to define an It\^o integral  when both the upper and the lower  limits are   stopping times.

\begin{defn}\label{def-Ito integral for stopping time}
Assume that  $\rho$ and $ \tau$  stopping time     satisfying Hypothesis 
\ref{hh-stopping times}.  Assume also that  $\xi\in
\mathscr{N}_{\loc}([0,\infty),\gamma(\rK,H))$ is such that the process
\begin{equation}\label{eqn-1xi}
  \mathds{1}_{[\rho,\tau)}\xi: [0,\infty) \times \Omega \mapsto \mathds{1}_{[\rho(\omega),\tau(\omega))}(t) \xi(t,\omega) \in \gamma(\rK,H)
\end{equation}
belongs to the space $
\mathscr{N}^2_{\loc}([0,\infty),\gamma(\rK,H))$. Then we define
\begin{align}
\int_\rho^{t \wedge \tau} \xi_{s}\, dW_s&:=\int_0^t  \mathds{1}_{[\rho,\tau)}\xi_s dW_s, \;\; t\geq 0,\label{eqn-Ito integral for stopping time}\\
\int_\rho^{\tau} \xi_{s}\, dW_s&:=\int_0^\infty  \mathds{1}_{[\rho,\tau)}\xi_s dW_s.\label{eqn-Ito integral for stopping time-2}
\end{align}
\end{defn}

\begin{remark}\label{rem-Ito integral for stopping time} In view of Remark \ref{rem-stopping time-progressive measurability}, the above definition makes sense. Indeed, the RHS's of identities \eqref{eqn-Ito integral for stopping time} and \eqref{eqn-Ito integral for stopping time-2} exist.
\end{remark}

From Theorems \ref{thm-Ito integral} and \ref{thm-localization} we have the following corollary.
\begin{cor}\label{cor-int_rho^t=0} Under the hypotheses listed in Definition \ref{def-Ito integral for stopping time}, the process \[\int_\rho^{\cdot \wedge \tau} \xi_{s}\, dW_s =\bigl(\int_\rho^{t \wedge \tau} \xi_{s}\, dW_s : t\geq 0\bigr),\]
introduced in Definition \ref{def-Ito integral for stopping time},
belongs to     $\mathfrak{C}([0,\infty),H)$ and, a.s. 
\begin{equation}\label{eqn-int_rho^t=0}
    \int_\rho^{t \wedge \tau} \xi_{s}\, dW_s=0 \mbox{ for   } t \in [0,\rho].
\end{equation}
Moreover, if the process 
  $\mathds{1}_{[\rho,\tau)}\xi$ defined by \eqref{eqn-1xi} 
belongs to the class $\mathscr{M}^2_{\loc}([0,\infty),\gamma(\rK,H))$, then 
then the process $\int_0^{\cdot} \xi_s dW_s$ is an $H$-valued $\mathbb{F}$-martingale and appropriate generalizations of 
 the  It\^o isometry \eqref{eqn-Ito} and  the Burkholder inequalities \eqref{eqn-Burkholder} 
hold true. 
 \end{cor}

\begin{remark}\label{rem-thm-Ito integral} Note that if $\xi \in \mathscr{N}^2_{\loc}([0,\infty),\gamma(\rK,H))$, then for every $T \in [0,\infty)$, the process
$\mathds{1}_{[0,T)}\xi$ belongs to $\mathscr{N}^2([0,\infty),\gamma(\rK,H))$.  Moreover, by the localization Theorem \ref{thm-localization},  a.s.
\[
\int_0^t \xi_{s}\, dW_s=\int_0^t \mathds{1}_{[0,T)}(s) \xi_{s}\, dW_s, \;\;\; t\in [0,T].
\]
In particular, if $\xi \in \mathscr{N}^2_{\loc}([0,\infty),\gamma(\rK,H))$, then for every $t \in [0,\infty)$,
\begin{equation}\label{eqn-Ito integral-2}
\int_0^t \xi_{s}\, dW_s=\int_0^\infty \mathds{1}_{[0,t)}(s) \xi_{s}\, dW_s,
\end{equation}
Thus, if $\tau$ is a constant random variable equal to $t$, the formula \eqref{eqn-Ito integral for stopping time-2} reduces  to equality \eqref{eqn-Ito integral-2}. In other words,
our definition formula \eqref{eqn-Ito integral for stopping time-2} extends the formula \eqref{eqn-Ito integral-2} from constant stopping times to arbitrary stopping times.
\end{remark}

An alternative definition of the It\^o integral  would be by stopping at time $\tau$ the continuous process $I(\xi)= \bigl(\int_0^t \xi_{s}\, dW_s: t \geq 0 \bigr)$. This approach was used in \cite[Definition on p. 72]{Friedman_2006}. We have the following result which is a special case of \cite[Lemma A.1]{Brz+Masl+Seidler_2005}.
\begin{prop}\label{prop-def-Ito integral for stopping time-3}
Assume that $\tau$  is a  stopping time and $\xi\in
\mathscr{N}^2_{\loc}([0,\infty),\gamma(\rK,H))$.  Then a.s.
\begin{equation}
\label{eqn-Ito integrals for stopping times-3}
I(\xi)(t \wedge \tau)=\int_0^t \mathds{1}_{[0,\tau)}(s) \xi(s)\, dW_s, \;\;\; t \geq 0.
\end{equation}
\end{prop}

We can now formulate some additional properties related to Theorem \ref{thm-Ito integral}.

\begin{thm}\label{thm-Ito integral-2}
In the framework of Theorem \ref{thm-Ito integral},
if $\xi \in \mathscr{M}^2_{\loc}([0,\infty),\gamma(\rK,H))$, $\tau$  is a  stopping time such that  $\mathds{1}_{[0,\tau)}\xi\in
\mathscr{M}^2([0,\infty),\gamma(\rK,H))$,
then the following It\^o isometry formula and the Burkholder inequality hold, i.e. for every $p \in [1,\infty)$ there exists $C_p>0$,   
\begin{equation}\label{eqn-Ito-2}
 \mathbb{E} \vert  \int_0^\tau \xi_s \,dW_s \vert^2=
 \mathbb{E}   \int_0^\tau \Vert \xi_s  \Vert_{\gamma(\rK,H)}^2\,ds.
\end{equation}
\begin{equation}\label{eqn-Burkholder-2}
 \mathbb{E} \bigl( \sup_{t \in [0,\tau]} \vert  \int_0^t \xi_s \,dW_s \vert \bigr)^p \leq C_p
 \mathbb{E}\Bigl[  \bigl(  \int_0^\tau \Vert \xi_s  \Vert_{\gamma(\rK,H)}^2\,ds \bigr)^{p/2} \Bigr]
\end{equation}
\end{thm}

The following definition deals with the most difficult case
when $\xi$ is in $\mathscr{N}^2_{\loc}([\rho,\tau),\gamma(\rK,H))$ but possibly not in $\mathscr{N}^2([\rho,\tau),\gamma(\rK,H))$ and only on a random set $B \in \mathscr{F}_{\rho}$.

Assume that   $\rho$ and $\tau$ are stopping times  satisfying  Hypothesis
\ref{hh-stopping times} on  an event  $B \in \mathscr{F}_{\rho}$. Assume that
\begin{align*}
\xi\colon [\rho,\tau) \times B \to \gamma(\rK,H)
\end{align*}
and that the process  $\mathds{1}_{ [\rho,\tau) \times B } \xi$ belongs to 
$\mathscr{N}^2_{\loc}([\rho,\tau),\gamma(\rK,H))$.
In particular, $\tau$ is $\rho$-accessible on $B$ and so we can 
choose and fix    an announcing sequence  $(\tau_n)_{n=1}^\infty $  for $\tau$. Because each $\tau_n$ is finite a.s., the process $\xi_n$ defined by
\begin{equation}\label{eqn-xi_n}
 \xi_n\colonequals \mathds{1}_{[\rho,\tau_n) \times B } \,\xi \colon [0,\infty) \times \Omega 
 \to \gamma(\rK,H)
\end{equation}
belongs to $\mathscr{N}^2_{\loc}([0,\infty),\gamma(\rK,H))$. Therefore, for every $n \in \mathbb{N}$, there exists a unique process $x_n=I(\xi_n)$ satisfying all the properties
of Theorems \ref{thm-Ito integral} and \ref{thm-Ito integral-2}. Because $\tau_n \leq \tau_m$ for $n \leq m$, so that $\mathds{1}_{[\rho,\tau_n)}\mathds{1}_{[\rho,\tau_m)}=\mathds{1}_{[\rho,\tau_n)}$,
by the localization Theorem \ref{thm-localization}, we infer that  a.s.
\begin{equation}\label{eq-x_n=x_m}
x_n(t,\omega)=x_m(t,\omega) \mbox{ for } t\in [0,\tau_n(\omega)].
\end{equation}
This argument is similar to the proof of \cite[Proposition 7.2]{Steele_2001}, although that books deals with a different question. Next, we argue differently.
The just proven property \eqref{eq-x_n=x_m} together with the assumptions that a.s. $\tau_n \toup \tau$ allows us to define a process $x\colon [0,\tau)\times \Omega \to H$,
as the set theoretical union of the processes $(x_n)_{n=1}^\infty$. This process $x$ is a continuous process. It is also adapted (and hence progressively measurable).

Arguing as in the proof of \cite[Proposition 7.4]{Steele_2001}, we can prove that the above construction does not depend on the choice of an announcing sequence:
 \begin{prop}\label{prop-indepence of an announcing sequence}
 Assume that $(\tau_n)_{n=1}^\infty $  and $(\hat{\tau}_n)_{n=1}^\infty $
 are two announcing sequences   for a $\rho$-accessible stopping time $\tau$.
 Let $x\colon [0,\tau)\times \Omega \to H$, resp. $\hat{x}\colon [0,\tau)\times \Omega \to H$,  be the processes defined above, using the sequence $(\tau_n)_{n=1}^\infty $, respectively  $(\hat{\tau}_n)_{n=1}^\infty $. Then these two processes are equivalent. Moreover, 
 \begin{enumerate}
 \item[(a)]    up to time $\rho$, process $x$ is identically equal to $0$ and 
 \item[(b)]  process $x$ is identically equal to $0$ on $[0,\infty) \times B^{\mathrm{c}}$.
 \end{enumerate}
  \end{prop}

\begin{defn}\label{eqn-def-general-notation}
In the above setting ($\rho$ and $\tau$ satisfy  Hypothesis
\ref{hh-stopping times} on  an event  $B \in \mathscr{F}_{\rho}$, $\xi:[\rho,\tau)\times B\to\gamma(\rK,H))$, $\mathds{1}_{ [\rho,\tau) \times B } \xi \in
\mathscr{N}^2_{\loc}([\rho,\tau),\gamma(\rK,H))$), the process $x$ constructed above, restricted to $[\rho,\tau)\times B$, will be denoted by 
\begin{equation}
\int_\rho^t \xi_s dW_s,  \;\;\; (t,\omega) \in [\rho,\tau)\times B.    
\end{equation}
\end{defn}

It follows from Definition \ref{eqn-def-general-notation}, that for every $\omega \in B$, the trajectory \[
[\rho(\omega),\tau(\omega))\ni t \mapsto 
\int_\rho^{t} \xi_s dW_s \in H\]
is a continuous function.

We give also the following generalization of the Localization Theorem \ref{thm-localization} (see also \cite[Proposition 2.10]{Brz+Elw_2000}):
\begin{lem}[Coincidence criterion]\label{lem:coincidence}
Assume that   $\rho$ and $\tau$ are stopping times  satisfying  Hypothesis
\ref{hh-stopping times}. Assume that the two processes  
\[
\xi^i: [\rho,\tau) \times \Omega \to H, \;\;\; i=1,2,
\]
satisfy hypotheses of Definition \ref{eqn-def-general-notation} with $B=\Omega$.  
Assume also that, for some $A\in\mathscr{F}_\rho$, $\xi^1=\xi^2$ a.s. on $[\rho,\tau) \times A$. Then,  a.s. on $A$, we have
\begin{align}\label{eq:coincidence}
\int_\rho^t \xi^1_r dW_r = \int_\rho^t \xi^2_r dW_r,\quad \forall t\in [\rho,\tau).
\end{align}
\end{lem}

\begin{proof}
Choose and fix an announcing sequence $(\tau_n)_n$ for $\tau$. By the Localization Theorem \ref{thm-localization}, applied to the stopping times $\tau_n'=\tau_n1_A+\rho1_{A^c}$, we have
\begin{align*}
\int_\rho^t \xi^1_r dW_r = \int_\rho^t \xi^2_r dW_r,\quad \forall t\in [\rho,\tau_n'].
\end{align*}
By arbitrariness of $n$, we conclude that \eqref{eq:coincidence} holds.
\end{proof}

A fundamental (classical) result in stochastic integration is the stochastic chain rule, namely the It\^o formula, which contains a second-order term. The following version of It\^o formula is for example in \cite[2.4]{BrNeVeWe2008}:

\begin{thm}[It\^o formula]
    Assume that $\rho$ and $\tau$ are stopping times with $\rho$ finite and $\rho\le \tau$ a.s.. Let $A$ be in $\mathcal{F}_\rho$ and assume that $\tau$ is $\rho$-accessible on $A$. Let $X\colon [\rho,\tau)\times A\to H$ be a process satisfying
    \begin{align*}
        X_t = X_\rho +\int_\rho^t B_r \,dr +\int_\rho^t \xi_r\, dW_r,\quad \text{on }[\rho,\tau)\times A,
    \end{align*}
    for some $B\colon [\rho,\tau)\times A\to H$, $\xi\colon [\rho,\tau)\times A\to \gamma(\rK,H)$ with $B1_{[\rho,\tau)}\in \mathscr{N}^1_{loc}$, $\xi1_{[\rho,\tau)}\in \mathscr{N}^2_{loc}$.
    Let $\tilde{H}$  be  a separable Hilbert space and $h\colon H\to \tilde{H}$ be a $C^2$ function. Then we have
    \begin{align}
    \begin{split}\label{eq:Ito_formula_H}
        h(X_t) &= h(X_\rho) +\int_\rho^t Dh(X_r)B_r \,dr +\int_\rho^t Dh(X_r)\xi_r \,dW_r\\
        &\quad +\frac12\int_\rho^t\tr_{\rK}[D^2h(X_r)[\sigma(X_r), \sigma(X_r)]]\, dr,\quad \text{on }[\rho,\tau)\times A,
    \end{split}
    \end{align}
    where $\tr_{\rK}$ has been defined in \eqref{eqn-tr K}.
\end{thm}

Now we give the following approximation/continuity result which can proved along the lines of the proof of \cite[Theorem 4.3.4]{Friedman_1975}.

\begin{lem}\label{lem:convergence_Ito_int}
Let $\tau$ be a stopping time. Let $\xi$ and, for $\delta>0$, $\xi^\delta$ be $\gamma(\rK,H)$-valued progressively measurable processes on $[0,\tau)$ satisfying
\begin{align}
\int_0^\tau \|\xi^\delta_s\|_{\gamma(\rK,H)}^2\,ds + \int_0^\tau \|\xi_s\|_{\gamma(\rK,H)}^2\,ds<\infty\quad \mbox{ a.s. },
\end{align}
\begin{align}
\lim_{\delta\to 0}\int_0^\tau \|\xi^\delta_s -\xi_s\|_{\gamma(\rK,H)}^2\,ds = 0 \text{ in  } \mathbb{P}.
\end{align}
Then, as $\delta\to 0$,
\begin{align*}
\sup_{t\in [0,\tau]} \left\|\int_0^t(\xi^\delta_s-\xi_s)\, dW_s\right\|\to 0 \text{ in  } \mathbb{P}.
\end{align*}
In particular, the above convergence holds if $\xi^\delta\to \xi \in \gamma(\rK,H)$, as $\delta\to 0$, uniformly on $[0,\tau]$,  a.s..
\end{lem}

We close this subsection with a change-of-time-variable property for the Wiener processes and the It\^o integrals, see for instance \cite[Proposition 2.12]{Brz+Elw_2000} and \cite[Definition 4.2]{Brz+Millet_2012}. 
We introduce an auxiliary filtration $\mathbb{F}^\rho=(\cF^\rho_t)_{t \geq 0}$  on $\Omega$ defined by \[\cF^\rho_t= \cF_{\rho+t}, \;\; t \geq 0.\] Let us recall  that since $\rho+t$ is an $\mathbb{F}$-stopping time, the above definition makes sense.  Next we define an auxiliary  process $W^\rho$  by 
\[W^\rho_t=W_{t+\rho}-W_t,  \;\; t \geq 0.
\]
One can show that  that the following assertions hold true.
\begin{trivlist}
\item[(a)] Filtration $\mathbb{F}^\rho$ satisfies the usual assumptions, \cite[Section 2.1.1]{Applebaum_2005}.
\item[(b)] Process  $W^\rho$ is an $E$-valued $\mathbb{F}^\rho$-Wiener process, see Theorem 2.20 in \cite{LeGall_2016} and  section 4.3 in \cite{Brz+Millet_2012} and that the RKHS of $\mathscr{L}(W^\rho_1)$ 
is equal to $\rK$, i.e. the RKHS of $\mathscr{L}(W_1)$. 
\item[(c)] A function   $\tau\colon \Omega \to [0,\infty]$ is a finite $\mathbb{F}$ stopping time such that a.s. $\rho\le \tau$, if and only if $\tau-\rho$ is a finite $\mathbb{F}^\rho$-stopping time. 
\item[(d)] A function   $\tau\colon \Omega \to [0,\infty]$  is a $\rho$-accessible $\mathbb{F}$-stopping time,  if and only if $\tau-\rho$ is an  accessible $\mathbb{F}^\rho$-stopping time. 
\item[(e)] If $\tau$ is finite $\mathbb{F}$-stopping time then   a process \[\xi:  [\rho,\tau] \to M\] is 
$\mathbb{F}$-progressively measurable if and only if the process \[
\xi^\rho:[0,\tau-\rho] \ni (s,\omega) =\xi_{\rho+s}(\omega)\] 
is $\mathbb{F}^\rho$-progressively measurable. 
\item[(f)] Analogous assertions  hold also if $\tau$ is a $\rho$-accessible (possibly infinite) stopping time on $A\in\cF_\rho$. 
\item[(g)]
If $\tau$ is finite  a.s. and $\xi$ is as in Definition \ref{eqn-def-general-notation}, then almost surely on $B$, 
\begin{align}
\label{eq:stoch_int_shift}
\int_\rho^{s+\rho} \xi_r dW_r = \int_0^s \xi^\rho_r dW^\rho_r,\quad \forall s\in[0,\tau-\rho].
\end{align}
\end{trivlist}

\subsection*{SDEs on Hilbert spaces}
\addcontentsline{toc}{subsection}{SDEs on Hilbert spaces}

We consider SDEs on $H$, as a preliminary step for SDEs on manifolds. Let $U$ be an open subset of $H$. Let $b:U\to H$ and $\sigma:U\to \mathscr{L}(E,H)$ be two continuous functions. Let $\rho$ be a finite stopping time and $\zeta$ be an $\mathcal{F}_\rho$-measurable $H$-values random variable. We consider the following It\^o SDE:
\begin{equation}\label{eq:SDE_U_Ito}
	\begin{aligned}
		dX_t &= b(X_t)dt +\sigma(X_t) dW_t,\quad t\ge\rho,\\
		X_\rho &= \zeta.
	\end{aligned}
\end{equation}

\begin{defn}\label{def:SDE_Ito_def}
Given a stopping time $\tau$ such that 
$\tau>\rho$ a.s. on the set $\{\zeta\in U\}$ and $\tau$ is $\rho$-accessible on the set $\{\zeta \in U\}$, a process
\begin{equation*}
X\colon  [\rho,\tau) \times \{\zeta \in U\} \to H,
\end{equation*}
is a local solution to the It\^o SDE \eqref{eq:SDE_U_Ito} on $U$ if almost surely on set $\{\zeta \in U\}$, $X_t$ is in $U$ for every $t\in [\rho,\tau)$ and
\begin{equation}\label{eq:SDE_U-2_Ito}
	\begin{aligned}
		X_t &= \zeta+ \int_0^t  b(X_s)ds +\int_0^t\sigma(X_t)  dW_s
  ,\quad t \in [\rho,\tau).
	\end{aligned}
\end{equation}
If $U=H$, we simply say that $X$ is a local solution to \eqref{eq:SDE_U_Ito}.
\end{defn}

We give the following standard result on local well-posedness on SDEs with smooth coefficients on an open set $U$, see for example \cite[Theorem A.14]{MMS19} for a similar statement.

\begin{prop}\label{prop:wellposed_SDE_Ito}
Assume that  $U$ is  an open subset of a Hilbert space  $H$. 
Assume that functions $b:U \to H$ and $\sigma : U \to  \mathscr{L}(E,H) $ satisfy the following conditions:
\begin{enumerate}
\item[(a)] For every  bounded open subset $V$ of $H$ with $V \subset \bar{V}\subset  U$, the restriction of $b$ to $\bar{V}$ belongs to $C^{0,1}(\bar{V},H)$, i.e. it is Lipschitz on $\bar{V}$ with values in $H$.
\item[(b)] For every  bounded open subset $V$ of $H$ with $V \subset \bar{V}\subset  U$, the restriction of $\sigma$ to $\bar{V}$ belongs to $C^{0,1}(\bar{V},\mathscr{L}(E,H))$, i.e. it is Lipschitz on $\bar{V}$ with values in $\mathscr{L}(E,H)$. 
\end{enumerate}
 Assume also  $\rho$ be a finite stopping time and $\zeta$ be an $\cF_\rho$-measurable $H$-valued random variable. Then  there exist a stopping time $\tau$, such that 
   $\tau>\rho$ a.s. on the set $\{\zeta\in U\}$ and $\tau$ is $\rho$-accessible on the set $\{\zeta \in U\}$,
  and  there exists  a process 
 \begin{equation*}
  X\colon [\rho,\tau) \times \{\zeta\in U\}  \to H,
 \end{equation*}
  such that:
  \begin{enumerate}
      \item[(i)]
  $X$ is a local solution to the It\^o SDE \eqref{eq:SDE_U_Ito} on $U$;
\item[(ii)] if  $V $ is an  open subset of $U$ such that $\bar{V} \subset  U$, then almost surely on the set $\{ \rho< \tau<\infty\}$, the process  $X$ leaves every set $V$, i.e.  $\tau_V < \tau$ almost surely on the set $\{ \rho<\tau<\infty\}$, where 
$\tau_V$ be the first exit time of $X$ from $V$, i.e.\footnote{To avoid ambiguity, we put $\tau_V=\rho$ when $X_\rho  \notin V$, $\tau_V=\tau$ when $X$ does not leave $V$.} 
$\tau_V=\inf\{ t \in [ \rho, \tau): X_t \notin V \}$;
\item[(iii)] if $\tau^\prime\ge \rho$ is a stopping time, which is $\rho$-accessible on $\{\tau^\prime>\rho\}$, and $X^\prime$ is a solution to \eqref{eq:SDE_U_Ito} on $\{\tau^\prime>\rho\}$, then $X=X^\prime$ almost surely on $[\rho,\tau\wedge\tau^\prime)\times \{\zeta\in U\}$.
  \end{enumerate}
\end{prop}

\begin{proof}
We first show uniqueness on every open set $V$ with $\bar{V}\subseteq U$, then we construct a solution $X$ and finally we show property (ii) for such $X$.

Concerning uniqueness, let $V$ be an open set with $\bar{V}\subseteq U$ and, for $i=1,2$, let $X^i$ be a solution to \eqref{eq:SDE_U_Ito} on $\{\tau^i>\rho\}$, for some stopping time $\tau^i\ge \rho$ which is $\rho$-accessible on $\tau^i>\rho$; call $(\tau^{i,n})_n$ the announcing sequence for $\tau^i$, with $\tau^{i,n}=\rho$ on the set $\{\tau^i= \rho\}$, recall $\tau_V:=\inf\{t\in [\rho,\tau)\mid X_t\notin V\}$. Using the Burkholder inequality \eqref{eqn-Burkholder-2} and the Lipschitz conditions on $b$ and $\sigma$, we obtain, for some $C>0$, for every $n$,
\begin{align*}
\E\left[\sup_{s\in [\rho,t\wedge \tau^{1,n}\wedge \tau^{2,n}\wedge \tau_V]} \|X^1_s -X^2_s\|^2\right] \le C\int_0^t \E\left[\sup_{s\in [\rho,r\wedge \tau^{1,n}\wedge \tau^{2,n}\wedge \tau_V]} \|X^1_s -X^2_s\|^2\right] \, dr.
\end{align*}
We conclude by applying Gronwall's lemma, \cite[30.8]{Sche97}, and by arbitrariness of $n \in \N_0$, that
\begin{align}
X^1=X^2 \,\, \text{a.s. on }[[\rho,\tau^1\wedge\tau^2\wedge\tau_V)).\label{eq:uniq_V}
\end{align}

Concerning existence, for $n$ positive integer, call again $U_n=\{x\in H\mid \text{dist}(x,U^c)>1/n\}$. By \cite[15.9 (2)]{MR1471480}, there exists a $C^1$ map $\varphi_n:H\to \R$, bounded with bounded derivative, such that $\varphi_n=1$ on $U_n$ and $\varphi_n=0$ on $U^c$. Hence the coefficients
\begin{align*}
    b^n(x):=b(x)\varphi_n(x),\quad \sigma^n(x)=\sigma(x)\varphi_n(x)
\end{align*}
are globally Lipschitz and globally bounded. Therefore, by classical results, see e.g. \cite[Theorem 7.2]{daPaZ}, there exists a unique solution $X^n$ to the SDE
\begin{align*}
    dX^n = b^n(X^n)dt +\sigma^n(X^n)dW.
\end{align*}
Taking $\tau_n = \inf\{t\ge \rho\mid X^n_t\notin U_n\}$, $X^n$, restricted to $[\rho,\tau_n)\times\{X_\rho\in U_n\}$, is a solution to \eqref{eq:SDE_U_Ito}. Hence, by the uniqueness statement \eqref{eq:uniq_V}, for every $n,m$ positive integer with $m\ge n$, $X^n$ and $X^m$ coincide up to $\tau_n$. Therefore we can take the stopping time $\tau=\sup_n \tau_n$ and define the process $X$ on $[\rho,\tau)\times \{X_\rho\in U\}$ by gluing the processes $X^n$, namely $X=X_n$ on $[\rho,\tau_n)\times \{X_\rho\in U_n\}$. Then $(\tau_n)_n$ is an announcing sequence for $\tau$ on $\{X_\rho\in U\}$ and $X$ is a solution to \eqref{eq:SDE_U_Ito} on $[\rho,\tau)\times \{X_\rho\in U\}$. Moreover, if $X^\prime$ is another solution on $[[\rho,\tau^\prime))$, then, by \eqref{eq:uniq_V}, $X=X^\prime$ almost surely on $[\rho,\tau_n\wedge \tau^\prime)\times \{\zeta\in U\}$, and hence on $[\rho,\tau\wedge \tau^\prime)\times \{\zeta\in U\}$, that is property (iii). 

Finally, concerning property (ii), note that, on $\{\rho<\tau<\infty\}$, we have $\tau_n<\tau$ and $X_{\tau_n}\notin U_n$, hence
\begin{align}
\lim_n\text{dist}(X_{\tau_n},U^c)=0.\label{eq:X_exits_U}
\end{align}
Given any open set $V$ with $\bar{V}\subseteq U$, recall that $\tau_V:=\inf\{t\in [\rho,\tau)\mid X_t\notin V\}$ and assume by contradiction that $A=\{\tau_V=\tau\}\cap \{\rho<\tau<\infty\}$ has positive probability. Since $b$ and $\sigma$ are bounded on $\bar{V}$, by the SDE itself \eqref{eq:SDE_U_Ito} there exists the a.s. limit $X_{\tau_V}=X_\tau\in \bar{V}$ on $A$. On the other hand, by \eqref{eq:X_exits_U}, $X_\tau$ is in $U^c$ a.s. on $A$, which is in contradiction with $X_\tau\in \bar{V}\subseteq U$. Hence $A$ must be $\mathbb{P}$-negligible. The proof is complete.
\end{proof}

We consider also Stratonovich SDEs on $H$. Let $U$ be an open subset of $H$. Let $b:U\to H$ be a continuous function and let $\sigma:U\to \mathscr{L}(E,H)$ be a $C^1$ function. Let $\rho$ be a finite stopping time and $\zeta$ be an $\mathcal{F}_\rho$-measurable $H$-values random variable. We consider the following Stratonovich SDE:
\begin{equation}\label{eq:SDE_U}
	\begin{aligned}
		dX_t &= b(X_t)dt +\sigma(X_t) \bullet dW_t,\quad t\ge\rho,\\
		X_\rho &= \zeta.
	\end{aligned}
\end{equation}

\begin{defn}\label{def:SDE_def}
Given a stopping time $\tau$ such that 
$\tau>\rho$ a.s. on the set $\{\zeta\in U\}$ and $\tau$ is $\rho$-accessible on the set $\{\zeta \in U\}$, a process
\begin{equation*}
X\colon [\rho,\tau) \times \{\zeta \in U\} \to H,
\end{equation*}
is a local solution to the Stratonovich SDE \eqref{eq:SDE_U} on $U$ if almost surely on set $\{\zeta \in U\}$, $X_t$ is in $U$ for every $t\in [\rho,\tau)$ and
\begin{equation}\label{eq:SDE_U-2_Strat}
	\begin{aligned}
		X_t &= \zeta+ \int_0^t  b(X_s)ds +\int_0^t\sigma(X_t)  dW_s + 
  \frac12 \int_0^t \tr_{\rK}[ D\sigma(X_s) \sigma(X_s)] \, ds
  ,\quad t \in [\rho,\tau).
	\end{aligned}
\end{equation}
\end{defn}

Note that, if $\sigma$ is in $C^1(H,\mathscr{L}(E,H))$, then $D\sigma(x)\sigma(x)$ can be identified as a bounded bilinear operator from $E\times E$ into $H$, hence \eqref{eqn-tr K} applies and the last integral in \eqref{eq:SDE_U-2_Strat} is well-defined.

\begin{remark}\label{rem:why_Strat}
    As we will see, the reason to introduce this type of SDEs is in the It\^o formula: if $X$ solves the Stratonovich SDE \eqref{eq:SDE_U} and $h:H\to \tilde{H}$ is a $C^2$ map, with $\tilde{H}$ another separable Hilbert space, then $h(X)$ satisfies
    \begin{align*}
        dh(X) = Dh(X)b(X)dt +Dh(X)\sigma(X)\bullet dW,
    \end{align*}
    where $Dh(X)\sigma(X)\bullet dW$ is a shorthand notation for the Stratonovich integral, see Definition \ref{def-Stratonovich integral} for the precise definition on a general Hilbert manifold. In particular, the second-order term in It\^o formula disappears here and the chain rule takes the same form as in the deterministic case. This is particularly convenient when extending the definition to manifolds, because we do need a connection or a Riemannian structure on the (infinite-dimensional) manifold.
\end{remark}

As a consequence of the previous Proposition \ref{prop:wellposed_SDE_Ito}, applied to the Stratonovich SDE \eqref{eq:SDE_U}, we have the following:

\begin{prop}\label{prop:SDE_wellposed_local}
Assume that  $U$ is  an open subset of a Hilbert space  $H$. 
Assume that functions $b:U \to H$ and $\sigma : U \to  \mathscr{L}(E,H) $ satisfy the following conditions:
\begin{enumerate}
\item[(a)] For every  bounded open subset $V$ of $H$ with $V \subset \bar{V}\subset  U$, the restriction of $b$ to $\bar{V}$ belongs to $C^{0,1}(\bar{V},H)$, i.e. it is Lipschitz on $\bar{V}$ with values in $H$.
\item[(b)] The function $\sigma$ is of $C^1-$class and for every  bounded open subset $V$ of $H$ with $V \subset \bar{V}\subset  U$, 
the \Frechet derivative $D\sigma:\bar{V} \to  \mathscr{L}(H,\mathscr{L}(E,H))$  is  Lipschitz.
\end{enumerate}
 Assume also  $\rho$ be a finite stopping time and $\zeta$ be an $\cF_\rho$-measurable $H$-valued random variable. Then  there exist a stopping time $\tau$, such that 
   $\tau>\rho$ a.s. on the set $\{\zeta\in U\}$ and $\tau$ is $\rho$-accessible on the set $\{\zeta \in U\}$,
  and  there exists  a process 
 \begin{equation*}
  X\colon [\rho,\tau) \times \{\zeta\in U\} \to H,
 \end{equation*}
  such that the properties (i), (ii) and (iii) in Proposition \ref{prop:wellposed_SDE_Ito} hold for the Stratonovich SDE \eqref{eq:SDE_U} in place of the It\^o SDE \eqref{eq:SDE_U_Ito}.
\end{prop}

\subsection*{SDEs on Hilbert manifolds}
\addcontentsline{toc}{subsection}{SDEs on Hilbert manifolds}

Now we consider SDEs on Hilbert manifolds. As mentioned before (see Remark \ref{rem:why_Strat}), and as it is often done classically, we use Stratonovich formulation, because in this way we do need a connection or a Riemannian structure on the (infinite-dimensional) manifold. We also extend first definitions and It\^o formula to the case of random initial times. 

Let $M$ be a separable and paracompact Hilbert manifold modelled on a separable Hilbert space $H$. We consider an SDE on $M$ starting from an initial stopping time $\rho$:
\begin{align}
	\begin{split}\label{eqn-SDE_man}
		dX_t &= b(X_t)dt +\sigma(X_t) \bullet dW_t,\quad t\ge\rho\\
		X_\rho &= \zeta.
	\end{split}
\end{align}
Here $W$ is a Wiener process with values in a separable Hilbert  space $E$. The RKHS of $W$ is denoted by $\rK$. The drift  
\begin{equation}\label{eqn-b}
 b\colon M\to TM   
\end{equation} and the diffusion coefficient 
\begin{equation}\label{eqn-sigma}
\sigma:M\to \mathscr{L}(E,TM)
\end{equation}
are  sections of the appropriate vector bundles and are  assumed to be  continuous and  of $C^1$-class  respectively. The initial datum $\zeta\colon\Omega\to M$ is $\cF_\rho$-measurable.
 
\begin{defn}\label{def:SDE_man_sol}
 Assume that $ \rho$ and $ \tau$    finite  stopping times   satisfy 
Hypothesis \ref{hh-stopping times}.  If   $\zeta$ is an $\mathscr{F}_\rho$-measurable $M$-valued random variable, then 
	an  admissible process $X\colon [[\rho,\tau)) \to M$,
  is called a local solution  to \eqref{eqn-SDE_man}
 if and only if   for every $C^2$ function $h \colon M\to \tilde{H}$ with values in a separable Hilbert space $\tilde{H}$, the following equality holds  a.s.,
	\begin{equation}\begin{split}
			h(X_t) = h(\zeta) +&\int_\rho^t Dh(X_s) \circ b(X_r)\, dr +\int_\rho^t Dh(X_s) \circ \sigma(X_r)\, dW_r \\
   +&\frac12 \int_\rho^t \tr_{\rK}[D(Dh \circ \sigma)(X_r) \circ \sigma(X_r)]\, dr, \quad \forall t\in[\rho,\tau).\end{split} \label{eq:SDE_man_def}
	\end{equation}
We sometimes say that $X$ is a local solution on $[\rho,\tau)$ (sometimes we say also that $X$ is a local solution on $[[\rho,\tau))$). The pair $(\rho,\zeta)$  we will be called  the initial data. The trace of a bilinear map has been defined in \eqref{eqn-tr K}. In the present context it can be expressed as 
	\begin{align}\label{eqn-trace-Q}
		\tr_{\rK}[D(Dh \circ \sigma)(X_r) \circ \sigma(X_r)] = \sum_k \Bigl[ D(Dh\circ \sigma)(X_r) \bigl( \sigma(X_r) (e_k) \bigr) \Bigr] (e_k)
	\end{align}
	with $\{e_k\}_{k}$ being  an   orthonormal basis of the RKHS $\rK$. 
\end{defn}

\begin{remark}\label{rem:time_restriction}
We collect several observations concerning the definitions of (local) solutions to stochastic equations made so far. For this, assume that $X$ is a local solution to the equation \eqref{eqn-SDE_man}. 
\begin{enumerate}
\item If $X$ takes values a.s. in an open subset $U$ of $M$, it is enough that $b$ and $\sigma$ are defined and of class $C^0$, or $C^1$ resp. on $U$. Moreover, by a localization argument as in \cite[Remark 1.6]{MMS19}, we can allow $h$ to be in $C^2(U,\tilde{H})$ (that is, $h$ can be defined only on $U$).
\item Consider the special case that $M=H$ for a separable Hilbert space $H$. Thanks to It\^o formula \eqref{eq:Ito_formula_H}, $X$ is also a solution to the SDE \eqref{eq:SDE_U} on $H$, in the sense of Definition \ref{def:SDE_def}. Conversely, every solution to \eqref{eq:SDE_U} is a solution of \eqref{eqn-SDE_man} in the sense of Definition \ref{def:SDE_man_sol}.
\item If $\rho'$ is a finite stopping time, $\tau'$ is an $\rho'$-accessible stopping time with $\rho\le\rho'<\tau'\le\tau$, then $X$ is also a local solution to \eqref{eqn-SDE_man} on $[[\rho',\tau'))$ with $(X_{\rho'},\rho')$ as initial data.
\end{enumerate}
\end{remark}

\begin{remark}\label{rem-trace-K} Definition \ref{def:SDE_man_sol}, especially the trace term above in \eqref{eqn-trace-Q}, can be presented in an alternative way using the RKHS $\rK$. Since $\sigma\colon M \to \mathscr{L}(E,TM)$ is a section and  $h\colon M \to \tilde{H}$, the function 
\[
\sigma_h\colon M \ni x\mapsto Dh(x) \circ \sigma(x) \in \mathscr{L}(E, \tilde{H}).
\]
Indeed, for every $x \in M$, $\sigma(x) \in \mathscr{L}(E,T_xM)$ and $Dh(x) \in \mathscr{L}(T_x, \tilde{H})$. 
Thus, for every $x \in M$,
\[
D\sigma_h(x) \in  \mathscr{L}\bigl(T_xM, \mathscr{L}( E, \tilde{H}) \bigr)
\]
and so 
\[
D\sigma_h(x) \sigma (x)  \in 
\mathscr{L}\bigl(E, \mathscr{L}(E, \tilde{H})\bigr)\equiv  \mathscr{L}_2(E,E, \tilde{H}).
\]
Therefore, 
\begin{align}\label{eqn-trace-K}
\tr_{\rK} [D\sigma_h(x) \sigma (x)] = 
\tr_{\rK} \bigl[ D(Dh \sigma)(x) \sigma (x)  \bigr]
\in \tilde{H}
\end{align}
is well defined.  
The above formula \eqref{eqn-trace-K} agrees with the expression \eqref{eqn-trace-Q}.

\end{remark}

\begin{defn}\label{def-Stratonovich integral}
    Given a solution $X$ to \eqref{eqn-SDE_man}, a separable Hilbert space  $\tilde{H}$   and a $C^1$-class map $g\colon M\to \mathscr{L}(E,\tilde{H})$,  we define the Stratonovich integral
\begin{align}\label{eqn-Stratonovich integral}
	\int_\rho^t g(X_r) \bullet  dW_r = \int_\rho^t g(X_r)\, dW_r +\frac12 \int_\rho^t \tr_{\rK}[Dg(X_r) \circ \sigma(X_r)]\, dr,\quad t\in[\rho,\tau)
\end{align}
\end{defn}
According to  this definition, the Stratonovich integral depends a priori on the process $X$, $g$ and $\sigma$ separately, though one can show that the definition depends only on $g(X)$. We will not use this fact here.

\begin{rem}
    With the above definition, formula \eqref{eq:SDE_man_def} in the Definition \ref{def:SDE_man_sol} reads as
    \begin{align}
        h(X_t) = h(\zeta) +&\int_\rho^t Dh(X_s) \circ  b(X_r)\, dr +\int_\rho^t Dh(X_s) \circ \sigma(X_r) \bullet dW_r,\  \forall t\in[\rho,\tau).\label{eq:SDE_man_def_Strat}
    \end{align}
\end{rem}

All the above definitions can be extended to the case when $\rho\le \tau$  a.s., $\tau$ is $\rho$-accessible on some $A\in\cF_\rho$ and the equation \eqref{eqn-SDE_man} is satisfied only on $A$:
\begin{defn}\label{def:SDE_man_sol-local}
 Assume that $ \rho$ and $ \tau$    are finite  stopping times   satisfying  
Hypothesis \ref{hh-stopping times}  on an event $A \in \mathscr{F}_\rho$.   Assume also that    $\zeta: A \to M$ is  $\mathscr{F}_\rho$-measurable. Then 
	an  admissible process \[X\colon [\rho,\tau)\times A \to M \]
  is called a local solution  to \eqref{eqn-SDE_man} (on $[\rho,\tau)\times A$)
 if and only if   for every separable Hilbert space $\tilde{H}$ and every $C^2$ function $h \colon M\to \tilde{H}$, the  equality \eqref{eq:SDE_man_def} holds  a.s. on $A$. Here the It\^o integral in \eqref{eq:SDE_man_def} is used in the sense of Definition 
 \ref{eqn-def-general-notation}.
    \end{defn}

The following lemma is an immediate consequence of property \eqref{eq:stoch_int_shift} and the discussion  around  it. It allows to reduce the analysis to the case $\rho=0$. 

\begin{lem}\label{lem:SDE_shift}
 Assume that $ \rho$ and $ \tau$    finite  stopping times   satisfy 
Hypothesis \ref{hh-stopping times}.  Assume that $A\in\cF_\rho$ and    $\zeta: A \to M$ is  $\mathscr{F}_\rho$-measurable and  
	 \[X\colon [\rho,\tau)\times A \to M \]
 is an   admissible process. Then the following two conditions are equivalent.
 \begin{enumerate}
 \item[($\alpha$)]Process  $X$   is a local solution  to problem \eqref{eqn-SDE_man}. 
     \item[($\beta$)] The shifted process  $X^\rho:=X_{\cdot-\rho}$ defined by 
     \begin{equation}\label{eqn-W-shifted}
     X^\rho\colon [0,\tau-\rho)\times A \rightarrow M, \quad (s,\omega) \mapsto X(s+\rho, \omega)  
     \end{equation}
      is a local solution  to \eqref{eqn-SDE_man} with the shifted Wiener process $W^\rho$, i.e. 
	\begin{align}\label{eq:SDE_man_shift}
		\begin{split}
			dX^\rho_t &= b(X^\rho_t)dt +\sigma(X^\rho_t) \bullet dW^\rho_t,\quad t\in[0,\tau-\rho)\\
			X^\rho_0 &= \zeta.
		\end{split}
	\end{align}
 \end{enumerate}
\end{lem}

The next lemma states the invariance property of SDEs under diffeomorphisms:

\begin{lem}\label{lem-Strat_diff_manifold}
Let $M,\tilde{M}$ be metrizable separable Hilbert manifolds with $U \subseteq M$ and  $\tilde{U}\subseteq \tilde{M}$ open subsets.  Assume also that  $\varphi\colon U\to \tilde{U}$ is  a $C^2$ diffeomorphism and that  a process $X\colon [\rho,\tau)\times \{X_\rho\in U\} \to U$ is a local solution to problem \eqref{eqn-SDE_man}.
 Then, the process  \[Y\coloneq\varphi \circ X\colon [\rho,\tau)\times \{X_\rho\in U\} \to \tilde{U}\] is a solution on $[\rho,\tau)\times \{X_\rho\in U\}$  to
	\begin{align}\label{eqn-Y}
		dY = \tilde{b}(Y)\, dt + \tilde{\sigma}(Y) \bullet dW.
	\end{align}
where $\tilde{b}$ and $\tilde{\sigma}$ are the push-forwards of the functions $b$ and $\sigma$ defined by 
\begin{align}\label{eqn-b-tilde}
  \tilde{b}\colon \tilde{U}&\rightarrow TM,\quad y \mapsto T_{\varphi^{-1}(y)}\varphi( b( \varphi^{-1}(y)))  \in T_y\tilde{M},   
  \\
  \label{eqn-sigma-tilde}
  \tilde{\sigma}\colon  \tilde{U} &\rightarrow \mathscr{L}(E,T\tilde{M}), \quad y \mapsto T_{\varphi^{-1}(y)}[\varphi( \sigma( \varphi^{-1}(y)))] \in  \mathscr{L}(E,T_y\tilde{M}),
\end{align}
i.e. $\tilde{b}= T\varphi \circ b \circ \varphi^{-1}    ,\;\; \tilde{\sigma}=
  T\varphi \circ \sigma \circ \varphi^{-1}.
$
\end{lem}

\begin{proof}
    By It\^o formula, namely formula \eqref{eq:SDE_man_def} (see also \eqref{eq:SDE_man_def_Strat}), applied to $h=\varphi$, one writes the SDE for $Y$ and, using that $\varphi$ is a diffemorphism, one easily identifies the coefficients of this SDE with \eqref{eqn-b-tilde} and \eqref{eqn-sigma-tilde}, see \cite[Lemma 1.7]{MMS19} for more details.
\end{proof}

\begin{rem}\label{rem:invariance_SDE_map}
	We can relax the assumption that $\varphi$ is a $C^2$ diffeomorphism, requiring instead the following condition: the function $\varphi\colon U\to V$ is of $C^2$-class  and there exist a continuous section $\tilde{b}\colon V\to TN$ of the tangent bundle and a $C^1$ section $\tilde{\sigma}\colon V\to \mathscr{L}(E,TN)$ of the bundle of linear maps such that,
\begin{equation*}
\tilde{b}\circ \varphi(x) = T\varphi \circ b(x) \mbox{ and }\tilde{\sigma}\circ \varphi(x) = T\varphi \circ \sigma(x),\mbox{ for every $x \in U$}.
\end{equation*}
 In this case, the process $Y=\varphi \circ X$ satisfies
	\begin{align*}
		dY = \tilde{b}(Y)\, dt +\tilde{\sigma}(Y) \bullet dW.
	\end{align*}
 \end{rem}

As a consequence of Lemma \ref{lem-Strat_diff_manifold}, we get the expression of the SDE \eqref{eqn-SDE_man} in charts. In the following, given a countable atlas $(U_\alpha,\psi_\alpha)_\alpha$, we introduce, for every  $\alpha$, with $ V_\alpha=\psi_\alpha(U_\alpha) \subset H$,  the following functions 
\begin{align*}
	& b^\alpha\colon  V_\alpha\to H, \quad b^\alpha = D\psi_\alpha \circ b \circ \psi_\alpha^{-1},\\
	& \sigma^\alpha\colon  V_\alpha\to \mathscr{L}(E,H), \quad \sigma^\alpha = D\psi_\alpha \circ \sigma \circ \psi_\alpha^{-1}.
\end{align*}
To be precise, we put, compare with \eqref{eqn-b-tilde} and \eqref{eqn-sigma-tilde}
\begin{align}\label{eqn-b-alpha}
	  b^\alpha\colon   V_\alpha&\rightarrow H, \quad  x \mapsto D\psi_\alpha(x; b(\psi_\alpha^{-1}(x)), \\
	\label{eqn-sigma-alpha}
   \sigma^\alpha\colon    V_\alpha &\rightarrow \mathscr{L}(E,H), \quad x \mapsto 
 \{ E\ni l \mapsto  D\psi_\alpha(x; \sigma(\psi_\alpha^{-1}(x))l ) \in H \}
\end{align}
where the semicolon in the derivative divides the non-linear component of the derivative from the second linear one.

\begin{lem}\label{cor:change_chart}
Assume that     $\rho$ and $ \tau$ are stopping times satisfying  Hypothesis \ref{hh-stopping times}.	Let $X$ be a $M$-valued process on $[[\rho,\tau))$.  Then the  following assertions are equivalent.
	\begin{itemize}
		\item[\textup{(a)}] $X$ solves \eqref{eqn-SDE_man} on $[\rho,\tau)$.
		\item[\textup{(b)}] $X_\rho = \xi$ and, for every chart $(U,\psi)$, for every stopping time $\rho'$ with $\rho\le \rho'\le \tau$, the process
	\[X^\psi = \psi \circ X \colon [\rho',\tau') \times\{X_{\rho'}\in U\} \to H, \]	
     where \[\tau'=\inf\{t>\rho'\mid X_t\notin U\}\wedge \tau,\] 
 is a local solution to  the following SDE
		\begin{align}
			dX^\psi_t = b^\psi(X^\psi_t)\, dt +\sigma^\psi(X^\psi_t) \bullet dW_t.\label{eq:SDE_man_chart-d}
		\end{align}
		\item[\textup{(c)}] $X_\rho = \xi$ and, for every countable  atlas $(U_\alpha,\psi_\alpha)$, for every $\alpha$, for every stopping time $\rho'$ with $\rho\le \rho'\le \tau$,  the process  
  \begin{equation}\label{eqn-X^alpha}
   X^\alpha = \psi_\alpha \circ X\colon  [\rho',\tau_\alpha') \times\{X_{\rho'}\in U_\alpha\} \to H,     
  \end{equation} 
    where $\tau_\alpha'=\inf\{t>\rho'\mid X_t\notin U_\alpha\}\wedge \tau$, 
is a local solution to the following SDE
		\begin{align}
			dX^\alpha_t = b^\alpha(X^\alpha_t)\, dt +\sigma^\alpha(X^\alpha_t) \bullet dW_t.
   \label{eq:SDE_man_chart}
\end{align}
		\item[\textup{(d)}] $X_\rho = \xi$ and there exists a countable atlas $(U_\alpha,\psi_\alpha)$ such that, for every $\alpha$, for every stopping time $\rho'$ with $\rho\le \rho'\le \tau$,  the process $X^\alpha$ defined above in \eqref{eqn-X^alpha}
is a local solution to equation  \eqref{eq:SDE_man_chart}. 
	\end{itemize}
\end{lem}

\begin{proof}
The implication (a)$\then$(b) follows applying It\^o formula, precisely Lemma \ref{lem-Strat_diff_manifold}, to the diffeomorphism $\psi$ and the process $X$ solving the SDE on $[[\rho',\tau'))$ (by Remark \ref{rem:time_restriction}). The implications (b)$\then$(c)$\then$(d) are obvious.

It remains to show the implication (d)$\then$(a). Given a $C^2$ map $h\colon M\to \tilde{H}$, with $\tilde{H}$ a separable Hilbert space, we must show\eqref{eq:SDE_man_def}. Let us denote by  $(\tau_n)_n$ an announcing sequence for $\tau$. We fix $n \in \mathbb{N}$ and define the stopping time $\tilde\tau$ as ``the first time before $\tau_n$ where the equality \eqref{eq:SDE_man_def} does not hold'', precisely
\begin{align*}
\tilde{\tau} = \tau_n \wedge &\inf\left\{t\ge \rho\middle| h(X_t) \neq h(X_\rho) +\int_\rho^t Dh \circ b(X_r)\, dr +\int_\rho^t Dh \circ \sigma(X_r)\, dW_r \right.\\
&\left.\quad\quad\quad\quad\quad\quad\quad\quad\quad+\frac12 \int_\rho^t \tr_{\rK}[D(Dh \circ \sigma) \circ \sigma(X_r)]\, dr \right\}
\end{align*}
Note that, by time continuity, the equality \eqref{eq:SDE_man_def} holds up to $t=\tilde\tau$ included, and so
\begin{align}
\begin{split}\label{eq:before_tilde_tau}
h(X_{\tilde\tau}) &= h(X_\rho) +\int_\rho^{\tilde\tau} Dh \circ b(X_r)\, dr +\int_\rho^{\tilde\tau} Dh \circ \sigma(X_r)\, dW_r \\
&\quad +\frac12 \int_\rho^{\tilde\tau} \tr_{\rK}[D(Dh \circ \sigma) \circ \sigma(X_r)]\, dr.
\end{split}
\end{align}

We argue by contradiction assuming that $\tilde{\tau}<\tau_n$ with positive probability. Then there exists a chart $(U_\alpha,\psi_\alpha)$ such that $\tilde{\tau}<\tau_n$ and $X_{\tilde{\tau}}\in U_\alpha$ with positive probability. By the hypothesis (d), $X^\alpha$ satisfies \eqref{eq:SDE_man_chart} on $[\tilde{\tau},\tau')\times \{\tilde{\tau}<\tau_n,X_{\tilde{\tau}}\in U_\alpha\}$, with $\tau'=\inf\{t>\rho'\mid X_t\notin U_\alpha\}\wedge \tau$. Therefore, by It\^o formula applied to $h(X_t)=h\circ \psi_\alpha^{-1}(X^\alpha_t)$ and the relation It\^o-Stratonovich integral (and by the coincidence criterion Lemma \ref{lem:coincidence} for the equality of stochastic integrals), the following identity holds   a.s. on $[\tilde{\tau},\tau')\times \{\tilde{\tau}<\tau_n,X_{\tilde{\tau}}\in U_\alpha\}$, 
\begin{align}
\begin{split}\label{eq:after_tilde_tau}
h(X_t) &= h(X_{\tilde{\tau}}) +\int_{\tilde{\tau}}^t D(h\circ \psi_\alpha^{-1}) \circ b^\alpha(X^\alpha_r)\, dr +\int_{\tilde{\tau}}^t D(h\circ \psi_\alpha^{-1}) \circ \sigma^\alpha(X^\alpha_r)\, dW_r\\
&\quad +\frac12 \int_{\tilde{\tau}}^t \tr_{\rK}[D(D(h\circ \psi_\alpha^{-1}) \circ \sigma^\alpha) \circ \sigma^\alpha(X^\alpha_r)]\, dr\\
& =h(X_{\tilde{\tau}}) +\int_{\tilde{\tau}}^t Dh \circ b(X_r)\, dr +\int_{\tilde{\tau}}^t Dh \circ \sigma(X_r)\, dW_r \\
&\quad +\frac12 \int_{\tilde{\tau}}^t \tr_{\rK}[D(Dh \circ \sigma) \circ \sigma(X_r)]\, dr.
\end{split}
\end{align}
Inserting \eqref{eq:before_tilde_tau} into \eqref{eq:after_tilde_tau}, we infer that   a.s. on $[\tilde{\tau},\tau')\times \{\tilde{\tau}<\tau_n,X_{\tilde{\tau}}\in U_\alpha\}$,
\begin{align*}
h(X_t) &= h(X_\rho) +\int_\rho^t Dh \circ b(X_r)\, dr +\int_\rho^t Dh \circ \sigma(X_r)\, dW_r \\
&\quad +\frac12 \int_\rho^t \tr_{\rK}[D(Dh \circ \sigma) \circ \sigma(X_r)]\, dr.
\end{align*}
But this is in contradiction with the definition of $\tilde{\tau}$. Therefore $\tilde{\tau}=\tau_n$, that is \eqref{eq:SDE_man_def} holds up to $t=\tau_n$. By arbitrariness of $n$, \eqref{eq:SDE_man_def} holds on $[[\rho,\tau))$. The proof is complete.
\end{proof}

We close this section with the situation where the manifold $M$ is embedded into a Hilbert space. For this, let $i\colon M\to \hat{H}$ be an embedding of $M$ as a split submanifold into a separable Hilbert space $\hat{H}$. We can view \eqref{eqn-SDE_man} also as an SDE on $\hat H$:
\begin{align}
\begin{split}\label{eq:SDE_man_emb}
&d X_t = \tilde b^M(X_t)dt +\tilde \sigma^M(X_t)\, dW_t,\\
&X_0 = \zeta.
\end{split}
\end{align}
where $\tilde b^M$ and $\tilde \sigma^M$ are defined on $i(M)$ as
\begin{align*}
\tilde b^M&= Di\circ b(x)+\frac12 \tr_{\rK}[D(Di\circ \sigma) \circ \sigma(X_t) ],\quad x\in i(M),\\
\tilde \sigma^M&= Di \circ \sigma(x),\quad x\in i(M),
\end{align*}
and are set equal to $0$ outside $i(M)$. Note that we could smoothly extend the vector fields onto an open neighborhood of $i(M)$ (for example using Lemma \ref{technical:Lemma}) to avoid the discontinuous vector fields constructed above. However, the values outside $i(M)$ do not matter for our arguments if $X$ takes values in $M$. Let us check now, that solutions to one equation also solve the other and vice versa.

\begin{lem}\label{lem:SDE_embedded}
Let $i \colon M \rightarrow \hat{H}$ be an embedding which identifies $M$ with the split submanifold $i(M)$ of the Hilbert space $\hat{H}$.
\begin{enumerate}
\item If the $M$-valued process $X$ solves \eqref{eqn-SDE_man} on $[\rho,\tau)$, then the $\tilde{H}$-valued process $i\circ X$ solves \eqref{eq:SDE_man_emb} on $[\rho,\tau)$.
\item Conversely assume that $i\circ X$ solves \eqref{eq:SDE_man_emb} on $[\rho,\tau)$ and assume in addition that $\tilde{H}$ is separable. 
Then $X$ solves \eqref{eqn-SDE_man} on $[\rho,\tau)$.
\end{enumerate}
\end{lem}

\begin{proof}
If $X$ solves \eqref{eqn-SDE_man}, we exploit that $i$ is smooth, so in particular $C^2$. Hence the definition of solution, that is equation \eqref{eq:SDE_man_def}, to $h=i$, we get \eqref{eq:SDE_man_emb} for $i\circ X$.

For the converse implication, assume that $i\circ X$ solves \eqref{eq:SDE_man_emb} on $[\rho,\tau)$. We have to show that the equation \eqref{eq:SDE_man_chart} holds for a suitable atlas. We start with an atlas $\{(U_\alpha,\psi_\alpha)\}_{\alpha}$ of $M$. Since $M$ is embedded as split submanifold,for every $\alpha$ and every $x_0 \in U_\alpha$, we may shrink $U_\alpha$ to an open set $\tilde U_{\alpha,x_0}$ which contains $x_0$ such that:
\begin{align}
 \text{there is } \hat U_{\alpha,x_0} \subseteq \hat H \text{ open, and smooth } \hat\psi_{\alpha,x_0}\colon \hat U_{\alpha,x_0}\to \hat{H}=H \times E \notag\\
 \text{ with } i(U_{\alpha,x_0}) \subseteq \hat U_{\alpha,x_0} \text{ and } (\psi_\alpha|_{U_{\alpha,x_0}},0)=\hat \psi_{\alpha,x_0} \circ i|_{U_{\alpha,x_0}}\label{resprop}
 \end{align}
 We can repeat this construction for every point $x_0 \in M$ and create a (possibly uncountable) atlas of submanifold charts $(U_{\alpha,x_0},\psi_\alpha)_{\alpha,x_0}$ for $M$ with the above properties. Now $\hat{H}$ is a separable Hilbert space, thus second countable. Since $M$ can be embedded as a submanifold of $\hat{H}$, also $M$ is second countable and thus a Lindel\"{o}ff space. Hence we can select (out of $(U_{\alpha,x_0},\psi_\alpha)_{\alpha,x_0}$) a countable atlas of manifold charts $(\tilde U_\beta,\tilde \psi_\beta)_\beta$ of $M$ and a corresponding countable atlas of submanifold charts $(\hat U_\beta,\hat \psi_\beta)_\beta$ of $i(M)$ which satisfies \eqref{resprop}.

Fix $\beta$ and a stopping time $\rho'$ with $\rho\le \rho'\le \tau$, call $\tau'=\inf\{t>\rho'\mid X_t\notin \tilde U_\beta\}\wedge \tau$, then $\hat\psi_\beta(i(X))=\tilde \psi_\beta(X)$ on $[\rho',\tau')$ by the previous construction. By It\^o formula applied to the SDE \eqref{eq:SDE_man_emb} for $i\circ X$ and to $\hat\psi_\beta$, $\hat\psi_\beta(i(X))=\tilde \psi_\beta(X)$ satisfies \eqref{eq:SDE_man_chart} on $[\rho',\tau')$. By Lemma \ref{cor:change_chart}, $X$ solves \eqref{eqn-SDE_man}. The proof is complete.
\end{proof}

\subsection*{Local well-posedness and maximal solutions}

In this Subsection, we will show the existence of a maximal solution to the SDE \eqref{eqn-SDE_man} on a Hilbert manifold $M$, under the usual regularity assumptions on the coefficients.

We assume that $\rho$ is a finite stopping time and that the initial condition $\zeta$ is an $\mathscr{F}_\rho$-measurable $M$-valued random variable. 
In order to define a maximal solution to problem \eqref{eqn-SDE_man}, we introduce a partial order in the set $\mathscr{LS}(\zeta)$ of all local solutions to \eqref{eqn-SDE_man}.

\begin{defn}\label{defn:maximal_sol}
Given two local solutions (in the sense of Definition \ref{def:SDE_man_sol-local})  $X^1,X^2$ to \eqref{eqn-SDE_man} resp. on $[[\rho,\tau^1))$, on $[[\rho,\tau^2))$, we say that
\begin{align*}
	(X^1,\tau^1) \preceq (X^2,\tau^2)
\end{align*}
if a.s. the following two conditions are satisfied 
\begin{enumerate}
    \item[(i)] $\tau^1\le \tau^2$, 
    \item[(ii)] $X^1=X^2$ on $[\rho,\tau^1)$. 
\end{enumerate}
We say that $(X^1,\tau^1)\sim(X^2,\tau^2)$ if a.s. $\tau^1=\tau^2$ and $X^1=X^2$.\\
A local  solution to \eqref{eqn-SDE_man} is
 a \emph{local maximal solution}
if and only if it is  a maximal element in the set $\mathscr{LS}(\zeta)$ of all local solutions to \eqref{eqn-SDE_man} with respect to the order $\preceq$.
\end{defn}

\begin{defn}\label{defn:local_uniq}
	We say that local uniqueness holds for the SDE \eqref{eqn-SDE_man} with initial data $(\rho,\zeta)$, if and only if, for every two solutions $X^1$, resp. $X^2$, to \eqref{eqn-SDE_man} on  $[\rho,\tau^1)$, resp.  $[\rho,\tau^2)$,  we have
	\begin{align}\label{eqn:local_uniq}
		X^1=X^2 \text{ on } [\rho,\tau^1\wedge \tau^2) \quad \mbox{ a.s. }.
	\end{align}
 A local solution  $X^1$ to \eqref{eqn-SDE_man} on  $[\rho,\tau^1)$ is said to be a unique local solution if and only if 
 for every local  solutions  $X^2$, to \eqref{eqn-SDE_man} on  $[\rho,\tau^2)$, condition \eqref{eqn:local_uniq} holds.
\end{defn}

We recall that the drift $b\colon M\to TM$ and the diffusion coefficient $\sigma\colon M\to \mathscr{L}(E,TM)$ are given sections assumed continuous and in $C^1$ resp.,  and the initial datum $\zeta\colon\Omega\to M$ is a $\cF_\rho$-measurable random variable.
The following Lemma, which is similar to \cite[Proposition 5.10 and Corollary 5.12]{Brz+H+Raza_2021}, shows that local uniqueness and local existence imply existence of a maximal solution.

\begin{lem}\label{lem:exist_maximal}
Assume that
$\rho$  is a stopping time and $\zeta$ is an $\mathscr{F}_\rho$-measurable $M$-valued random variable. Assume that local uniqueness holds for the SDE \eqref{eqn-SDE_man}.
Then, the following holds.
\begin{enumerate}
\item[(1)] If  $(X,\tau)$ and $(Y,\sigma)$  are  two local solutions to problem
		\eqref{eqn-SDE_man} with initial initial data $(\rho,\zeta)$, then
 a local process $(X\vee Y,\tau \vee \sigma)$ defined by
		\begin{equation}\label{eqn-sum of local solutions}
	(X\vee Y) _{t}(\omega)=
		\begin{cases}
		X_{t}(\omega),  \mbox{ for } t \in [\rho(\omega),\tau(\omega) ), \mbox{ if }  \tau(\omega)  
 \geq \sigma(\omega) ,\\
		Y_{t}(\omega),  \mbox{ for } t \in [\rho(\omega), \sigma(\omega) ) \text{ if }  \tau(\omega) < \sigma (\omega).
		\end{cases}
		\end{equation}
is a local solution to problem 		\eqref{eqn-SDE_man} with initial initial data $(\rho,\zeta)$.
\item[(2)] If additionally,  there exists at least one  local solution $(X^0,\tau^0)$  to problem \eqref{eqn-SDE_man} with initial initial data $(\rho,\zeta)$  and  the local uniqueness holds for \eqref{eqn-SDE_man} with initial initial data $(\rho,\zeta)$, 
 then there exists a unique local maximal solution $(\widehat{X},\hat{\tau})$ to \eqref{eqn-SDE_man} with initial data $(\rho,\zeta)$,
and this maximal solution 
 satisfies  \[(X^0,\tau^0)\preceq (\widehat{X},\hat{\tau}).\]
\end{enumerate}
\end{lem}

\begin{proof}[Proof of part (1) Lemma \ref{lem:exist_maximal}]
Let us choose and fix  two local solutions $(X,\tau)$ and $(Y,\sigma)$    to problem
    \eqref{eqn-SDE_man} with initial initial data $(\rho,\zeta)$.  By the uniqueness assumption  we infer that
\begin{equation}\label{eqn-equivalence}
(X_{\lvert _{[\rho,\tau\wedge \sigma) }},\tau\wedge \sigma) \sim (Y_{\lvert _{[0,\tau\wedge \sigma) } },\tau\wedge \sigma).
\end{equation}
  We define a process $(Z,\xi):=(X\vee Y,\tau \vee \sigma)$  by  formula \eqref{eqn-sum of local solutions}.  
We  now claim that this process  is a local solution to Problem 
\eqref{eqn-SDE_man} with initial initial data $(\rho,\zeta)$. Let us observe that we can prove   the admissibility of the process  $(Z,\xi)$ is very similar to the proof in \cite[Corollary 2.28]{Brz+Elw_2000}.

Let us first observe that by part (ii) of  \cite[Remark 3.4]{Brz+H+Raza_2021}, 
$\xi$ is a $\rho$-accessible stopping time. 
 Now choose and fix a  $C^2$ function $h \colon N\to\tilde{H}$ with values in a separable Hilbert space $\tilde{H}$.  We need to show that   a.s. 
	\begin{align}\label{eqn-Z+xi}
		\begin{split}
			h\circ f(Z_t) = h\circ f(Z_\rho) +&\int_\rho^t Dh \circ Tf \circ B_r dr +\int_\rho^t Dh \circ Tf \circ \sigma(Z_r)\, dW_r  \\ +&\frac12 \int_\rho^t \tr_{\rK}[D(Dh \circ Tf \circ \sigma) \circ \sigma(Z_r)]\, dr, \mbox{ for all $t\in [\rho,\xi)$.}
		\end{split}
	\end{align}
Since $(Z,\tau)$ and $(X,\tau)$ coincide on $[[\rho,\tau))$, applying the coincidence criterion Lemma \ref{lem:coincidence} we get that formula \eqref{eqn-Z+xi} holds on $[[\rho,\tau))$. By symmetry, formula \eqref{eqn-Z+xi} holds on $[[\rho,\sigma))$. Therefore \eqref{eqn-Z+xi} holds on $[[\rho,\tau\vee\sigma))$.

Hence,  we proved that $(Z,\xi)$  is a local solution to Problem 
\eqref{eqn-SDE_man} with initial initial data $(\rho,\zeta)$.
\end{proof}

\begin{proof}[Proof of part (2) Lemma \ref{lem:exist_maximal}]

		Let us choose and fix a local solution $(X^0,\tau^0)$ to problem
		\eqref{eqn-SDE_man} with initial initial data $(\rho,\zeta)$ and let us consider  the family  $\mathscr{ LS}_0$
of all local solution
		$(X,\tau)$ to  the problem \eqref{eqn-SDE_man} with initial initial data $(\rho,\zeta)$ such that $(X^0,\tau^0)\preceq (X,\tau) $.
\medskip

By assumptions, $\mathscr{ LS}_0$ is a non-empty set. 
Hence, due to our assumption about the local uniqueness and  the already  proved assertion (1), by the Amalgamation Lemma, see \cite[Lemma 5.6]{Brz+H+Raza_2021}, we infer that there exists a $\rho$-accessible stopping time
			\[
		\hat{\tau} := \sup\left\{ \tau: (X,\tau)\in \mathscr{ LS}\right\}
		\]
and an admissible process $\widehat{X}: [0, \hat{\tau} )
\times \Omega \to M $,
such that  for all $(X,\tau)\in \mathscr{LS}_0$ and,  a.s.,  
		\begin{equation} \label{eqn-amalgamation_02-a}
		\widehat{X}= X \; \mbox{on} \; [\rho ,\tau).
		\end{equation}
Moreover there is an  $\mathscr{ LS}_0$-valued  sequence $(X^k,\tau^k)$ such that, for a.e. $\omega \in \Omega$,
\begin{equation}\label{eqn-99}
 (\tau^k(\omega))_k \text{ is non-decreasing and }\hat{\tau}(\omega) = \lim_{k \to \infty} \tau^k(\omega).
\end{equation}
Let us denote by $\tilde{\Omega}$ an event of full $\mathbb{P}$-measure such that for every $\omega\in \tilde{\Omega}$, the 
equalities \eqref{eqn-amalgamation_02-a} and \eqref{eqn-99} are 
satisfied.

In order to complete the proof of the existence of a maximal local solution, we shall prove that $(\widehat{X},\hat{\tau})\in \mathscr{LS}_0$. For this aim,  we closely follow the proof of \cite[Theorem 2.26]{Brz+Elw_2000} and proceed in three steps.\smallskip

\textbf{Step 1} We claim  that $(\widehat{X},\hat{\tau})$ is local solution
		 to  the problem \eqref{eqn-SDE_man} with initial initial data $(\rho,\zeta)$. Let us choose and fix a  $C^2$ function $h \colon M\to \tilde{H}$ with values in a separable Hilbert space $\tilde{H}$. 
We introduce  a useful auxiliary notation.
\begin{equation}\label{eqn-107}
  g(x):=Dh \circ b(x)+\frac12 \tr_{\rK}[D(Dh \circ \sigma) \circ \sigma(x)]  , \;\;x \in M.
\end{equation}

We need to show that a.s., 
	\begin{equation}\begin{split}
			h(\widehat{X}_{t }) = h(\zeta) +&\int_{\rho } ^{t} g(\widehat{X}_r)\, dr +\int_{\rho}^{t } Dh \circ \sigma(\widehat{X}_r)\, dW_r 
   \mbox{ for every }  t\in[\rho,\hat{\tau}).
   \end{split} \label{eqn-101}
	\end{equation}
Equality \eqref{eqn-101} holds for $X_k$ in place of $\hat{X}$ on $[\rho,\tau^k)$ a.s.. By the coincidence criterion Lemma \ref{lem:coincidence}, equality \eqref{eqn-101} holds for $\hat{X}$ on $[\rho,\tau^k)$ a.s.. In particular, taking $(\tau_n)_n$ an announcing sequence for $\tau$, for every $n$, equality \eqref{eqn-101} holds on $[\rho,\tau_n\wedge \tau^k]$ a.s. for every $k$, and so, letting $k\to\infty$, on $[\rho,\tau_n\wedge \tau^k]$. Letting then $n\to\infty$, we get that equality \eqref{eqn-101} holds on $[\rho,\tau)$ a.s..

 This completes the existence of a local maximal solution to problem \eqref{eqn-SDE_man} with initial initial data $(\rho,\zeta)$   and thus of   the claim from \textbf{Step 1}.
 \smallskip
 
\textbf{Step 2}. $(X^0,\tau^0)\preceq (\hat{X},\hat{\tau}) $: this follows from the fact that $(\hat{X},\hat{\tau})$ is a local solution and that $\hat{\tau}\ge \tau^0$ and $\hat{X}=X^0$ on $[\rho,\tau^0)$.
\smallskip

\textbf{Step 3} Local maximal solution to problem \eqref{eqn-SDE_man} with initial initial data $(\rho,\zeta)$ is unique.

For this aim let us suppose that $(X^i,\tau^i)$. $i=1,2$, are  two local maximal solutions to problem \eqref{eqn-SDE_man} with initial initial data $(\rho,\zeta)$.
By parti (1), the process $(\newtilde{X},\newtilde{\tau})\colon =(X^1 \vee X^2, \tau^1\vee\tau^2)$ is a local solution to problem \eqref{eqn-sum of local solutions}, i.e. $(\newtilde{X}, \tilde{\tau})\in \mathscr{LS}(\zeta)$.

Now, by construction we have $(X^i, \tau^i) \preceq (\newtilde{X}, \newtilde{\tau})$, for $i\in \{1,2\}$ and there exists $i_0 \in \{1,2\}$ such that
$ (X^{i_0},\tau^{i_0}) \not \sim (\newtilde{X},\newtilde{\tau} )$. This contradicts the maximality of the solution $(X^{i_0},\tau^{i_0})$ and completes the proof of assertion (2) of the Lemma.
\end{proof}

Theorem \ref{thm:maximal} below shows existence and local uniqueness of local maximal solutions to an SDE. We remind that, for metrizable separable manifolds $M$ (in particular separable Hilbert manifolds), there exist countable atlases.

\begin{thm}\label{thm:maximal}
Let $\rho$ be a finite stopping time and let $\xi$ be $\cF_\rho$-measurable. Assume that $b$ is in $C^{0,1}_{loc}$ and that $\sigma$ is in $C^{1,1}_{loc}$, that is, for every $x$ in $M$, there exists a chart $(U,\psi)$ with $x\in U$, such that $b^\psi$ is Lipschitz in $\psi(U)$ and $\sigma^\psi$ is $C^1$ with Lipschitz derivative in $\psi(U)$. Then existence of a local maximal solution and local uniqueness hold for the SDE \eqref{eqn-SDE_man}.
\end{thm}


\begin{proof}
By assumption, there exists an atlas of charts $(U,\psi)$ where the local representative of $b$ and $\sigma$ are resp. $C^{0,1}$ and $C^{1,1}$. We assumed that the Hilbert manifold is paracompact and separable, whence it is Lindel\"of, cf.\ \cite[Corollary 5.1.26]{Eng89}. Thus there exists a countable subcover of the atlas which gives rise to a countable atlas $(U_\alpha,\psi_\alpha)_\alpha$ of $M$ such that $b^\alpha$ and $\sigma^\alpha$ are resp. $C^{0,1}$ and $C^{1,1}$ on $U_\alpha$, for every $\alpha$.
    
	\textbf{Local uniqueness}: Let $(X^1,\tau^1)$ and $(X^2,\tau^2)$ be two local solutions. We let $\bar\rho$ be the first stopping time where $X^1$ and $X^2$ are different:
	\begin{align*}
		\bar\rho = \inf\{t\in[\rho,\tau^1\wedge \tau^2) \mid X^1_t\neq X^2_t\} \wedge \tau^1 \wedge \tau^2.
	\end{align*}
	We will show that $\bar{\rho} = \tau^1 \wedge \tau^2$  a.s., that is local uniqueness. By contradiction, we assume that $A=\{\bar\rho< \tau^1 \wedge \tau^2\}$ has positive probability. On $A$, by continuity of paths, we have $X^1_{\bar\rho} = X^2_{\bar\rho}$. For $i=1,2$, we call $X^{i,\alpha}=\psi_\alpha(X^i)$. By Lemma \ref{cor:change_chart}, for $i=1,2$, $X^{i,\alpha}$ solves, up to the minimum between $\tau^i$ and the first exit time of $X^i$ from $U_\alpha$, the SDE
	\begin{align}
        \begin{aligned}\label{eq:SDE_local_uniq}
		&dY_t = b^\alpha(Y_t)dt +\sigma^\alpha(Y_t) \bullet dW^{\bar\rho}_t,\\
		&Y_0 = \psi_\alpha(X^1_{\bar\rho}).
        \end{aligned}
	\end{align}
	By the regularity assumptions on $b^\alpha$ and $\sigma^\alpha$, we can apply Proposition \ref{prop:SDE_wellposed_local} $X^{1,\alpha}$ and $X^{2,\alpha}$ to \eqref{eq:SDE_local_uniq} must coincide up to a stopping time $\tau$ which is  a.s. $>\bar\rho$ on $A$. But this is in contradiction with the definition of $\bar\rho$. We have proved local uniqueness.\smallskip
 
\textbf{Existence of a maximal solution}: By Lemma \ref{lem:exist_maximal} and local uniqueness, it is enough to show existence of a local solution. For this, we call, for any $\alpha$, $A_\alpha = \{X_\rho \in U_\alpha\}$ (set in $\mathscr{F}_\rho$). By the regularity assumptions on $b^\alpha$ and $\sigma^\alpha$, by Proposition \ref{prop:SDE_wellposed_local} there exists a solution $\newtilde{X}^\alpha$ on $[\rho,\tau)$ to
	\begin{align}
		d\newtilde{X}^\alpha = b^\alpha(\newtilde{X}^\alpha)dt +\sigma^\alpha(\newtilde{X}^\alpha)\bullet dW\label{eq:SDE_chart_exist}
	\end{align}
	for a stopping time $\tau^\alpha\ge \rho$, with $\tau^\alpha$ $\rho$-accessible on $A_\alpha$; we fix an announcing sequence $(\tau^{\alpha,m})_m$ for $\tau^\alpha$. Now we order the set of indices $\{\alpha\}=(\alpha_k)_k$ and we build a local solution $X$ recursively: on $A_{\alpha_1}$ we define $\tau=\tau^{\alpha_1}$ and $X_t=\psi_{\alpha_1}^{-1}(\newtilde{X}^{\alpha_1}_t)$ for every $t \in [\rho,\tau^{\alpha_1})$; for every $k\ge 2$, on $A_{\alpha_k}\setminus (A_{\alpha_1}\cup \ldots \cup A_{\alpha_{k-1}})$ we define $\tau=\tau^{\alpha_k}$ and $X_t=\psi_{\alpha_k}^{-1}(\newtilde{X}^{\alpha_k}_t)$ for every $t \in [\rho,\tau^{\alpha_k})$.	
		
	First we show that $\tau$ is a $\rho$-accessible stopping time. The random time $\tau$ is a stopping time because, for every $t$,
	\begin{align*}
	\{\tau\le t\} = \cup_k \{\tau^{\alpha_k}\le t\} \cap [A_{\alpha_k}\setminus (A_{\alpha_1}\cup \ldots \cup A_{\alpha_{k-1}})] \cap \{\rho\le t\}
	\end{align*}
	is in $\mathscr{F}_t$. Similarly, we can define a sequence of stopping times $\tau^m$ as $\tau^m=\tau^{\alpha_k,m}$ on $A_{\alpha_k}\setminus (A_{\alpha_1}\cup \ldots \cup A_{\alpha_{k-1}})$. Then $(\tau^m)_m$ is a non-decreasing sequence of stopping times, converging to $\tau$ and satisfying $\tau^m<\tau$  a.s. Therefore $\tau$ is a $\rho$-accessible stopping time.
	
	Finally we show that $X$ is a local solution on $[\rho,\tau)$. By Lemma \ref{cor:change_chart}, it is enough to show that, for every $\alpha$, for every stopping time $\rho$ with $\rho\le \rho'\le \tau$, $X^\alpha=\psi_\alpha(X)$ solves on $[\rho',\tau')$
	\begin{align}
		dX^\alpha= b^\alpha(X^\alpha)dt +\sigma^\alpha(X^\alpha)\bullet dW,\label{eq:SDE_chart_exist_2}
	\end{align}
	where $\tau'=\inf\{t\ge\rho' \mid X_t\notin U_\alpha\}\wedge \tau$. By definition, on each set $A_{\alpha_k}\setminus(A_{\alpha_1}\cup \ldots \cup A_{\alpha_{k-1}})$, $X^\alpha=\psi_\alpha(X)$ coincide with $\psi_\alpha\circ \psi_{\alpha_k}^{-1}(\newtilde{X}^{\alpha_k})$ on $[\rho',\tau')$, moreover $\tau'\le \tau=\tau^{\alpha_k}$  a.s.. The process $X^{\alpha_k}$ satisfies \eqref{eq:SDE_chart_exist} (with $\alpha$ replaced by $\alpha_k$), therefore, by the invariance of Stratonovich SDEs Lemma \ref{lem-Strat_diff_manifold}, $\psi_\alpha\circ \psi_{\alpha_k}^{-1}(\newtilde{X}^{\alpha_k})$ satisfies \eqref{eq:SDE_chart_exist_2} on $[\rho,\tau^{\alpha_k})$; by Lemma \ref{lem:coincidence}, $X^\alpha$, restricted to $A_{\alpha_k}\setminus(A_{\alpha_1}\cup \ldots \cup A_{\alpha_{k-1}})$, satisfies \eqref{eq:SDE_chart_exist_2} on $[\rho',\tau')$. Since the sets $A_{\alpha_k}\setminus(A_{\alpha_1}\cup \ldots \cup A_{\alpha_{k-1}})$ form a partition of $\Omega$, the process $X^\alpha$ satisfies \eqref{eq:SDE_chart_exist_2} on $[\rho',\tau')$. The proof is complete.
\end{proof}

Finally we give a criterion for the blow-up to solutions.

\begin{lem}\label{lem:explosion_maximal_time}
Under the hypotheses of Theorem \ref{thm:maximal}, let $\tau$ be the maximal existence time of the solution $X$. Then,  a.s. on $\{\tau<\infty\}$, there does not exist $\lim_{t\nearrow \tau}X_t$.
\end{lem}

\begin{proof}
Fix a countable atlas $(U_\alpha,\psi_\alpha)_\alpha$ as in Theorem \ref{thm:maximal} and take an announcing sequence $(\tau_n)_n$ for $\tau$. Assume by contradiction that the limit $\lim_{t\nearrow \tau}X_t$ exists with positive probability on $\{\tau<\infty\}$. In particular, there exists $\bar{\alpha}$ such that the event $\{\lim_{t\nearrow \tau}X_t \in U_{\bar{\alpha}},\,\tau<\infty\}$ has positive probability, and so there exists also $\bar{n}$ such that
\begin{align*}
    A:=\{\lim_{t\nearrow \tau}X_t \in U_{\bar{\alpha}},\,\tau<\infty,\,X_{\tau_{\bar{n}}}\in U_{\bar{\alpha}}\}
\end{align*}
has positive probability. By Proposition \ref{prop:SDE_wellposed_local}, there exists a process $Y^{\bar{\alpha}}$ which solves \eqref{eq:SDE_man_chart} (the expression of the SDE \eqref{eqn-SDE_man} in the chart $(U_{\bar{\alpha}},\psi_{\bar{\alpha}})$),  on $[\tau_{\bar{n}},\tau')\times \{X_{\tau_{\bar{n}}}\in U_{\bar{\alpha}}\}$, with initial data $(\psi_{\bar{\alpha}}(X_{\tau_{\bar{n}}}),\tau_{\bar{n}})$, and $Y^{\bar{\alpha}}$ leaves every open set $V$ with $\bar{V}\subseteq \psi_{\bar{\alpha}}(U_{\bar{\alpha}})$. Hence $Y=\psi_{\bar{\alpha}}^{-1}(Y^{\bar{\alpha}})$ solves the SDE \eqref{eqn-SDE_man}, on $[\tau_{\bar{n}},\tau')\times \{X_{\tau_{\bar{n}}}\in U_{\bar{\alpha}}\}$, with initial data $(X_{\tau_{\bar{n}}},\tau_{\bar{n}})$, and $Y$ leaves every open set $V$ with $\bar{V}\subseteq U_{\bar{\alpha}}$.
In particular we have $\tau'>\tau$ a.s. on $A$ and hence on $\{X_{\tau_{\bar{n}}}\in U_{\bar{\alpha}}\}$. We define the stopping time $\sigma=\tau 1_{\tau_{\bar{n}}\notin U_{\bar{\alpha}}} +\tau'1_{\tau_{\bar{n}}\in U_{\bar{\alpha}}}$ and the process $Z$ as
\begin{align*}
Z=X \text{ on }[\rho,\tau)\times \{\tau_{\bar{n}}\notin U_{\bar{\alpha}}\},\quad Z=Y \text{ on }[\rho,\tau')\times \{\tau_{\bar{n}}\in U_{\bar{\alpha}}\}.
\end{align*}
The stopping time $\sigma$ is $\rho$-accessible, having as announcing sequence $\sigma_n=\tau_n1_{\tau_{\bar{n}}\notin U_{\bar{\alpha}}}+\tau_n'1_{\tau_{\bar{n}}\in U_{\bar{\alpha}}}$ for $n\ge \bar{n}$, where $(\tau_n')_n$ is announcing sequence for $\tau'$. The process $Z$ solves \eqref{eqn-SDE_man} on $[\rho,\sigma)$ and $\sigma>\tau$ with positive probability. But this is in contradiction with the uniqueness and maximality of $X$. The proof is complete.
\end{proof}

\section{Differentiability of SDEs with respect to initial conditions}
\label{sec-Strat_diff_manifold}

In this section we consider the problem of differentiability of SDEs with respect to the initial condition. By this we mean the following: if $X^x$ is the unique, global-in-time solution to a certain SDE on a Hilbert space $H$, we ask about differentiability of the map
\begin{align*}
H\ni x\mapsto X^x\in C([0,T],L^2(\Omega,H)),
\end{align*}
where the space of continuous mappings is endowed with the usual compact open topology, \cite[Appendix B]{schmeding_2022}, turning it into a Banach space.
For SDEs on manifold the question is similar, though we will consider continuous mappings with values in the space $L^0(\Omega,H)$ (see Definition \ref{defn:L0} below) rather than $L^2$ and we will stop the solution before the maximal existence time. Other stronger concepts of differentiability could be investigated, like differentiability with values in $L^2(\Omega,C([0,T],H))$, or differentiability for a.e. $\omega$ fixed (the latter requires the notion of stochastic flows). However such stronger concepts may require substantial more work or more restrictive assumptions and will not be considered here. The approach we take here come from \cite[Section 9.1]{daPaZ} and \cite[Chapter VII]{Elw82}, though with slight technical differences. In particular, we have taken from \cite[Chapter VII]{Elw82} the localization argument around a compact set.

\subsection*{The case of SDE on separable Hilbert spaces}
\addcontentsline{toc}{subsection}{The case of SDE on separable Hilbert spaces}
We start with the case of an SDE on a separable Hilbert space $H$. Fix $(\Omega,\mathscr{A},(\mathscr{F}_t)_t,\mathbb{P})$ a filtered probability space satisfying the standard assumption, let $W$ be a Wiener process with values in a separable Hilbert space $E$. For technical reasons due to the use of stopping times, we need to work with possibly random drift and random diffusion coefficient. To use the next result from \cite{daPaZ}, we have to introduce a slightly different version of progressive measurability, which is predictability: the predictable $\sigma$-algebra $\cP$ on $[0,\infty)\times \Omega$ is generated by the sets of the form
\begin{align*}
&(a,b]\times A, \quad 0\le a<b,\quad A\in\cF_a,\\
&\{0\}\times A, \quad A\in\cF_0.
\end{align*}
We say that a process is predictable if it is measurable with respect to $\cP$. Since the indicator functions of all the above sets are progressively measurable, every predictable process is also progressively measurable. Not every progressively measurable is predictable, every predictable process has a version (coinciding $dt\otimes \mathbb{P}$a.s.) which is progressively measurable, but we will not use this fact in this section. Any adapted and left-continuous process is predictable. Let $b:[0,T]\times\Omega\times H \to H$ and $\sigma:[0,T]\times\Omega\times H \to \mathscr{L}(E,H)$ be given maps. Consider the It\^o SDE on $H$
\begin{align*}
dX_t &= b(t,\omega,X_t)dt +\sigma(t,\omega,X_t)dW_t,\\
X_0&=x_0.
\end{align*}
To highlight the dependence on the initial condition, we will write $X^{x_0}$ for the solution to the above SDE. The following result follows from \cite[Theorem 9.8]{daPaZ}. Here we denote by $\mathcal{B}(\gamma(\rK,H))\mid_{\mathscr{L}(E,H)}$ the restriction of the Borel $\sigma$-algebra on $\gamma(\rK,H)$ to the space $\mathscr{L}(E,H)$, see \cite[Proposition 2.10]{Brz+Ran22} for details on measurability of $\gamma(\rK,H)$-valued maps.

\begin{thm}\label{thm:deriv_IC_H}
Fix $T>0$ and assume the following conditions:
\begin{itemize}
\item $b$ and $\sigma$ are predictable, that is, they are measurable from $\mathscr{P}\otimes \mathscr{B}(H)$ to respectively $\mathscr{B}(H)$ and $\mathcal{B}(\gamma(\rK,H))\mid_{\mathscr{L}(E,H)}$;
\item uniformly in $t$ (on $[0,T]$) and $\omega$, $b(t,\omega,\cdot)$ is bounded and globally Lipschitz and $\sigma(t,\omega,\cdot)$ is $\gamma(\rK,H)$-bounded and $\gamma(\rK,H)$-globally Lipschitz, that is,
\begin{align*}
&\|b(t,\omega,x)\|_H +|\sigma(t,\omega,x)\|_{\gamma(\rK,H)}\le C, \quad \forall x\in H,\,\forall (t,\omega)\in [0,T]\times\Omega,\\
&\|b(t,\omega,x)-b(t,\omega,y)\|_H +\|\sigma(t,\omega,x)-\sigma(t,\omega,y)\|_{\gamma(\rK,H)}\\
&\quad \le C\|x-y\|_H,\quad \forall x,y\in H,\,\forall (t,\omega)\in [0,T]\times\Omega;
\end{align*}
\item for every $t$ and $\omega$, $b(t,\omega,\cdot)$ is Gateaux-differentiable and its Gateaux derivative is continuous and bounded, uniformly in $t$ and $\omega$, that is,
\begin{align*}
&H\times H\ni (x,h)\mapsto D_xb(t,\omega,x)h \in H \text{ continuous},\,\forall (t,\omega)\in [0,T]\times\Omega\\
&\|D_xb(t,\omega,x)h\|_H\le C\|h\|_H,\quad \forall x,h\in H,\,\forall (t,\omega)\in [0,T]\times\Omega;
\end{align*}
\item for every $t$ and $\omega$, $\sigma(t,\omega,\cdot)$ is Gateaux-differentiable (with respect to the $\mathscr{L}(E,H)$ norm) and its Gateaux derivative is continuous and $\gamma(\rK,H)$-bounded, uniformly in $t$ and $\omega$, that is,
\begin{align*}
&H\times H\ni (x,h)\mapsto D_x\sigma(t,\omega,x)h \in \mathscr{L}(E,H) \text{ continuous},\,\forall (t,\omega)\in [0,T]\times\Omega\\
&\|D_x\sigma(t,\omega,x)h\|_{\gamma(\rK,H)}\le C\|h\|_H,\quad \forall x,h\in H,\,\forall (t,\omega)\in [0,T]\times\Omega.
\end{align*}
\end{itemize}
Then, for every finite $p\ge 2$, the map
\begin{align*}
H\ni x\mapsto X^x \in C([0,T],L^p(\Omega,H))
\end{align*}
is Gateaux-differentiable at every point $x$, with continuous derivative $D_xX$. Moreover, for every $x$ and $h \in H$, there exists a time-continuous version of $A^{x,h}:=D_xX(x)h$, still denoted by $A^{x,h}$, satisfying the SDE
\begin{align}
\begin{split}\label{eq:SDE_derivative}
&dA^{x,h}_t= D_xb(X^x_t)A^{x,h}_tdt + D_x\sigma(X^x_t)A^{x,h}_tdW_t,\\
&A^{x,h}_0=h.
\end{split}
\end{align}
\end{thm}

\begin{rem}
According to \cite[Theorem 9.8]{daPaZ}, we can relax the assumption on $\sigma$, asking Gateaux-differentiability of $\sigma(t,\omega,\cdot)$, as well as continuity of its derivative, with respect to the $\gamma(\rK,H)$ norm rather than the $\mathscr{L}(E,H)$ norm. However, in view of the application to Stratonovich SDEs, we will not need this extension.
\end{rem}

\begin{cor}\label{cor:strong_Gateaux}
Under the assumptions of Theorem \ref{thm:deriv_IC_H}, for every finite $p\ge 2$, the map
\begin{align*}
H\ni x\mapsto X^x \in C([0,T],L^p(\Omega,H))
\end{align*}
is strongly Gateaux-differentiable.
\end{cor}

\begin{proof}
We use an argument close to \cite[Lemma 8D]{Elw82}. We first get a standard Lipschitz bound on the solution $X$ with respect to the initial condition. We fix $x,y \in H$. Using the Lipschitz condition on $b$ and $\sigma$ and the Burkholder-Davis-Gundi inequality \eqref{eqn-Burkholder-2}, we have
\begin{align*}
\mathbb{E}\sup_{t\le T}\|X^x_t-X^y_t\|_H^p &\le C\|x-y\|_H^p + C \mathbb{E} \sup_{t\le T}\left\| \int_0^t (b(r,\omega,X^x_r) -b(r,\omega,X^y_r))\, dr \right\|_H^p\\
&\quad + C \mathbb{E} \sup_{t\le T}\left\| \int_0^t (\sigma(r,\omega,X^x_r) -\sigma(r,\omega,X^y_r))\, dW_r \right\|_H^p\\
&\le C\|x-y\|_H^p + C \int_0^T \mathbb{E}\|b(r,\omega,X^x_r) -b(r,\omega,X^y_r)\|_H^p dr \\
&\quad+ C \int_0^T \mathbb{E}\|\sigma(r,\omega,X^x_r) -\sigma(r,\omega,X^y_r)\|_{\gamma(\rK,H)}^p dr\\
&\le C\|x-y\|_H^p + C \int_0^T \mathbb{E}\sup_{t\le r}\|X^x_t-X^y_t\|_H^p dr.
\end{align*}
Applying Gronwall`s inequality, cf.\ \cite[30.8]{Sche97}, we get
\begin{align*}
\mathbb{E}\sup_{t\le T}\|X^x_t-X^y_t\|_H^p &\le C\|x-y\|_H^p.
\end{align*}
Hence $x\mapsto X^x$ is Lipschitz continuous as a $L^p(C([0,T],H))$-valued map and so in particular as a $C([0,T],L^p(\Omega,H))$-valued map. From Lemma \ref{la:Gateaux} we thus conclude that $x\mapsto X^x$ is strongly Gateaux differentiable. The proof is complete.
\end{proof}

\begin{cor}\label{cor:strong_Gateaux_stopped}
Let $\rho$ be a stopping time. Under the assumptions of Theorem \ref{thm:deriv_IC_H}, for every finite $p\ge 2$, the map
\begin{align*}
H\ni x\mapsto X^{x,\rho}\colonequals X^x_{\cdot\wedge \rho} \in C([0,T],L^p(\Omega,H))
\end{align*}
is Gateaux-differentiable at every point $x$, with continuous derivative $D_xX^{\cdot,\rho}$. The above map is also strongly Gateaux-differentiable,
Finally, for every $x$ and $h \in H$, there exists a time-continuous version of $A^{x,h,\rho}:=D_xX^{\cdot,\rho}h$ such that $A^{x,h,\rho}=A^{x,h}_{\cdot\wedge \rho}$, where $A^{x,h}$ is the time-continuous version of $D_xX(x)h$ as from Theorem \ref{thm:deriv_IC_H}.
\end{cor}

\begin{proof}
The statements about Gateaux-differentiability with continuous derivative and strong Gateaux-differentiability follow from Theorem \ref{thm:deriv_IC_H} and Corollary \ref{cor:strong_Gateaux} applied to the SDE for $X^{x,\rho}$
\begin{align*}
dX^{x,\rho} = b^\rho(t,\omega,X^{x,\rho}_t)dt + \sigma^\rho(t,\omega,X^{x,\rho}_t)dW_t,
\end{align*}
where $b^\rho = 1_{t\le \rho}b$ and $\sigma^\rho = 1_{t\le \rho}\sigma$ (note that $1_{t\le \rho}$ is adapted and left-continuous, hence predictable).
To show the identity $A^{x,h,\rho}=A^{x,h}_{\cdot\wedge \rho}$, we note that, again by Theorem \ref{thm:deriv_IC_H}, $A^{x,h,\rho}$ (precisely its time-continuous version) satisfies the SDE \eqref{eq:SDE_derivative} on $t\le \rho$, hence we have, for $\Delta A:= A^{x,h}_{t\wedge \rho}-A^{x,h,\rho}_t$,
\begin{align*}
&d\Delta A_t= 1_{t\le \rho}D_xb(X^x_t)\Delta A_tdt +1_{t\le \rho} D_x\sigma(X^x_t)\Delta A_tdW_t,\\
&\Delta A_0=0.
\end{align*}
Using the boundedness assumptions on $D_xb$ and $D_x\sigma$ and the Burkholder-Davis-Gundi inequality \eqref{eqn-Burkholder-2}, we have
\begin{align*}
E\sup_{t\le T}\|\Delta A_t\|_H^2 &
\le C \int_0^T 1_{t\le \rho} (E\|D_xb(r,\omega,X^x_r)\Delta A_r\|_H^2 +E\|D_x\sigma(r,\omega,X^x_r)\Delta A_r\|_{\gamma(\rK,H)}^2)\, dr\\
&\le C \int_0^T E\sup_{t\le r}\|\Delta A_t\|_H^2 dr.
\end{align*}
Applying Gronwall inequality, we get that $\Delta A=0$, that is $A^{x,h,\rho}=A^{x,h}_{\cdot\wedge \rho}$. The proof is complete.
\end{proof}

\subsection*{The case of SDE on embedded Hilbert manifolds}
\addcontentsline{toc}{subsection}{The case of SDE on embedded Hilbert manifolds}
Now we consider an SDE on a (separable) Hilbert manifold $M$. Let $W$ be a Wiener process with values in a separable Hilbert space $E$, let $b\colon M\to TM$ be a continuous section and let $\sigma \colon M\to \mathscr{L}(E,TM)$ be a $C^1$ section. We consider the (Stratonovich) SDE on $M$
\begin{align}
\begin{split}\label{eq:SDE_M_x0}
&dX_t = b(X_t)dt +\sigma(X_t)\bullet dW_t,\\
&X_0=x_0.
\end{split}
\end{align}
Also here, to highlight the dependence on $x_0$, we will write $X^{x_0}$ for the solution to the above SDE. We call $\tau=\tau^{x_0}$ the maximal existence time of the above SDE and $\tau_n=\tau^{x_0}_n$ an announcing sequence for $\tau^{x_0}$; by possibly replacing $\tau_n$ with $\tau_n\wedge n$, we can assume without loss of generality that each $\tau_n$ is bounded.

\begin{setup}\label{setting:Hilbert}
From now on, we assume that $M$ is embedded in a separable Hilbert space $\hat{H}$ as split submanifold. For later use we remark that for any compact interval, the space $C([a,b],\hat{H})$ endowed with the compact open topology is separable since $\hat{H}$ is separable, \cite[Theorem 3.4.16]{Eng89} the maximal solution $X=(X_t)_{t\in [0,\tau)}$ to the SDE \eqref{eq:SDE_M_x0}, as an $\hat{H}$-valued process, solves on $[0,\tau)$
\begin{align*}
dX_t = \newtilde{b}(X_t)dt +\tilde\sigma(X_t)dW_t,
\end{align*}
where, for $x \in M$,
\begin{align*}
&\newtilde{b}(x)=b(x)+\frac12\tr_{\rK}[D\sigma\circ \sigma(x) ],\\
&\newtilde{\sigma}(x) = \sigma(x),
\end{align*}
and, with a little abuse of notation, we have identified $M$ with its embedding in $\hat{H}$.
\end{setup}
Hence we can now consider the SDE \eqref{eq:SDE_M_x0} as an It\^o SDE on a Hilbert space $\hat{H}$, though the drift and diffusion coefficients are only defined on $M$.

The strategy of showing differentiability, inspired by \cite[Chapter VII, Sections 7 and 8]{Elw82}, is to extend the coefficients to an open set in $\hat{H}$ and then apply the differentiability result Theorem \ref{thm:deriv_IC_H} on Hilbert spaces. For this, a key trick is that, up to small $\mathbb{P}$-measure sets in $\Omega$, the solution lives on a compact set in $M$, hence it is enough to extend the coefficients in a neighborhood of this compact set.

For $x_0 \in M$, recall that $\tau=\tau^{x_0}$ is the maximal existence time of the SDE \eqref{eq:SDE_M_x0} and $\tau_n=\tau^{x_0}_n<\tau^{x_0}$ is a corresponding announcing sequence of bounded stopping times; take $T_n$ a (deterministic) upper bound for $\tau_n$. We start with the following:

\begin{lem}\label{lem:cpt_1}
Fix $x_0 \in M$, $n$ positive integer. For every $\epsilon>0$, there exists a compact set $K=K_\epsilon$ such that
\begin{align}\label{eq:path_cpt}
\mathbb{P}\{X^{x_0}_t\in K\,\,\forall t\in[0,\tau^{x_0}_n]\} \ge 1-\epsilon.
\end{align}
\end{lem}

\begin{proof}
We omit the superscript $x_0$ from the notation. We consider the probability measure $\text{Law}(X_{\cdot\wedge \tau_n})$ on $C([0,T_n],\hat{H})$. Due to \ref{setting:Hilbert} we can apply the Prokhorov Theorem, \cite[Theorem 11.5.4]{Dud02}, to this single probability measure. For every $\epsilon>0$, there exists a compact set $\newtilde{K}$ in $C([0,T_n],\hat{H})$ such that
\begin{align*}
\mathbb{P}\{X_{\cdot\wedge \tau_n} \in \newtilde{K}\}\ge 1-\epsilon.
\end{align*}
Now it is enough to take $K$ as the image of the compact set $[0,T_n]\times \tilde K$ through the continuous evaluation map $[0,T_n] \times C([0,T_n],\hat{H})\rightarrow \hat{H}, (t,\gamma)\mapsto \gamma(t)$. The proof is complete.
\end{proof}

In the next lemma, we show a continuity property of $X$ with respect the initial condition, and a lower semi-continuity property for the maximal existence time.

\begin{lem}\label{lem:SDE_Lip}
Assume that $\tilde b\colon M\to \hat{H}$ and $\tilde \sigma\colon M\to \mathscr{L}(E,\hat{H})$ are locally Lipschitz maps with respect to the $\hat{H}$ norm and the operator norm, respectively. Fix $x_0 \in M$, $n\in \N$. Let $x^m$, $m\in\mathbb{N}$, be a sequence in $M$ converging to $x_0$, let $X^{x_m}$ be the corresponding solutions, with $\tau^{x_m}$ the associated maximal existence times. Then we have, as $m\to \infty$,
\begin{align*}
&\mathbb{P}\{\tau^{x_m}\le \tau^{x_0}_n\}\to 0,\\
&X^{x_m}_{\cdot\wedge \tau^{x_0}_n}1_{\tau^{x_m}> \tau^{x_0}_n}\to X^{x_0}_{\cdot\wedge \tau^{x_0}_n}\quad \text{in probability on }C([0,T_n],\hat{H}).
\end{align*}
\end{lem}

\begin{proof}
We fix $\epsilon>0$ and we take a compact set $K=K_{\epsilon}$ such that \eqref{eq:path_cpt} holds. Since $K$ is compact, there exists an open neighborhood $U$ of $K$ such that (using $\|\cdot\|_{C^{0,1}(B,B')}$ for the Lipschitz norm from set $B$ to set $B'$)
\begin{align}\label{eq:Lip_U}
\|\tilde b\|_{C^{0,1}(\bar U\cap M,\hat{H})} + \|\tilde \sigma\|_{C^{0,1}(\bar U\cap M,\mathscr{L}(E,\hat{H}))} <\infty.
\end{align}

We take $y \in U$ and show first that, for the solution $X^y$ to the SDE \eqref{eq:SDE_M_x0} (with $y$ as initial condition), the maximal existence time $\tau^y$ satisfies $\tau^y>\sigma^y_U$  a.s., where
\begin{align*}
\sigma^y_U= \inf\{t\ge 0\mid X^y_t\notin U\}\wedge \tau^y
\end{align*}
is the exit time of $X^y$ from $U$. Indeed, for every $\omega$ in the set $\{\tau^y= \sigma^y_U\}$, we have
\begin{align*}
\sup_{t<\tau^y(\omega)} \|\tilde b(X^y_t(\omega))\|_{\hat{H}} + \|\tilde \sigma(X^y_t(\omega))\|_{\mathscr{L}(E,\hat{H})} <\infty;
\end{align*}
therefore we can extend by continuity $X^y$ to $t=\tau(\omega)$, but, by Lemma \ref{lem:explosion_maximal_time}, this is only possible on a $\mathbb{P}$-zero measure set. We conclude that $\tau^y>\sigma^y_U$  a.s. and so the solution $X^y$ is defined  a.s. up to $\sigma^y_U$.

Now we give a Lipschitz property of the solution with respect to the initial condition, when in $U$. Precisely, calling $\sigma^{x_0}_U$ the exit time of $X^{x_0}$ from $U$, we exploit the Lipschitz bound \eqref{eq:Lip_U} on $b$ and $\sigma$, the Burkholder-Davis-Gundi inequality \eqref{eqn-Burkholder-2} and the Gronwall lemma \cite[30.8]{Sche97} as in the proof of Corollary \ref{cor:strong_Gateaux}, to get, for any $p\ge 2$,
\begin{align*}
\mathbb{E}\sup_{t\le T_n} \|X^{x_0}_{t\wedge \sigma^{x_0}_U\wedge \sigma^y_U}-X^y_{t\wedge \sigma^{x_0}_U\wedge \sigma^y_U}\|_{\hat{H}}^p \le C\|x_0-y\|_{\hat{H}}^p.
\end{align*}
In particular, we get, for every $y \in U$,
\begin{align*}
\mathbb{P}\{\sup_{t\le T_n} \|X^{x_0}_{t\wedge \sigma^{x_0}_U\wedge \sigma^y_U}-X^y_{t\wedge \sigma^{x_0}_U\wedge \sigma^y_U}\|_{\hat{H}} \ge \delta \} \le C\frac{\|x_0-y\|_{\hat{H}}^p}{\delta^p}.
\end{align*}
We take $\delta_0=\inf\{\|x-x'\|_{\hat{H}} \mid x\in K,y\in U^c\}$ the distance between $K$ and $U^c$. We have then (we write $\tau_n$ for $\tau^{x_0}_n$)
\begin{align*}
\{X^{x_0}_{t\wedge \tau_n}\in K\,\forall t,\,\, \tau_n\ge \sigma^y_U\}
&\subseteq \{\sup_{t\le T_n} \|X^{x_0}_{t\wedge \sigma^{x_0}_U\wedge \sigma^y_U}-X^y_{t\wedge \sigma^{x_0}_U\wedge \sigma^y_U}\|_{\hat{H}} \ge \delta_0\}
\end{align*}
We choose $R>0$ such that $(R/\delta_0)^p<\epsilon$. Then, for all $y \in B_R(x_0)$, we get
\begin{align}
\begin{split}\label{eq:lsc_time}
\mathbb{P}\{\tau^{x_0}_n\ge \tau^y\} &\le \mathbb{P}\{\tau^{x_0}_n\ge \sigma^y_U\}\\
&\le \mathbb{P}\{X^{x_0}_{t\wedge \tau^n}\notin K\text{ for some } t\} + \mathbb{P}\{X^{x_0}_{t\wedge \tau^n}\in K\,\forall t,\,\, \tau_n\ge \sigma^y_U\}\\
&\le \epsilon +C\frac{\|x_0-y\|^p}{\delta_0^p} \le (1+C)\epsilon.
\end{split}
\end{align}
Moreover we have
\begin{align}
\begin{split}\label{eq:cont_prob}
&\mathbb{P}\{\tau^{x_0}_n\ge \tau^y \text{ or }\sup_{t\le T_n} \|X^{x_0}_{t\wedge \tau^{x_0}_n}-X^y_{t\wedge \tau^{x_0}_n}\|_{\hat{H}} \ge \delta \}\\
&\le \mathbb{P}\{\tau^{x_0}_n\ge \sigma^y_U \text{ or }\sup_{t\le T_n} \|X^{x_0}_{t\wedge \tau^{x_0}_n}-X^y_{t\wedge \tau^{x_0}_n}\|_{\hat{H}} \ge \delta \}\\
&\le \mathbb{P}\{\tau^{x_0}_n\ge \sigma^y_U\} + \mathbb{P}\{X^{x_0}_{t\wedge \tau^n}\notin K\text{ for some } t\}\\
&\quad +\mathbb{P}\{\sup_{t\le T_n} \|X^{x_0}_{t\wedge \sigma^{x_0}_U\wedge \sigma^y_U}-X^y_{t\wedge \sigma^{x_0}_U\wedge \sigma^y_U}\|_{\hat{H}} \ge \delta_0\}\\
&\le (1+C)\epsilon + \epsilon +C\frac{\|x_0-y\|^p}{\delta^p} \le 2(1+C)\epsilon.
\end{split}
\end{align}
Choosing $m$ high enough (such that $y_m$ belongs to $B_R(x_0)$), the inequalities \eqref{eq:lsc_time} and \eqref{eq:cont_prob} imply the desired result.
\end{proof}

\begin{lem}\label{lem:cpt_2}
In the setting of Lemma \ref{lem:SDE_Lip}, for every $\epsilon>0$, there exists a compact set $K=K_\epsilon$ such that
\begin{align*}
\mathbb{P}\{X^{x_m}_t 1_{\tau^{x_m}>\tau^{x_0}_n} \in K,\,\,\forall t\in[0,\tau^{x_0}_n]\} \ge 1-\epsilon,\quad \text{for every }m\in\mathbb{N}_0.
\end{align*}
\end{lem}

\begin{proof}
The proof is similar to that of Lemma \ref{lem:cpt_1}, but using Lemma \ref{lem:SDE_Lip} to get uniformity with respect to $m$. We claim that the family of probability measures $\text{Law}(X^{x_m}_{\cdot\wedge \tau_n})$, $m\in\N_0$, is tight on $C([0,T_n],\hat{H})$: indeed, as $m\to \infty$, by Lemma \ref{lem:SDE_Lip} the sequence
\begin{align*}
X^{x_m}_{\cdot\wedge \tau^{x_0}_n} 1_{\tau^{x_m}>\tau^{x_0}_n}
\end{align*}
converges in probability, and so in law, on the space $C([0,T_n],\hat{H})$. Therefore by Prokhorov`s theorem \cite[Theorem 11.5.4]{Dud02}, for every $\epsilon>0$, there exists a compact set $\newtilde{K}$ on $C([0,T_n],\hat{H})$ such that
\begin{align*}
\mathbb{P}\{X^{x_m}_{\cdot\wedge \tau_n}1_{\tau^{x_m}>\tau^{x_0}_n} \in \newtilde{K}\}\ge 1-\epsilon,\quad \text{for every }m\in\mathbb{N}^+\cup\{0\}.
\end{align*}
It is enough to take $K$ as the image of the compact set $[0,T_n]\times \tilde K$ through the continuous evaluation map $(t,\gamma)\mapsto \gamma(t)$. The proof is complete.
\end{proof}

\begin{defn}\label{defn:L0}
Let $(\Omega,\mathscr{A},P)$ be a probability space and $(E,\lVert \cdot \rVert_E)$ a separable Banach space. Two random variables $X,Y \colon \Omega \rightarrow E$ are equivalent if they agree almost everywhere. We denote by $L^0(\Omega,E)$ the \emph{space of all (equivalence classes of) $E$-valued random variables}. As usual we suppress equivalence classes in our notation. Endow $L^0(\Omega,E)$ with the metric distance
\begin{align*}
d_{L^0}(X,Y):= \E \left[\frac{\|X-Y\|_E}{1+\|X-Y\|_E}\right],
\end{align*}
\end{defn}

Classically, thanks to the inequalities (for every $0<\epsilon<1$)
\begin{align}\label{eq:conv_probab}
\mathbb{P}\{\|X-Y\|_E\ge \epsilon\} \le \frac{2}{\epsilon}d_{L^0}(X,Y),\quad d_{L^0}(X,Y) \le \epsilon + \mathbb{P}\{\|X-Y\|_E\ge \epsilon\},
\end{align}	
the convergence with respect to $d_{L^0}$ is the convergence in probability (see e.g. \cite[4. on p.\ 291]{Dud02} also cf.\ \cite[VII \S 7]{Elw82}). As a result, The space $L^0(\Omega,E)$ is a complete metric space (cf.\ \cite{Fre37} or \cite[Section 9.2]{Dud02} for the proofs with respect to the equivalent Ky Fan metric, \eqref{Ky-Fan}) and a topological vector space with respect to the pointwise vector space operations. It is however in general not a locally convex space (indeed one can show that if $(\Omega, \cA, P)$ has no atoms, then the dual of $L^0(\Omega, E)$ is $\{0\}$). 

For fixed $x_0 \in M$ and fixed $n$, we can view $X^x_{\cdot\wedge \tau^{x_0}_n}$ as an element in $C([0,T_n],L^0(\Omega,\hat{H}))$, which is also a complete metric space (with respect to the compact open topology, cf. \cite[Theorem 4.2.17 and Theorem 4.3.13]{Eng89}).

\begin{rem}\label{rem:conv_sup_P}
Consider an $C([0,T_n],L^0(\Omega,\hat{H}))$-valued  sequence $(Y^m)_{m=1}^\infty$.  If the sequence converges to $Y \in C([0,T_n], L^0(\Omega,\hat{H}))$, then there exists a subsequence $(Y^{m_k})_k$ such that, for every $t \in [0,T_n]$, $(Y^{m_k}_t)_k$ converges to $Y_t$  a.s.. Indeed, we can take $(m_k)_k$ such that
\begin{align*}
\sup_{t\in[0,T_n]} \mathbb{P}\left\{\|Y^{m_k}_t-Y_t\|_{\hat{H}} >\frac{1}{k}\right\} \le 2^{-k}
\end{align*}
and use then the Borel-Cantelli lemma, \cite[8.3.4]{Dud02}.
\end{rem}

In the main result of this section, we want to study the strong Gateaux differentiability of $X^x_{\cdot\wedge \tau^{x_0}_n}$ as $C([0,T_n],L^0(\Omega,\hat{H}))$-valued map. A technical point is that Definition \ref{Gateaux:diff} does not apply here directly to $X^x_{\cdot\wedge \tau^{x_0}_n}$, since, for $x\neq x_0$, the process $X^x_{\cdot\wedge \tau^{x_0}_n}$ may not be defined for all times up to $\tau^{x_0}_n$. Hence we slightly extend the definition of strong Gateaux differentiability as follows: 
\begin{defn}
Let $M$ be a smooth manifold and $x_0 \in M$. Given $h \in T_{x_0}M$, we say that the map
\begin{align*}
M\rightarrow C([0,T_n],L^0(\Omega,\hat{H})), \quad  x\mapsto X^x_{\cdot\wedge \tau^{x_0}_n} 
\end{align*}
is \emph{strongly Gateaux-differentiable in $x_0$ in the direction $h$} if the following two conditions are satisfied:
\begin{enumerate}
\item[(i)] for every sequence $(x_m)_m$ tending to $x_0$,
\begin{align}
\mathbb{P}\{\tau^{x_m}\le \tau^{x_0}_n\}\to 0 \text{ as } m\to \infty;\label{eq:max_time_lsc}
\end{align}
\item[(ii)] the map $x\mapsto X^x_{\cdot\wedge \tau^{x_0}_n}1_{\tau^x>\tau^{x_0}_n}$ is strongly Gateaux differentiable, in the sense of Definition \ref{Gateaux:diff}.
\end{enumerate}
In this case, the Gateaux derivative $A^{x_0,h}$ of $x\mapsto X^x_{\cdot\wedge \tau^{x_0}_n}1_{\tau^x>\tau^{x_0}_n}$ is also called the Gateaux derivative of $x\mapsto X^x_{\cdot\wedge \tau^{x_0}_n}$.
\end{defn}
We are now ready to state and show the main result on differentiability with respect to initial condition.

\begin{thm}\label{thm-Strat_diff_manifold}
Assume that $\tilde b:M\to \hat{H}$ and $\tilde \sigma\colon M\to \mathscr{L}(E,\hat{H})$ are $C^1$. Fix $x_0 \in M$, $n \in \N$. Then the map
\begin{align*}
M\rightarrow C([0,T_n],L^0(\Omega,\hat{H})), \quad  x\mapsto X^x_{\cdot\wedge \tau^{x_0}_n}
\end{align*}
is strongly Gateaux-differentiable at $x_0$ in every direction $h \in T_{x_0}M$. Moreover there exists a time-continuous version of the Gateaux derivative $\frac{d}{dh}X^{x_0}_{\cdot\wedge \tau^{x_0}_n}$ (that is, there exists an $\hat{H}$-valued, progressively measurable process $\tilde{A}^{x_0,h}$, with continuous paths, such that, for every $t\in [0,T_n]$, $\tilde{A}^{x_0,h}=\frac{d}{dh}X^{x_0}_{t\wedge \tau^{x_0}_n}$).
\end{thm}

\begin{proof}
For notational convenience, we omit the superscript $x_0$ on $\tau_n$. We fix $h \in T_{x_0}M$. The property \eqref{eq:max_time_lsc} follows from Lemma \ref{lem:SDE_Lip}. Consider given $C^1$-curves $\gamma, \newtilde{\gamma} \colon (-a,a)\to M$ given sequences $\delta_m\to 0$, $\tilde\delta_m\to 0$, and call $x_m=\gamma(\delta_m)$, $\tilde x_m=\tilde\gamma(\tilde\delta_m)$. Assume that we have
\begin{align}
\lim_m\sup_{k\ge m} \sup_{t\in [0,T_n]}d_{L^0} \left(\frac{X^{x_m}_{\cdot\wedge \tau_n}1_{\tau^{x_m}>\tau_n} -X^{x_0}_{\cdot\wedge \tau_n}}{\delta_m}, \frac{X^{\tilde x_k}_{\cdot\wedge \tau_n}1_{\tau^{\tilde{x}_m}>\tau_n} -X^{x_0}_{\cdot\wedge \tau_n}}{\tilde \delta_k}\right) =0. \label{eq:incr_Cauchy}
\end{align}
Then, taking $\gamma=\tilde\gamma$ and $\delta_m=\tilde\delta_m$, the sequence $(X^{x_m}_{\cdot\wedge \tau_n}1_{\tau^{x_m}>\tau_n} -X^{x_0}_{\cdot\wedge \tau_n})/\delta_m$ is Cauchy and therefore it has a limit $A \in C([0,T_n],L^0(\Omega,\hat{H}))$. By \eqref{eq:incr_Cauchy} the limit $A$ does not depend on the curve $\gamma$ nor on the sequence $\delta_m$. Hence we get the desired strong Gateaux-differentiability property.

We have to show \eqref{eq:incr_Cauchy}. For simplicity of notation, we call
\begin{align*}
\delta X^{m}_t = \frac{X^{x_m}_{\cdot\wedge \tau_n}1_{\tau^{x_m}>\tau_n} -X^{x_0}_{\cdot\wedge \tau_n}}{\delta_m},\quad \delta \tilde X^{m}_t = \frac{X^{\tilde x_m}_{\cdot\wedge \tau_n}1_{\tau^{\newtilde{x}_m}>\tau_n} -X^{x_0}_{\cdot\wedge \tau_n}}{\tilde \delta_m}.
\end{align*}
We fix $\epsilon>0$ and, thanks to Lemma \ref{lem:cpt_2}, we take a compact set $K=K_{\epsilon}$ in $\hat{H}$ such that, for every $m \in \mathbb{N}_0$,
\begin{align*}
\mathbb{P}\{X^{x_m}_t 1_{\tau^{x_m}>\tau^{x_0}_n} \in K\,\,\forall t\in[0,\tau^{x_0}_n]\}, \, \mathbb{P}\{X^{\tilde x_m}_t 1_{\tau^{\tilde x_m}>\tau^{x_0}_n} \in K\,\,\forall t\in[0,\tau^{x_0}_n]\} \ge 1-\epsilon.
\end{align*}
By Lemmas \ref{lem:bounded_extension} and \ref{lem:Hilbert_paracompact}, there exist an open (in $\hat{H}$) neighborhood $U$ of $K$ and two maps $B\colon \hat{H}\to \hat{H}$, $\Sigma \colon \hat{H}\to \mathscr{L}(E,\hat{H})$ such that $B$ and $\Sigma$ coincide with respectively $\newtilde{b}$ and $\newtilde{\sigma}$ on $M\cap U$ and $B$ and $\Sigma$ are globally $C^1$ and bounded with bounded derivative. For $x \in \hat{H}$, we call $Z^x$ the solution of the SDE
\begin{align}
\begin{split}\label{eq:SDE_extended}
dZ^x &= B(Z^x)dt +\Sigma(Z^x)dW,\\
Z^x_0&=x.
\end{split}
\end{align}
By Corollary \ref{cor:strong_Gateaux_stopped}, the map $\hat{H}\ni x\mapsto Z^x_{\cdot\wedge \tau_n}\in C([0,T_n],L^2(\Omega,\hat{H}))$ is strongly Gateaux-differentiable, with derivative $A^\epsilon=A^{\epsilon,x_0,h}$. In particular, for every $\epsilon>0$, we have
\begin{align*}
\lim_m \sup_{t\in[0,T_n]}\mathbb{P}\{\|\delta Z^{m}_t- A^\epsilon_t\|_{\hat{H}}\ge \epsilon\} = \lim_m \sup_{t\in[0,T_n]}\mathbb{P}\{\|\delta \tilde Z^{m}_t- A^\epsilon_t\|_{\hat{H}}\ge \epsilon\} = 0,
\end{align*}
where we have called
\begin{align*}
\delta Z^{m}_t = \frac{Z^{x_m}_{\cdot\wedge \tau_n} -Z^{x_0}_{\cdot\wedge \tau_n}}{\delta_m},\quad \delta \tilde Z^{m}_t = \frac{Z^{\tilde x_m}_{\cdot\wedge \tau_n} -Z^{x_0}_{\cdot\wedge \tau_n}}{\tilde \delta_m},
\end{align*}
and so
\begin{align}
\lim_m\sup_{k\ge m} \sup_{t\in[0,T_n]}\mathbb{P}\{\|\delta Z^{m}_t- \delta \tilde Z^{k}_t\|_{\hat{H}}\ge \epsilon\} = 0\label{eq:diff_Z}
\end{align}
For any $x \in K$, $X^x$ lives in $U\cap M$ and so it solves \eqref{eq:SDE_extended} up to the first exit time from $U$. Therefore, by the uniqueness result Proposition \ref{prop:wellposed_SDE_Ito} for the SDE \eqref{eq:SDE_extended}, for $x \in K$, $X^x_{\cdot\wedge \tau_n}$ and $Z^x_{\cdot\wedge \tau_n}$ coincide  a.s. on the set $\{\tau^x>\tau_n,\,\,X^x_{t\wedge \tau_n}\in U\,\,\forall t\}$. Hence Lemmas \ref{lem:SDE_Lip} and \ref{lem:cpt_2} and formula \eqref{eq:diff_Z} imply
\begin{align*}
&\lim_m\sup_{k\ge m}\sup_{t\in [0,T_n]} \mathbb{P}\{\|\delta X^{m}_t -\delta \tilde X^{m}_t\|_{\hat{H}}\ge \epsilon \}\\
&\le \limsup_m \mathbb{P}\{\tau^{x_m}\le \tau_n\} + \limsup_k \mathbb{P}\{\tau^{\tilde x_k}\le \tau_n\}\\
&\quad +\limsup_m \mathbb{P}\{\tau^{x_m}> \tau_n,\,X^{x_m}_{t\wedge \tau_n}\notin K \text{ for some }t\}\\
&\quad +\limsup_k \mathbb{P}\{\tau^{\tilde x_k}> \tau_n,\,X^{\tilde x_k}_{t\wedge \tau_n}\notin K \text{ for some }t\}\\
&\quad +\lim_m\sup_{k\ge m} \sup_{t\in[0,T_n]}\mathbb{P}\{\|\delta Z^{m}_t- \delta \tilde Z^{m}_t\|_{\hat{H}}\ge \epsilon\}\\
&\le 2\epsilon.
\end{align*}
By formula \eqref{eq:conv_probab}, we get
\begin{align*}
\lim_m\sup_{k\ge m}\sup_{t\in [0,T_n]} d_{L^0}\left(\delta X^{m}_t, \delta \tilde X^{m}_t\right)\le 3\epsilon.
\end{align*}
Since $\epsilon>0$ was arbitrary, we obtain \eqref{eq:incr_Cauchy}. The proof of the strong Gateaux-differentiability is complete.

It remains to show the existence of a time-continuous version of the Gateaux-derivative. Again we fix $h \in \hat{H}$, $\epsilon>0$ and the compact set $K$ and its open neighborhood $U$ defined above. We note that, as a consequence of Lemma \ref{lem:SDE_Lip}, there exists a subsequence $(m_k)_k$ such that
\begin{align*}
X^{x_{m_k}}_{\cdot\wedge \tau^n}1_{\tau^{x_{m_k}}>\tau_n}\to X^{x_0}_{\cdot\wedge \tau^n}\text{ in }C([0,T_n],\hat{H})\quad \text{a.s.}
\end{align*}
Therefore, since $U$ is open, we have
\begin{align*}
\liminf_k\{X^{x_{m_k}}_{t\wedge \tau_n}1_{\tau^{x_{m_k}}>\tau_n} \in U\,\forall t\}\supseteq \{X^{x_0}_{t\wedge \tau_n} \in U\,\forall t\}.
\end{align*}
Moreover, again by Lemma \ref{lem:SDE_Lip} and Borel-Cantelli lemma, up to taking a subsequence of $(m_k)_k$, we have
\begin{align*}
\mathbb{P}(\liminf_k\{\tau^{x_{m_k}}> \tau_n\})=1.
\end{align*}
Recall that $X^x_{\cdot\wedge \tau_n}$ and $Z^x_{\cdot\wedge\tau_n}$ coincide  a.s. on $\{\tau^x>\tau_n,\,X^x_{t\wedge\tau_n}\in U\,\forall t\}$. Hence we get
\begin{align}
\begin{split}\label{eq:Z_X}
&\mathbb{P}(\liminf_k\{\delta X^{m_k}_{\cdot\wedge \tau_n} = \delta Z^{m_k}_{\cdot\wedge \tau_n}\})\\
&\ge \mathbb{P}(\liminf_k\{X^{x_{m_k}}_{\cdot\wedge \tau_n}1_{\tau^{x_{m_k}}>\tau_n} = Z^{x_{m_k}}_{\cdot\wedge \tau_n}\} \cap \{X^{x_0}_{\cdot\wedge \tau_n} = Z^{x_0}_{\cdot\wedge \tau_n}\})\\
&\ge \mathbb{P}(\liminf_k\{\tau^{x_{m_k}}>\tau_n,\,X^{x_{m_k}}_{t\wedge \tau_n} \in U\,\forall t\} \cap \{X^{x_0}_{t\wedge \tau_n} \in U\,\forall t\})\\
&= \mathbb{P}(\liminf_k\{\tau^{x_{m_k}}>\tau_n\} \cap \liminf_k\{X^{x_{m_k}}_{t\wedge \tau_n}1_{\tau^{x_{m_k}}>\tau_n} \in U\,\forall t\} \cap \{X^{x_0}_{t\wedge \tau_n} \in U\,\forall t\})\\
&= \mathbb{P}\{X^{x_0}_{t\wedge \tau_n} \in U\,\forall t\} \ge 1-\epsilon.
\end{split}
\end{align}
On the other side, by Corollary \ref{cor:strong_Gateaux_stopped} and the convergence in $C([0,T_n],L^0(\Omega,\hat{H}))$ of $\delta X^{m}$, up to taking a subsequence of $(m_k)_k$, for every $t \in [0,T_n]$ we have:
\begin{align*}
&\delta X^{m_k}_{t\wedge \tau_n}\to \frac{d}{dh}X^{x_0}_{t\wedge \tau_n}\quad \text{a.s.},\\
&\delta Z^{m_k}_{t\wedge \tau_n}\to \frac{d}{dh}Z^{x_0}_{t\wedge \tau_n}\quad \text{a.s.}.
\end{align*}
Therefore, for every $t$, $\frac{d}{dh}X^{x_0}_{t\wedge \tau_n}$ and $\frac{d}{dh}Z^{x_0}_{t\wedge \tau_n}$ coincide  a.s. on the (time-independent) set
\begin{align*}
B_\epsilon:=\liminf_k\{\delta X^{m_k}_{\cdot\wedge \tau_n} = \delta Z^{m_k}_{\cdot\wedge \tau_n}\}.
\end{align*}
In particular, on the set $B_\epsilon$, $\frac{d}{dh}X^{x_0}_{t\wedge \tau_n}$ has a time-continuous version. By \eqref{eq:Z_X}, the set $B_\epsilon$ has probability $\ge 1-\epsilon$. By arbitrariness of $\epsilon$, $\frac{d}{dh}X^{x_0}_{t\wedge \tau_n}$ has a time-continuous version on $\Omega$. The proof is complete.
\end{proof}

\section{Local well-posedness and no-loss-no-gain theorem for stochastic Euler equations}\label{sec:noloss_nogain}

The main result of the present section will be to establish a stochastic version of the no-loss-no-gain result. For concreteness, we focus on the Euler equations as an example, but the method generalises to other equations amenable to the Ebin-Marsden approach with obvious modifications. 

\subsection*{Local well-posedness of maximal solutions for the Lagrangian Euler equations}
\addcontentsline{toc}{subsection}{Local well-posedness of maximal solutions for the Lagrangian Euler equations}

We are given a filtered probability space $(\Omega,\mathcal{A},\mathbb{F},\mathbb{P})$, where $\mathbb{F}=(\mathcal{F}_t)_t$ satisfies the standard assumption (see section \ref{sec:maximal_existence}). We consider the stochastic Euler equation on $K$, in Lagrangian form on the infinite-dimensional manifold  $T\Diff^s_\mu(K)$:
\begin{equation}
\begin{aligned}
\label{eq:stochEuler_Lagr}
&d\eta_t = B(\eta_t)dt +\Sigma(\eta_t)\bullet dW_t,\\
&\eta_0= \etazero, 
\end{aligned}
\end{equation}
with a given $\etazero \in T\Diff^s_\mu(K)$. 
In the following, we write $\text{vl} \colon TK \rightarrow T^2K$ for the vertical lift on $TK$ and $\Pi$ for the Leray projection on divergence-free vector fields. The Ebin-Marsden drift $B$ is given by 
\begin{align}
B\colon T\Diff^s_\mu(K)&\to T^2\Diff^s_\mu(K), \notag\\ \label{Spray_form} T_\gamma \Diff^s_{\mu}(K) \ni X &\mapsto T(\eta \circ \gamma^{-1})\circ \eta - \text{vl}(\eta, \Pi [\nabla_{\eta \circ \gamma^{-1}}\eta \circ \gamma^{-1}]),
\end{align}
i.e. it is the geodesic spray associated with the right invariant $L^2$ metric on $\Diff^s_\mu(K)$. By \cite[Theorem 11.1]{EM70} it is of $C^\infty$-class. The diffusion coefficient is defined by    
\begin{align}
\Sigma\colon T\Diff^s_\mu(K) \to \mathscr{L}(\VFsmu[s],T^2\Diff^s_\mu(K))   \\
\label{eqn-Sigma-02}
\Sigma(\eta)w = \text{vl}(\eta,w\circ \pi_K(\eta)),\quad \eta\in T\Diff^s_\mu(K),\,w\in \VFsmu[s],
\end{align}
 where $\pi_K\colon TK\to K$ is the bundle projection. In general $\Sigma$ will only be continuous. However, as is established below in Lemma \ref{lem:drift_is_C11}, $\Sigma$ gains differentiability and local Lipschitz properties if restricted to a subspace of more regular vector fields, We will thus require (see Hypothesis \ref{hp:noise} below) that the noise takes values in a Hilbert space of more regular vector fields.

Using the embedding $i\colon K \rightarrow \R^m$ of $K$, we obtain the following identity in $H^s (K,\R^{2m})$ (cf.\ \cite[2.8]{MMS19})
\begin{align}
DTi \circ \Sigma(\eta)w = (0_{\R^{m}}, D i \circ w \circ \pi_K \circ \eta). \label{eq:Sigma_vertical}
\end{align}
Note that for every $r \geq s$ we shall also denote the canonical restriction of $\Sigma$ by
\begin{equation}
\label{eqn-Sigma-03}
\Sigma\colon T\Diff^r_\mu(K) \to \mathscr{L}(\VFsmu[s],T^2\Diff^s_\mu(K)).
\end{equation}
Finally, we make the following hypothesis about the Wiener process $W$. 
\begin{hh}\label{hp:noise}
The process $W$ is an $\VFsmu[r]$-valued  $\mathbb{F}$-Wiener process with $r \geq s$. The RKHS of the law of $W_1$ will be denoted by $\rK$.
\end{hh}
In most of our results, we will assume a stronger version of Hypothesis \ref{hp:noise} by requiring $r$ to be strictly larger. The reason for this is that spacial regularity can be traded for differentiability of the diffusion coefficient as the following technical result shows (see also \cite[Chapter VIII Section 1]{Elw82} for a similar result):

\begin{lem}\label{lem:drift_is_C11}
Let $s > d/2+1$ and assume that the noise satisfies Hypothesis \ref{hp:noise} with $r = s+k$. Then the diffusion coefficient $\Sigma$ in \eqref{eq:stochEuler_Lagr} is of class $C^{k-1,1}_{loc}$.
\end{lem}

\begin{proof}
The regularity of $\Sigma$ can be established localising $\Sigma$ in manifold charts and applying \cite[Proposition 2.11]{MMS19}. Unfortunately, we were lazy in the first paper, and stated the result only for a $\psi$ centered at the identity. So for the cited result, we needed a chart such that $\psi (\id) = 0 \in \VFsmu$ and $T_{\id} \psi = \id_{\VFsmu}$. The latter identity uses (cf.\ \eqref{tangent_ident} and \ref{setup:diffgp})
$T_\gamma \Diff_\mu^s (K) = \{X \circ \gamma \mid X \in \VFsmu\} =TR_\gamma (\VFsmu )$,
where $R_\gamma \colon \Diff_\mu^s (K) \rightarrow \Diff_\mu^s (K)$ is the right shift by composition with $\gamma$. Both requirements are inessential and we can replace $\psi$ with the chart $\psi_\gamma \coloneq \psi \circ R_{\gamma^{-1}} \colon R_{\gamma }(O) \rightarrow \VFsmu$, Then  $\psi_\gamma(\gamma)=0$ and $T_\phi \psi_\gamma =T_{\phi \circ \gamma^{-1}}\psi \circ T_{\phi} R_{\gamma^{-1}}$. We compute now in a localisation for the chart $\psi_\gamma$ with respect to $\gamma \coloneq \pi_K \circ \eta$ for some fixed (but arbitrary) $\eta \in T\Diff_\mu^s (K)$.  Inspecting the arguments leading to the proof of \cite[Proposition 2.11]{MMS19}, we see that they carry over. In particular \cite[2.8]{MMS19} yields verbatim for $V \in \VFsmu[s]$ and $\psi_\gamma$ as above:
\begin{align*}
T^2\psi_\gamma \circ (\Sigma \circ T\psi_\gamma^{-1} (\theta, \xi)V) &= (\theta,\xi,0,T_{\psi_\gamma^{-1}(\theta)} \psi_\gamma (\text{comp}(V, \psi^{-1}_\gamma(\theta))) \\ &=(\theta,\xi,0,T_{\psi^{-1}(\theta)} \psi (\text{comp}(V,\psi^{-1}(\theta)))) 
\end{align*}
This is exactly the formula used in the proof of \cite[Proposition 2.11]{MMS19} to establish the regularity of $\Sigma$ in the chart $\psi$. Hence the argument works also for $\psi_\gamma$ and thus yields the claimed regularity for $\Sigma$ on all of $\Diff_\mu^s (K)$.
\end{proof}

We shall now establish the existence and uniqueness of a maximal solution for the Lagrangian formulation \eqref{eq:stochEuler_Lagr} of the stochastic Euler equation.  This result is an extension of \cite[Theorem 3.5]{MMS19} as we here also show maximality of the solution.

\begin{thm}\label{thm:wellposed_Lagr}
Assume that  $s>d/2+1$ and that Hypothesis \ref{hp:noise} holds with $r\ge s+2$. Then, for every $\etazero$ belonging to the connected component of the identity in $T\Diff^s_\mu(K)$, existence of a local maximal solution and local uniqueness hold for the stochastic Euler equation in the Lagrangian form \eqref{eq:stochEuler_Lagr}.
\end{thm}

\begin{proof}
Recall from \cite[B.11]{MMS19} that the connected component of the identity in $T\Diff_\mu^s (K)$ is a separable Hilbert manifold. By the Ebin-Marsden theory, the drift term $B$ is  of $C^\infty$-class so that in suitable charts, it is locally Lipschitz. Due to Lemma \ref{lem:drift_is_C11} the diffusion coefficient $\Sigma$ is locally of $C^{1,1}$ class. Hence the conclusion follows from Theorem \ref{thm:maximal} applied to the SDE \eqref{eq:stochEuler_Lagr}.
\end{proof}


\subsection*{The no-loss-no-gain theorem for stochastic equations}
\addcontentsline{toc}{subsection}{The no-loss-no-gain theorem for stochastic equations}
We are now able to prove the main result in this section: the no-loss-no-gain theorem for the stochastic Euler equation in Lagrangian form \eqref{eq:stochEuler_Lagr}.

\begin{thm}\label{thm:noloss_nogain}
Take $s>d/2+1$ and let Hypothesis \ref{hp:noise} be satisfied with $r\ge s+3$; take $\etazero$ in the connected component of the identity in $T\Diff^s_\mu(K)$ and let $\eta=(\eta_t)_{t\in [0,\tau)}$ be the corresponding maximal solution to \eqref{eq:stochEuler_Lagr}. If $\etazero$ belongs to $T\Diff^{s+1}_\mu(K)$, then,  a.s., the map $t\mapsto \eta_t$, $t\in [0,\tau)$, is well-defined and continuous with values in $T\Diff^{s+1}_\mu(K)$ and it coincides with the (unique) maximal solution to \eqref{eq:stochEuler_Lagr} in $T\Diff^{s+1}_\mu(K)$.
\end{thm}

Note that the statement of Theorem \ref{thm:noloss_nogain} holds for arbitrary components of the manifold $T\Diff^{s}_\mu (K)$ (in the proof we only need that the initial condition does not randomly jump components). Since we only need the statement later for the unit component we chose to formulate the theorem only for this case. 
The idea of the proof is as follows:
\begin{itemize}
\item First, we use that a function $\varphi \in H^s(K)$ is in $H^{s+1}$ if and only if $T\varphi\circ X$ is in $H^s$ for a finite number of smooth divergence free vector fields $X$ (these vector fields play the role of local coordinates), Lemma \ref{lem: testing}.
\item Calling $\eta^{\etazero}$ the solution of the SDE starting at $\etazero$, the right-invariance of the coefficients of the SDE in (the tangent of) the half-Lie group $\Diff^s_\mu(K)$, implies that 
\begin{equation}\label{eqn-nvariance}
 \eta_t^{\etazero}\circ \phi = \eta_t^{\etazero\circ \phi}, \;\; \mbox{ for every $\phi \in \Diff^s_\mu(K)$}.
\end{equation}
\item Using the above relation, we write the finite-dimensional derivative $T\eta_t^{\etazero}\circ X$ in terms of the infinite-dimensional Gateaux-derivative $\frac{d}{d(\etazero\circ X)}\eta_t^{\etazero}$. The latter, by the differentiability with respect to the initial condition, Theorem \ref{thm-Strat_diff_manifold}, there exists in $H^s$, therefore also $T\eta_t^{\etazero}\circ X$ is in $H^s$, and we conclude that $\eta_t$ is in $H^{s+1}$.
\end{itemize}

For the first step, we need the following:
\begin{lem}\label{VF:aux1}
 Let $(U,\varphi)$ is a manifold chart around $x\in K$ and fix an open $x$-neighborhood $V \subseteq \overline{V} \subseteq U$. Then there exist divergence free vector fields $X_{\varphi,V,i} \in \mathcal, i\in \{1,\ldots,d\}$ such that the vector fields act locally on functions as partial derivatives, i.e.~for every $u \in V$ and $f \in C^1 (K,\R^m)$ we have
 $$X_{\varphi,i}. f (u) = Df X_{\varphi, V, i}(u) = \partial^i (f\circ \varphi^{-1})(\varphi(u)), i\in \{1,\ldots, d\}.$$
\end{lem}

\begin{proof}
 Recall from the proof of \cite[p.21, Claim A]{BaM18} that for any fixed open set $V$ (with positive distance between the boundaries of $V$ and $U$), there exists a smooth divergence free vector field $X_{V,j} \in \mathfrak{X}_\mu (K)$ such that $X_{V,j}$ acts on functions on $V$ as the $j$th-partial derivative, i.e.\ in the chart $\varphi$ we have for every function $f \in C^1(K,\R^m)$:
	\begin{equation}
		Df \circ X_{V,j} (y) = \partial^j (f \circ \varphi^{-1}) \varphi (y) \text{ for all  }y \in V.\label{eq:test}
	\end{equation}
	If $\text{dim } K =d$ we simply repeat the construction for every $1\leq j \leq d$.
\end{proof}

The first step is given morally by the following:
\begin{lem}\label{lem: testing}
	Let $\mathscr{U}_I = \{(U_i,\kappa_i)\}_{i \in I}$ be a finite atlas for the compact manifold $K$ and $\{V_i \subseteq \overline{V}_i \subseteq U_i\}_{ i \in I}$ be a family of open sets such that $K = \bigcup_{i\in I} V_i$. Then the finite subset $\mathscr{T} =\{X_{V_i,j} \mid i\in I, 1\leq j\leq d\} \subseteq \mathfrak{X}_\mu (K)$ of divergence free vector fields constructed via \eqref{VF:aux1} has the following property:  If for $f \in H^s (K,N)$, if $Tf \circ X \colon K \to TN$ is an $H^s$-map for all $X \in \mathscr{T}$, then $f \in H^{s+1}(K,N)$.
    Further, if $f \in T^\ell\Diff^s (K) \subseteq H^s (K,T^\ell K)$ for some $\ell \in \N_0$ satisfies the condition, then $f \in T^\ell \Diff^{s+1}(K)$.
\end{lem}

\begin{proof}
 We construct the family $\mathscr{T}$ by applying \eqref{VF:aux1} to every pair $(U_i,\kappa_i),V_i)$. Since $I$ is finite, the set $\mathscr{T}$ is finite. Since $K = \bigcup_{i\in I} V_i$, the construction yields a set of vector fields such that for every $x \in K$ and $1\leq i \leq d$ there is a smooth vector field $X \in \mathscr{T}$ such that $X$ acts as the $i$th partial derivative on functions, i.e. \eqref{eq:test} holds (with respect to the charts $\kappa_i$ on $V_i$).
 Assume now that $f \in H^s(K,N)$ such that $Tf \circ X$ is an $H^s$-map for all $X \in \mathscr{T}$. Due to \eqref{eq:test} we see that (locally) every partial derivative of $f$ is of class $H^s$. This shows that $f \in H^{s+1}(K,N)$.
	
 To prove the final assertion, let us assume first that $f \in  \Diff^s (K)$. Assuming that the above condition holds for $f$, we deduce as above that $f \in H^{s+1} (K,K)$. Recall from \cite[Lemma 3.8]{IKT13} that this already entails that $f^{-1} \in H^{s+1}(K,K)$. We deduce that also $f \in \Diff^{s+1}(K)$.
 Now we continue with $\ell \in \N$ and $f \in T^\ell \Diff^s (K)$. Applying \eqref{tangent_ident} iteratively, we identify $T^\ell \Diff^s (K)$ with an open subset of $H^s (K,T^\ell K)$. Arguing as before, we see that then $f$ is an $H^{s+1}$ map. Projecting $f $ onto its base point in $T^{\ell}\Diff^s (K)$, the above argument shows that this basepoint is again a $H^{s+1}$-diffeomorphism. We conclude that $f \in T^\ell  \Diff^{s+1} (K)$.
\end{proof}

\begin{rem}
 Lemma \ref{lem: testing} is a version of \cite[Lemma 12.2]{EM70} or \cite[p.21]{BaM18}. However, these results have stronger prerequesits. They asserts that $\varphi$ is of class $H^{s+1}$ if the composition $T\varphi \circ X$ is an $H^s$-map for all $X \in \VFsmu$.
 Note that we obtain the stronger condition by testing against the family $\mathscr{T}$ constructed in the proof of Lemma \ref{lem: testing}.
 An essential ingredient for the proof of Lemma \ref{lem: testing} was compactness of $K$. If the manifold $K$ would only be paracompact, a similar statement holds if we instead ask the family $\mathscr{T}$ to be only countable.
\end{rem}

To conclude the first step, we need a final technical lemma dealing with continuity of mappings into spaces of Sobolev functions. It is convenient to define the notion of a nice cover of a manifold (similar to the notions of covers introduced in \cite{IKT13}). 

\begin{defn}
Let $K$ be a compact manifold (possibly with smooth boundary). An open cover $(U_i,\kappa_i)_{i \in I}$ of $K$, is called a \emph{cover of bounded type}, if for any $i,j \in I$ with $U_i \cap U_j \neq \emptyset$, the change of charts $\kappa_j \circ \kappa_i^{-1}$ extends to a mapping in $C^\infty_{\text{glob}} (\overline{\kappa_i(U_i \cap U_j)},\R^d)$.

Further, we say that a cover $(U_i,\kappa_i)_{i \in I}$ of bounded type is a \emph{nice cover} if \begin{itemize}
\item $I$ is finite and for every $i \in I$, $\kappa_i (U_i)$ is bounded with Lipschitz boundary
\item for every pair $i,j \in I$ with $U_i \cap U_j \neq \emptyset$, the boundary of $\kappa_i (U_i \cap U_j)$ is piecewise $C^\infty$-smooth, i.e.~it is given by a finite (possibly empty) union of transversally intersecting $C^\infty$-embedded hypersurfaces, so in particular $\kappa_i (U_i\cap U_j)$ has a Lipschitz boundary.
\end{itemize}
\end{defn}
Nice covers are used to construct fine covers for mappings as in \cite[Definition 3.2]{IKT13}). Moreover, if $K$ has no boundary, they can always be constructed as in \cite[Lemma 3.1]{IKT13}. As noticed in \cite[Appendix B]{MMS19}, it is straight forward to adapt the construction of \cite[Lemma 3.1]{IKT13} to obtain nice covers for the case when $K$ has non-empty boundary.

\begin{lem}\label{lem:testing_time_cont}
 Let $P$ be a topological space and $f \colon P \to H^s (K,\R^m)$ be a continuous map. Assume that $\mathscr{U}_I \coloneq \{(U_i,\kappa_i) \mid i \in I\}$ is a finite  manifold atlas, $V_i \subseteq \overline{V}_i \subseteq U_i$ is open such that the $(V_i,\kappa_i)_{i \in I}$ form a nice cover.
 
 Let now $\mathscr{T} = \{X_{V_i,j}, i\in I, 1\leq j\leq d\}$ be the finite family of divergence free vector fields constructed via Lemma \ref{lem: testing} for $\mathscr{U}_I$ (and the $V_i$). If $f(E) \subseteq H^{s+1}(K,\R^m)$ and for every $i \in I, 1\leq j \leq d$ the map $F_{ij} \colon P \to H^s (K,\R^{m}), p \mapsto D(f(p))\circ X_{V_i,j}$ is continuous. Then $f \colon P\to H^{s+1}(K,\R^m)$ is continuous.
\end{lem}

\begin{proof}
We will first treat the case that $K$ has \textbf{no boundary}. Recall first from \cite[Proposition 3.7 and Proposition 3.5]{IKT13} that a base of the topology of $H^{s+1}(K,\R^m)$ is given by certain sets $\mathscr{O}^{s+1}(\mathscr{U}_I,\mathscr{V}_I)$, where $(\mathscr{U}_i , \mathscr{V}_i)$ runs through all fine covers of the pair $(K,\R^m)$. Since $\R^m$ is a vector space, the euclidean ball $B_R(0) \subseteq \R^m$ of radius $R>0$ around $0$ is bounded with smooth boundary. Hence our assumptions 1.-3. on the $V_i$ show that the pair $\mathscr{U}_I = \{(V_i,\kappa_i) , i \in I\}, \mathscr{V}_I = \{(B_R (0), \id_{B_R^{\R^m} (0)}), i\in I\}$ forms a fine cover for every $R>0$ (see \cite[Definition 3.2]{IKT13} for the definition.\footnote{Indeed $\mathscr{V}_I$ is completely superfluous as the target manifold is the vector space $\R^m$. We chose to provide it anyway as they make the results of \cite{IKT13} verbatim applicable. However, the reader can convince herself that in all of the arguments the family $\mathscr{V}_I$ could just be replaced by $(\R^m,\id_{\R^m})$.} Now a function $g \in H^{s+1}(K,\R^m)$ is contained in $\mathscr{O}^{s+1}(\mathscr{U}_I,\mathscr{V}_I)$ if and only if for every $i \in I$ $f(\overline{V_i}) \subseteq B_R(0)$ holds. Thus every $g \in H^s (K,\R^m)$ is contained in $\Omega_R \coloneq\mathscr{O}^{s+1}(\mathscr{U}_I,\{(B_R (0), \id_{B_R^{\R^m} (0)}), i\in I\})$ for some $R>0$ (as $g(K)$ is compact whence bounded). In view of \cite[Proposition 3.5]{IKT13} it suffices to check that the mappings $f|_{f^{-1}(\Omega_R)}^{\Omega_R}$ are continuous.

To ease notation let us assume that $f$ takes its image in $\Omega_R$ for some $R>0$. This also entails that $D(f(p)) \in T\Omega_R$ for every $p\in P$. Combining \ref{setup:top_struct_sect} and \cite[Proposition 3.3]{IKT13}, the topology of $\Omega_R$ is initial with respect to the mapping
\begin{align}\label{loc.descr}
\Phi \colon \Omega_R \to \prod_{i\in I} H^{s+1}(\kappa_i (V_i),\R^m),\quad h \mapsto (h\circ \kappa^{-1}_i|_{\kappa_i (V_i)})_i,
\end{align}
where the spaces on the right hand side are the Hilbert spaces from Definition \ref{defn:Sobolev_funII}. So froom \eqref{loc.descr} we write $\Phi \circ f=(f_i)_i$ and thus it suffices to check that the components $f_i$ are continuous. We can test this with help of  the $H^{s+1}$-norm associated to the $H^{s+1}$-inner product \eqref{Hell_product} 
\begin{align}\label{Sobolevnormpart}
\lVert f_i(p)\rVert_{H^{s+1}} = \left(\sum_{|\alpha| \leq s+1} \int_{\varphi_i (V_i)} \langle D^\alpha f_i(p), D^\alpha f_i(p)\rangle \mathrm{d}x\right)^{1/2}, \quad p\in P
\end{align}
where $D^\alpha$ denotes the distributional derivatives. However, we already know that the $f_i$ are continuous as maps $P \to H^s(\kappa_i (V_i),\R^m)$. In particular, all summands in \eqref{Sobolevnormpart} with $|\alpha|<s+1$ are continuous as functions of $p$. Consider now $D^\alpha f_i$ with $|\alpha|=s+1$. We rewrite $D^\alpha = D^\beta D^j$ for some $j \in \{1,\ldots,d\}$ and $\beta$ a multiindex with $|\beta|=s$. Now by choice of the charts $\varphi_i$, \eqref{eq:test} yields on $\varphi_i(V_i)$ for every $p \in P$ the identity
$$D^j f_i(p) = D^j (f(p) \circ \kappa_i^{-1}|_{\kappa_i (V_i)})= D (f(p))\circ X_{V_i,j} \circ \kappa^{-1}_i = F_{i,j}(p)\circ \kappa_i^{-1}.$$
Thus $D^j f_i(p)$ coincides with the $i$th component of $\Phi\circ F_{i,j}(p)$ and by assumption this mapping is continuous in $p$ as a map to $H^s (K,\R^m)$. As $i,j$ were arbitrary, we deduce that also the summands with $|\alpha|=s+1$ in \eqref{Sobolevnormpart} depend continuously on $p$ and this concludes the proof if $K$ has no boundary.

Assume now that $K$ has non-empty smooth boundary $\partial K$, whence \cite{IKT13} is not directly applicable. However, the proof carries over exactly as presented in case without boundary since the Sobolev manifolds of mappings $H^s (K,\R^m)$ are constructed analogously. Indeed the proof used only the existence of nice covers (actually of fine covers but for the vector space valued case, nice covers are sufficient as pointed out above). As a consequence the statement also holds if $K$ is a manifold with smooth boundary.
\end{proof}

In the following, for given $\etazero \in T\Diff^s_\mu(K)$, we call $\eta=\eta^{\etazero}$ the solution to the stochastic Euler equation in Lagrangian form \eqref{eq:stochEuler_Lagr} with $\etazero$ as initial condition; we also call $\tau^{\etazero}$ the maximal existence time of the solution. The second step is given by the following:
\begin{lem}\label{lem:invariance_SDE} 
For every $\etazero \in T\Diff^s_\mu(K)$ and $\phi \in \Diff^s_\mu(K)$, we have  a.s.: $\tau^{\etazero}=\tau^{\etazero\circ \phi}$ and
\begin{align*}
\eta_t^{\etazero} \circ \phi = \eta_t^{\etazero\circ \phi},\quad \forall t\in [0,\tau^{\etazero}).
\end{align*}
\end{lem}

\begin{proof}
We show first that $\eta_t^{\etazero} \circ \phi$ solves, on $[0,\tau^{\etazero})$, the SDE \eqref{eq:stochEuler_Lagr} (with initial condition $\etazero\circ \phi$). For this, we consider the $C^2$ transformation (cf.\ \ref{setup:diffgp}) 
\begin{equation}\label{eqn-varphi}
TR_\phi \colon T\Diff^s_\mu(K) \ni V \mapsto V\circ\phi \in T\Diff^s_\mu(K),
\end{equation}
which satisfies
$T(TR_\phi)(Z) = Z\circ \phi,\quad Z\in T^2\Diff^s_\mu(K)
$ (using the identification of $T^2 \Diff^s_\mu (K)$ with a subset of $H^s (K,T^2K)$). 
Hence, applying Lemma \ref{lem-Strat_diff_manifold} to the map $\varphi$ and process $\eta^{\etazero}$, we infer that the process , on $[0,\tau^{\etazero})$,
\begin{align*}
d[\eta^{\etazero}\circ \phi] = B(\eta^{\etazero}_t)\circ \phi dt +\Sigma(\eta^{\etazero}_t)\circ \phi \bullet dW.
\end{align*}
Now the geodesic spray $B$ is constructed with respect to the right-invariant $L^2$-metric, whence \cite[Theorem 11.1]{EM70} shows that the spray satisfies
\begin{align*}
B(V)\circ \phi = B(V\circ \phi),\quad V\in T\Diff^s_\mu(K)
\end{align*}
and also the diffusion coefficient $\Sigma$ satisfies
\begin{align}
[\Sigma(V)\circ \phi]w &= \text{vl}_{T\Diff^s_\mu(K)}(V,w\circ \pi_K(V))\circ \phi = \Sigma(V\circ \phi)w,
\\ &\mbox{ for }
\quad V\in T\Diff^s_\mu(K),w\in \VFsmu[s].
\end{align}
Hence $\eta^{\etazero}\circ\phi$ satisfies \eqref{eq:stochEuler_Lagr}. By uniqueness of \eqref{eq:stochEuler_Lagr} and maximality of $\eta^{\etazero\circ \phi}$, we have  a.s.: $\tau^{\etazero}\le \tau^{\etazero\circ \phi}$ and $\eta^{\etazero}\circ \phi$ coincides with $\eta^{\etazero\circ \phi}$ on $[0,\tau^{\etazero})$. Replacing $\etazero$ and $\phi$ with resp. $\etazero\circ \phi$ and $\phi^{-1}$, we get the reverse inequality $\tau^{\etazero\circ \phi}\le \tau^{\etazero}$. The proof is complete.
\end{proof}

Now we show the third step and conclude the proof of Theorem \ref{thm:noloss_nogain}. We take an integer $m$ such that $K$ is embedded in $\R^m$. By Corollary \ref{cor:embedding_TDiff}, $T\Diff^s_\mu(K)$ is embedded into $H^s(K,\R^{2m})$ as split submanifold.  More precisely, the connected component of $T\Diff^s_\mu(K)$ containing the identity is embedded as split submanifold; anyway, since we will always work on this connected component, with some abuse of notation we still write $T\Diff^s_\mu(K)$ for the connected component. Note that $H^s(K,\R^{2m})$ is separable by \cite[Lemma B.4]{MMS19}. Thus, we can apply the results of Section \ref{sec-Strat_diff_manifold}. We also recall from Lemma \ref{lem:eval} that fir $s>d/2$, the evaluation map $H^s(K,\R^{2m})\to \R^{2m}$, $v\mapsto v(k)$, for given $k\in K$, is continuous.

\begin{proof}[Proof of Theorem \ref{thm:noloss_nogain}]
We take an announcing sequence $(\tau_n=\tau_n^{\etazero})_n$ for $\tau=\tau^{\etazero}$, without loss of generality we can assume that each $\tau_n$ is bounded. We take a smooth divergence-free vector field $X \in \VFsmu$ and a $C^1$ curve $\phi:(-a,a)\to \Diff^s_\mu(K)$ with $\dot\phi(0)=X$. By Lemma \ref{lem:invariance_SDE}, for every $s \in (-a,a)$ we have  a.s.,
\begin{align}
\eta_t^{\etazero} \circ \phi_s = \eta_t^{\etazero\circ \phi_s},\quad \forall t\in [0,\tau^{\etazero}).\label{eq:SDE_invariance_s}
\end{align}
Concerning the left-hand side of \eqref{eq:SDE_invariance_s}, by the Sobolev embedding we have  a.s.: for every $k \in K$, the $\R^{2m}$-valued map $s\mapsto \eta_t^{\etazero} \circ \phi_s(k)$ is $C^1$ and its derivative at $s=0$ is
\begin{align*}
\left.\frac{d}{ds}\right|_{s=0} \eta_t^{\etazero} \circ \phi_s(k) = D\eta_t^{\etazero} \circ X(k), \quad \forall t\in [0,\tau^{\etazero}),\quad \forall k\in K.
\end{align*}
In particular, for every $n$, for every $t\ge 0$, for every sequence $(s_N)_N$ with $s_N\to0$, we have  a.s.:
\begin{align}
\lim_N \frac{\eta_{t\wedge \tau_n}^{\etazero} \circ \phi_{s_N}(k) - \eta_{t\wedge \tau_n}^{\etazero}(k)}{s_N} = D\eta_{t\wedge \tau_n}^{\etazero} \circ X(k),\quad \forall k\in K.\label{eq:diff_1}
\end{align}
Concerning the right-hand side of \eqref{eq:SDE_invariance_s}, in view of Theorem \ref{thm-Strat_diff_manifold}, we embed the stochastic Euler equation \ref{eq:stochEuler_Lagr} into $H^s(K,\R^{2m})$, getting
\begin{align}
d\eta = \tilde B(\eta)dt + \tilde\Sigma (\eta)\, dW,\label{eq:stochEuler_Lagr_emb}
\end{align}
where, for $x \in T\Diff^s_\mu(K)$,
\begin{align*}
&\newtilde{B}(x)=B(x)+\frac12\tr_{\rK}[D\Sigma\circ \Sigma(x) ],\\
&\newtilde{\Sigma}(x) = \Sigma(x).
\end{align*}
Now $B$ is smooth by \cite[Theorem 11.1]{EM70} and we can argue as in the proof of Theorem \ref{thm:wellposed_Lagr}: By hypothesis \ref{hp:noise} with $r\ge s+3$ and Lemma \ref{lem:drift_is_C11}, $\Sigma$ is $C^{2,1}_{loc}$. In particular the coefficients $\newtilde{B}$ and $\newtilde{\Sigma}$ are $C^1 \in H^s(K,\R^{2m})$. Hence we can apply Theorem \ref{thm-Strat_diff_manifold} to the SDE \eqref{eq:stochEuler_Lagr_emb}: For $n \in \N$, the map 
 $x\mapsto \eta^x_{\cdot\wedge\tau_n}$ allows a time-continuous Gateaux derivative at the point $\etazero$ in the direction of $X$. So if $\frac{d}{dh}\phi(h,k) = X(\phi(h,k))$ we set:
 $$\frac{d}{dX} \eta_{\cdot\wedge \tau_n}^{\etazero} \coloneq \left.\frac{d}{dh}\right|_{h=0} \eta^{\etazero \circ \phi_h}_{\cdot \wedge \tau_n} \in C([0,T_n],L^0(\Omega,H^s(K,\R^{2m}))) $$
 
 By Remark \ref{rem:conv_sup_P} there exists a sequence $(s_N)_N$ with $s_N\to 0$ such that, for every $t\ge0$, we have  a.s.:
\begin{align*}
\lim_{N \rightarrow \infty} \frac{\eta_{t\wedge \tau_n}^{\etazero\circ \phi_{s_N}} -\eta_{t\wedge \tau_n}^{\etazero}}{s_N} = \frac{d}{dX} \eta_{t\wedge \tau_n}^{\etazero},
\end{align*}
as limit in $H^s(K,\R^{2m})$, and in particular, evaluating at $k \in K$,
\begin{align}
\lim_N \frac{\eta_{t\wedge \tau_n}^{\etazero\circ \phi_{s_N}}(k) -\eta_{t\wedge \tau_n}^{\etazero}(k)}{s_N} = \frac{d}{dX} \eta_{t\wedge \tau_n}^{\etazero}(k),\quad \forall k\in K.\label{eq:diff_2}
\end{align}
Putting together \eqref{eq:SDE_invariance_s}, \eqref{eq:diff_1} and \eqref{eq:diff_2}, for every $t\ge 0$ we get:  a.s.,
\begin{align}
D\eta_{t\wedge \tau_n}^{\etazero} \circ X(k) =\frac{d}{dX} \eta_{t\wedge \tau_n}^{\etazero}(k),\quad \forall k\in K.\label{eq:T_eta}
\end{align}
Since $t\mapsto \eta^{\etazero}_t$ is continuous with values in $H^s(K,\R^{2m})$ (and $s>d/2+1$), the map
\begin{align*}
[0,T]\times K \ni (t,k)\mapsto D\eta^{\etazero}_{t\wedge \tau_n}(X(k)) \in \R^{2m}
\end{align*}
is continuous  a.s.. Similarly, since the map $t\mapsto \frac{d}{dX} \eta_{t\wedge \tau_n}^{\etazero}$ is continuous with values in $H^s(K,\R^{2m})$, the map
\begin{align*}
[0,T]\times K \ni (t,k)\mapsto \frac{d}{dX} \eta_{t\wedge \tau_n}^{\etazero}(k) \in \R^{2m}
\end{align*}
is continuous  a.s.. Then in \eqref{eq:T_eta} we can make the $\mathbb{P}$-exceptional subset of $\Omega$ independent of time, namely we have  a.s.
\begin{align*}
D\eta_{t\wedge \tau_n}^{\etazero} \circ X(k) =\frac{d}{dX} \eta_{t\wedge \tau_n}^{\etazero}(k),\quad \forall (t,k)\in [0,T]\times K.
\end{align*}
In particular,  a.s. we have
\begin{align}
t\mapsto D\eta_{t\wedge \tau_n}^{\etazero} \circ X \text{ is continuous with values in } H^s(K,\R^{2m}).\label{eq:deriv_cont}
\end{align}
We can make the $\mathbb{P}$-exceptional set for \eqref{eq:deriv_cont} (that is, the $\mathbb{P}$-exceptional set where \eqref{eq:deriv_cont} does not hold) independent of $X \in \mathscr{T}$, where $\mathscr{T}$ is the finite set from Lemma \ref{lem: testing}. Hence, by Lemma \ref{lem: testing} and Lemma \ref{lem:testing_time_cont} we have a.s.:
\begin{align*}
t\mapsto \eta_{t\wedge \tau_n}^{\etazero}
\end{align*}
is continuous with values in $H^{s+1}(K,\R^{2m})$  and so it is continuous with values in $T\Diff^{s+1}_\mu(K)$, for every $n$. By arbitrariness of $n$, we conclude that,  a.s., $t\mapsto \eta_t$ is continuous with values in $T\Diff^{s+1}_\mu(K)$ up to its maximal existence time $\tau^{\etazero}$.

It remains to show that $\eta$ is the maximal solution to the stochastic Lagrangian Euler equation \eqref{eq:stochEuler_Lagr} in $T\Diff^{s+1}_\mu(K)$. Recall first that the metric spray is smooth by \cite[Theorem 11.1]{EM70} and by Assumption \ref{hp:noise} with $r\ge s+3$, Lemma \ref{lem:drift_is_C11} yields 
\begin{align*}
&B\mid_{T\Diff^{s+1}_\mu(K)}:T\Diff^{s+1}_\mu(K) \to T^2\Diff^{s+1}_\mu(K)\text{ is }C^\infty_{loc},\\
&\Sigma\mid_{T\Diff^{s+1}_\mu(K)}:T\Diff^{s+1}_\mu(K) \to \mathscr{L}(\mathfrak{X}^{s+3}_\mu(K),T^2\Diff^{s+1}_\mu(K))\text{ is }C^{1,1}_{loc}.
\end{align*}
Therefore the coefficients of the embedded SDE \eqref{eq:stochEuler_Lagr_emb} satisfy
\begin{align*}
&\tilde B\mid_{T\Diff^{s+1}_\mu(K)}\colon T\Diff^{s+1}_\mu(K) \to H^{s+1}(K,\R^{2m})\text{ is }C^{0,1}_{loc},\\
&\tilde\Sigma\mid_{T\Diff^{s+1}_\mu(K)}\colon T\Diff^{s+1}_\mu(K) \to \mathscr{L}(\mathfrak{X}^{s+3}_\mu(K),H^{s+1}(K,R^{2m}))\text{ is }C^{1,1}_{loc}.
\end{align*}
Since $t\mapsto \eta_t$ is  a.s. continuous with values in $T\Diff^{s+1}_\mu(K)$, then
\begin{align*}
t\mapsto \int^t_0 \newtilde{B}(\eta_r)dr,\quad t\mapsto \int^t_0 \newtilde{\Sigma}(\eta_r)dW_r
\end{align*}
make sense and are  a.s. continuous with values in $H^{s+1}(K,\R^{2m})$. Therefore the SDE \eqref{eq:stochEuler_Lagr_emb}, namely
\begin{align*}
\eta_t = \etazero +\int^t_0 \newtilde{B}(\eta_r)dr +\int^t_0 \newtilde{\Sigma}(\eta_r)dW_r,\quad t\in [0,\tau^{\etazero}),
\end{align*}
which holds in $H^s(K,\R^{2m})$, must hold also in $H^{s+1}(K,\R^{2m})$. Since $T\Diff^{s+1}_\mu(K)$ is a split submanifold of $H^{s+1}(K,\R^{2m})$, by Lemma \ref{lem:SDE_embedded} $\eta$ satisfies (on $[0,\tau^{\etazero})$) the Lagrangian Euler equation \eqref{eq:stochEuler_Lagr} in $T\Diff^{s+1}_\mu(K)$. If $\zeta=(\zeta_t)_{t\in [0,\rho)}$ is another solution to \eqref{eq:stochEuler_Lagr} in $T\Diff^{s+1}_\mu(K)$, then $\zeta$ solves \eqref{eq:stochEuler_Lagr} also in $T\Diff^s_\mu(K)$ (since $T\Diff^s_\mu(K)$ is smoothly embedded into $T\Diff^{s+1}_\mu(K)$). Hence by maximality of $\eta \in T\Diff^s_\mu(K)$, we must have  a.s.: $\rho\le \tau^{\etazero}$ and $\zeta=\eta$ on $[0,\rho)$. Therefore $\eta$ is the maximal solution also on $T\Diff^{s+1}_\mu(K)$. The proof is complete.
\end{proof}

\section{Applications to the stochastic Euler equations in the Eulerian form}\label{sec:Euler_Lagr}

In this section, as before, $K$ denotes a smooth manifold (possibly with smooth boundary) modelled on $\mathbb{R}^d$.
Assume that 
\begin{equation}\label{eqn-s}
s>d/2+1.
\end{equation}
and that $\mathbb{F}=(\mathcal{F}_t)_t$ denotes our choice of filtration. 
We choose and fix $u_0 \in \mathfrak{X}^s_\mu(K)$. We also recall the stochastic Euler equations in Lagrangian form \eqref{eq:stochEuler_Lagr} on manifold  $T\Diff^s_\mu(K)$, that is:
\begin{equation*}
\begin{aligned}
d\eta_t &= B(\eta_t)dt +\Sigma(\eta_t)\bullet dW_t,\\
\eta_0&= \etazero, 
\end{aligned}
\end{equation*}
with a given $\etazero \in T\Diff^s_\mu(K)$ and suitable noise in $\VFsmu[r], r\geq s$. 
Here $B$ is the Ebin-Marsden drift and $\Sigma$ is the diffusion coefficient, defined in Section \ref{sec:noloss_nogain}. We now would like to translate our results to the stochastic Euler equation on $K$. Writing it in the Eulerian form on $\mathfrak{X}^s_\mu(K)$, the equation reads 
\begin{align}
\begin{aligned}\label{eq:stochEuler_Eulerian}
    du_t &+\Pi[\nabla_{u_t} u_t]dt = dW_t,\\
    u(0)&=u_0,
\end{aligned}
\end{align}
where   
$\Pi$ is the Leray-Helmholtz projection from $\mathfrak{X}^{s-1}(K)$ onto the closed subspace $\VFsmu[s-1]$ of the divergence-free vector fields, cf.\ \cite[Remark 1.6]{Tem01}.

In the rigorous definition of solutions to \eqref{eq:stochEuler_Eulerian}, we cannot use directly Definition \ref{def:SDE_man_sol} of solutions to SDEs, because here the drift function $u\mapsto \Pi[\nabla_u u]$ is not continuous (not even globally defined) on $\VFsmu[s]$. However we can easily modify the definition as follows, by requiring the integral equality to hold in $\VFsmu[s-1]$:

\begin{defn}
    A local (strong and smooth) solution to the stochastic Euler equation in the Eulerian form \eqref{eq:stochEuler_Eulerian} is an $\mathfrak{X}^s_\mu(K)$-valued continuous and $\mathbb{F}$-progressively measurable process $u=(u_t)_{t\in [0,\tau)}$, with $\tau>0$ accessible stopping time, which satisfies, in $\mathfrak{X}^{s-1}_\mu(K)$, 
    \begin{align*}
        u_t = u_0 +\int_0^t \Pi[\nabla_{u_r}u_r] dr +W_t \quad \forall t\in [0,\tau).
    \end{align*}
\end{defn}

Definition \ref{defn:maximal_sol} of a local maximal solution and Definition \ref{defn:local_uniq} of local uniqueness are easily adapted to the Eulerian form \eqref{eq:stochEuler_Eulerian}. Our third main result is the equivalence between the Eulerian form and the Lagrangian form of the stochastic Euler equation. 
\begin{thm}\label{thm:equivalence_flow_pde}
Assume \eqref{eqn-s} and that  $W$ is a Wiener process satisfying Hypothesis \ref{hp:noise} with $r= s+2$. Then the Eulerian form \eqref{eq:stochEuler_Eulerian} and the Lagrangian form \eqref{eq:stochEuler_Lagr} of the stochastic Euler equations are equivalent, in the following sense:
\begin{itemize}
\item If $\eta$ is a solution on $[0,\tau)$ to the Lagrangian form \eqref{eq:stochEuler_Lagr} in $T\Diff^s_\mu(K)$, with $\pi(\etazero)=\id$, then $$u(t)=\eta(t)\circ \pi(\eta(t))^{-1}$$ is a solution on $[0,\tau)$ to the Eulerian form \eqref{eq:stochEuler_Eulerian} in $\VFsmu$, where the initial condition $u_0$ is the vector field corresponding to $\etazero$, see \eqref{tangent_ident}.
\item If $u$ is a solution on $[0,\tau)$ to the Eulerian form \eqref{eq:stochEuler_Eulerian} in $\VFsmu$ and $\Phi$ is the unique flow solution on $[0,\tau)$ to the (random) ODE
\begin{align}
d\Phi(t) = u(t,\Phi(t))\, dt,\quad \Phi(0)=\id,
\label{eqn-flow_ODE}
\end{align}
then the process $\eta$ defined by 
\begin{equation}\label{eqn-eta}
\eta(t)=u(t)\circ \Phi(t), \;\;\; t\in [0,\tau),    
\end{equation}
 is a solution on $[0,\tau)$ to the stochastic Euler Equations in the Lagrangian form \eqref{eq:stochEuler_Lagr} on $T\Diff^s_\mu(K)$, with $\etazero$ is the element in $T_{\id}\Diff^s_\mu(K)$ corresponding to the vector field $u_0$.
\end{itemize}
\end{thm}

The proof of Theorem \ref{thm:equivalence_flow_pde} will take most of the present section. The basic idea is to embedd the group of $H^s$-diffeomorphisms as a split submanifold of a separable Hilbert space and work with the embedded equations. Before we investigate the details let us record some consequences of Theorem \ref{thm:equivalence_flow_pde} first. Combining Theorem \ref{thm:equivalence_flow_pde} and Theorems \ref{thm:wellposed_Lagr}, \ref{thm:noloss_nogain}, we obtain both local existence and uniqueness (with the same assumptions of the Lagrangian form) and the no loss-no gain result for the Eulerian form:

\begin{thm}\label{thm-unique solution}
Assume that $s>d/2+1$ and $W$ satisfies Assumption \ref{hp:noise} with $r\ge s+2$. Then, for every initial condition $u_0 \in \mathfrak{X}^s_\mu(K)$, there exists a unique local maximal solution to the stochastic Euler equations in the Eulerian form \eqref{eq:stochEuler_Eulerian}.
\end{thm}

\begin{proof}
    The existence of a local solution and the local uniqueness for the Eulerian form follow from the existence of a local solution and local uniqueness for the Lagrangian form, Theorem \ref{thm:wellposed_Lagr}, via Theorem \ref{thm:equivalence_flow_pde}. Concerning the maximality, let $(\eta(t))_{t\in [0,\tau)}$ be the maximal solution to the Lagrangian form (which exists by Theorem \ref{thm:wellposed_Lagr}) and take $u(t)=\eta(t)\circ\pi(\eta(t))^{-1}$. Let $(\tilde{u}(t))_{t\in [0,\tilde{\tau})}$ be another solution to the Eulerian form and take $(\tilde{\Phi}(t))_{t\in [0,\tilde{\tau})}$ the flow solution to \eqref{eqn-flow_ODE} with $\tilde{u}$ in place of $u$. Then, by Theorem \ref{thm:equivalence_flow_pde}, $\tilde{u}\circ \tilde{\Phi}$ is a solution to the Lagrangian form and hence, by maximality, we must have  a.s.
    \begin{align*}
        \tilde{\tau}\le \tau, \quad \tilde{u}(t)\circ\tilde\Phi(t) = u(t)\circ\Phi(t) \quad \forall t\in [0,\tilde{\tau}).
    \end{align*}
    Hence we get,  a.s.: $\tilde{\Phi}(t)=\Phi(t)$ and $\tilde{u}(t)=u(t)$ for $t\in [0,\tilde{\tau})$, that is $u$ is a maximal solution to the Eulerian form. The proof is complete.
\end{proof}

\begin{cor}
Assume that $s>d/2+1$ and $W$ satisfies \ref{hp:noise} with $r= s+3$. Let $u=(u_t)_{[0,\tau)}$ be a solution to the Eulerian Euler equation \eqref{eq:stochEuler_Eulerian} in $\mathfrak{X}^s_\mu(K)$. If $u_0$ belongs to $\VFsmu[s+1]$, then $u$ takes values in $\VFsmu[s+1]$ and solves the Eulerian Euler equation \eqref{eq:stochEuler_Eulerian} in $\VFsmu[s+1]$.
\end{cor}

\begin{proof}
    Take $\Phi$ the flow solution to \eqref{eqn-flow_ODE} and $\eta=u\circ \Phi$, by Theorem \ref{thm:equivalence_flow_pde} $\eta$ solves the Lagrangian Euler equation \eqref{eq:stochEuler_Lagr} in $T\Diff^s_\mu(K)$. If $u_0$ belongs to $\VFsmu[s+1]$, then $\etazero$ is in $T\Diff^{s+1}_\mu(K)$ and so, since $r= s+3$,  by Theorem \ref{thm:noloss_nogain}, $\eta$ is a solution to the Lagrangian Euler equation \eqref{eq:stochEuler_Lagr} in $T\Diff^{s+1}_\mu(K)$. We let $\pi \colon T\Diff^{s+1}_\mu (K) \rightarrow \Diff^{s+1}_\mu (K)$  be the bundle projection. Then Theorem \ref{thm:equivalence_flow_pde} (with $s$ replaced by $s+1$) implies that $u=\eta\circ \pi(\eta)^{-1}$ is a solution to the Eulerian Euler equation \eqref{eq:stochEuler_Eulerian} in $\VFsmu[s+1]$.
\end{proof}

We finish this subsection with the following result about the global existence of solutions to equation in the 2-dimensional case. See \cite{Shkoller01} for a discussion of global solutions to the deterministic Euler equations in the Ebin-Marsden setting.  

\begin{cor}\label{cor-global for d=2}
Assume that $d=2$,  $s>2$ and the Wiener process $W$ satisfies condition \ref{hp:noise} with $r= s+3$. 
Assume that $\etazero \in T\Diff^s_\mu(K)$, with $\pi(\etazero)=\id$. 
Then the Lagrangian form \eqref{eq:stochEuler_Lagr} in $T\Diff^s_\mu(K)$
has a global solution. 
\end{cor}
\begin{proof} Assume that $s>2$ and $W$ is a Wiener process satisfying condition  \ref{hp:noise} with $r= s+3$. 
Let us choose and fix $\etazero \in T\Diff^s_\mu(K)$, with $\pi(\etazero)=\id$. Let $u_0 \in \mathfrak{X}^s_\mu(K)$ be the corresponding divergence free vector field. It is known that the stochastic Euler equations in the Eulerian form \eqref{eq:stochEuler_Eulerian} have a unique global $ \mathfrak{X}^s_\mu(K)$-valued solution. Then, by Lemma \ref{lem-random_ODE}, there exists
 the unique global flow solution on $[0,\infty)$ to the (random) ODE \eqref{eqn-flow_ODE}. We conclude the proof by applying the second bullet point of Theorem \ref{thm:equivalence_flow_pde}.
\end{proof}

\subsection*{Passage from Eulerian to Lagrangian solution}
\addcontentsline{toc}{subsection}{Passage from Eulerian to Lagrangian solution}
We prove first the second implication in Theorem \ref{thm:equivalence_flow_pde}, namely the passage from the Eulerian solution to the Lagrangian solution. We start with a solution $u$ on $[0,\tau)$ to the Eulerian form of the Euler equation on $\VFsmu$.

\begin{lem}\label{lem-random_ODE}
Let $u \colon [0,\tau) \times \Omega \rightarrow \VFsmu$ be a solution to \eqref{eq:stochEuler_Eulerian}. There exists a unique progressively measurable process $\Phi\colon [0,\tau)\times\Omega \to \Dmu[s]$ such that  a.s. for every $k \in K$, $\Phi(k)$ solves the random ODE on $K$
\begin{align}
\begin{aligned}\label{eqn-random_ODE}
&d\Phi_t(k) = u(t,\Phi_t(k))\, dt,\\
&\Phi_0(k) = k.
\end{aligned}
\end{align}
\end{lem}

\begin{proof}
Let $(\tau_n)_n$ be an announcing sequence for $\tau$. For $n\in\mathbb{N}$, $T>0$, we take $u^n_t=u_{t\wedge \tau^n}$. By Proposition \ref{prop:Hs-flows}, for a.s. $\omega$, there exists a unique flow $\phi^{n,\omega}\colon [0,T]\to \Diff^s_\mu(K)$ solving the following random ODE
\begin{align}
    d\Phi^{n,\omega}_t(k) &= u^n(t,\Phi^{n,\omega}_t(k),\omega)\, dt,\\
    \Phi^{n,\omega}(k)&=k.
\end{align}
For each $n$ and $t \in [0,T]$, the map $\Omega\ni \omega \mapsto \Phi^{n,\omega}_t\in \Diff^s_\mu(K)$ is the composition of the map
\begin{align*}
    \Omega\ni \omega \mapsto (u^{n,\omega}_r)_{r\in [0,t]} \in L^1([0,t],\mathfrak{X}^s_\mu(K))
\end{align*}
which is $\mathcal{F}_t$-$\mathcal{B}(L^1([0,t];\mathfrak{X}^s_\mu(K)))$-measurable (because $u$ is $(\mathcal{F}_r)_r$-progressively measurable), and the map
\begin{align*}
    (u^{n,\omega}_r)_{r\in [0,t]} \to \Phi^{n,\omega}_t\in \Diff^s_\mu(K),
\end{align*}
which is continuous, hence Borel, by Proposition \ref{prop:Hs-flows} ($\Diff^s(K)$ can be replaced by $\Diff^s_\mu(K)$ since the latter is a closed submanifold of the former). Therefore $\Phi^n$ is a $(\mathcal{F}_t)_t$-adapted, continuous process with values in $\Diff^s_\mu(K)$ and so it is $(\mathcal{F}_t)_t$-progressively measurable. For every $m\ge n$, since $u_{\cdot\wedge \tau^n} = u^n_{\cdot\wedge\tau^n} = u^m_{\cdot\wedge\tau^n}$, we have $\Phi^n_{\cdot\wedge\tau^n}=\Phi^m_{\cdot\wedge\tau^n}$  a.s.. Hence we can glue together the flows $\Phi^n$ into a $\Diff^s_\mu(K)$-valued, $(\mathcal{F}_t)_t$-progressively measurable process $\Phi$ solving the random ODE \eqref{eqn-random_ODE}. The proof is complete.
\end{proof}

\begin{setup}\label{eta:process}
Using the process $\Phi$ 
from Lemma \ref{lem-random_ODE}, we introduce a $T\Diff^s_\mu(K)$-valued progressively measurable process process $\eta$ defined by  formula \eqref{eqn-eta}, i.e.
\begin{align*}
\eta_t=u(t)\circ \Phi(t),
\end{align*}
where $\Phi$ is a solution to \eqref{eqn-flow_ODE}. In particular, for $k \in K$, the $TK$-valued  
\[\eta(k)=(\eta_t(k): t \in [0,\tau)) \in TK,\]
is a  progressively measurable process.
\end{setup}
In the following, we denote by $\Sigma$ the diffusion term from \eqref{eqn-Sigma-02}, $\pi_K\colon TK\to K$ is the bundle projection and  $i\colon K\to \R^m$ an embedding of $K$ as a spliy submanifold, for a suitable $m$, Proposition \ref{prop:embedding_Diff}.
In particular, we have the embedding $Ti\colon TK \rightarrow \R^{2m}$. The process $\bar{\eta}_t = Ti\circ \eta$ takes values in $H^s(K,\R^{2m})$.

\begin{prop}\label{prop:Lagr_eval_emb}
 The $H^s(K,\R^{2m})$-valued process $\eta$ from \ref{eta:process} satisfies on $[0,\tau)$ the SDE 
\begin{align}
\begin{split}\label{eq:eta_Hs}
d\bar{\eta}_t &= B(\eta_t)dt +\Sigma(\eta_t)\, dW_t +\frac12\tr_{\rK}[D\Sigma\circ \Sigma(\eta_t)] dt,\\
\bar{\eta}_0&=Ti\circ u_0.
\end{split}
\end{align}
Here, we suppress the identification and write $B(\eta_t)(k)$ instead of $D (Ti)\circ B(\eta_t)(k)$ and $\Sigma$ for $D (Ti)\circ \Sigma \colon T\Diff^s_\mu (K)\to \mathscr{L}(\VFsmu[r],H^s(K,\R^{2m}))$, cf.\ \eqref{eqn-Sigma-02}, hence $D\Sigma\circ\Sigma$ stands for $D(D(Ti)\circ \Sigma) \circ \Sigma$.
\end{prop}

\begin{proof} 
Since $\bar{\eta}$ and all the integrals in \eqref{eq:eta_Hs} are continuous processes with values in $H^s(K,\R^{2m})$, it is enough to show that on $[0,\tau)$, for every $k \in K$,
\begin{align}\begin{split}
\bar{\eta}_t(k) &= \bar{\eta}_0(k) +\int_0^t DTi \circ B(\eta_r)(k)dt +\int_0^t DTi\circ \Sigma(\eta_r)\, dW_r(k) \\ & \qquad +\frac12\int_0^t\tr_{\rK}[D(DTi\circ \Sigma)\circ \Sigma(\eta_r)](k)\, dr.
\end{split} \label{eq:SDE_evaluated_1}
\end{align}
To clearly write the stochastic integral in \eqref{eq:SDE_evaluated_1}, consider the map $\text{ev}_k\colon H^s(K,\R^{2m}) \rightarrow \R^{2m}$, which evaluates a function in $k$. Recall from Lemma \ref{lem:eval} that $\text{ev}_k$ is continuous and linear. Hence, by It\^o formula \cite[Theorem A.10]{MMS19}, applied to $\text{ev}_k$
\begin{align*}
\int_0^t DTi\circ \Sigma(\eta)\, dW(k) &= \text{ev}_k \left(\int_0^t DTi\circ\Sigma(\eta)\, dW\right)= \int_0^t \text{ev}_k\circ DTi\circ\Sigma(\eta)\, dW\\ &= \int_0^t DTi\circ\Sigma(\eta)(k)\, dW,
\end{align*}
where $DTi\circ\Sigma(\eta)(k)\colon T\Diff^s_\mu(K)\to \mathscr(\VFsmu[r],\R^{2m})$. Hence \eqref{eq:SDE_evaluated_1} is equivalent to
\begin{align}\begin{split}
\bar{\eta}_t(k) =& \bar{\eta}_0(k) +\int_0^t DTi \circ B(\eta_r)(k)dt +\int_0^t DTi\circ \Sigma(\eta_r)(k)\, dW_r \\ &\qquad +\frac12\int_0^t\tr_{\rK}[D(DTi\circ \Sigma)\circ \Sigma(\eta_r)](k)\, dr.
\end{split}\label{eq:SDE_evaluated}
\end{align}
Fix $k \in K$ and let $u \in \mathfrak{X}_\mu^{s} (K)$. To get \eqref{eq:SDE_evaluated}, we would like to apply It\^o formula to 
$$\text{eval}_{H^{s-1}(K,\R^{2m})}(Ti \circ u,k)=(i(k),\text{eval}(Di \circ u,k)),$$ where $\text{eval}_{H^{s-1}(K,\R^{2m})}$ is the evaluation on $H^{s-1}(K,\R^{2m})$. Moreover, we denoted by $\text{eval} \colon H^{s-1} (K,\R^{m}) \times K \rightarrow \R^{m}$ the evaluation map from Lemma \ref{lem:eval}. However, we cannot apply the formula directly, because $u$ satisfies an equation on $\VFsmu[s-1]$ (identified via $(Ti)_*$ with a subset of $H^{s-1}(K,\R^{2m})$) and on this space the evaluation is not of $C^2$-class. Therefore we need to use a regularization argument.

For $\delta>0$, we can apply the regularization operator $R_\delta\colon H^{s-1}(K,\R^{m})\to H^{s+1}(K,\R^{m})$ from Lemma \ref{lem-regularization}. Consider now
\begin{align*}
G_\delta\colon \VFsmu[s-1] \times K \to K\times \R^{m}, \quad G_\delta(X,k)= (k,\text{eval}(R_\delta\circ Di\circ X,k)).
\end{align*}
By Lemma \ref{lem:eval}, since $s+1>d/2+2$, the map $G_\delta$ is of $C^2$-class and the formula for its derivative yields the following identity for the derivative of the second component $(G_\delta)_2 \coloneq \text{pr}_{\R^m} \circ G_\delta$ of $G_\delta$
\begin{align*}
D(G_\delta)_2((X,Y),v) &= D(R_\delta \circ Di\circ X)(v) +R_\delta\circ Di \circ Y \circ \pi_K(v).
\end{align*}
We shall now work again with the flow solution $\Phi$ to \eqref{eqn-random_ODE} which satisfies $\eta_t = u(t)\circ \Phi(t)$ (cf.\ \ref{eta:process}). Define now for $\delta >0$ the map 
$$\overline{u}^\delta \colon K \rightarrow K \times \R^{m}, \quad \overline{u}^\delta(k) = (k, R_\delta \circ Di \circ u (k)) = G_\delta (u,k).$$ 
Therefore by It\^o`s formula, precisely formula \eqref{eq:SDE_man_def} from Definition \ref{def:SDE_man_sol}, see also Remark \ref{rem:invariance_SDE_map}, we deduce that the $K\times \R^{m}$-valued process $\bar u^\delta(\Phi(k))$
satisfies the following SDE on $[0,\tau)$ (where we identify $T(K\times\R^m) \cong TK\times \R^{2m}$):
\begin{align*}
d\bar u^\delta(\Phi(k))
&= (u(\Phi(k)),(G_\delta)_2 (u,\Phi(k)),
D(R_\delta \circ Di \circ u) (u(\Phi(k))) -R_\delta\circ Di\circ [\Pi\nabla_u u] \circ \Phi(k))\, dt\\
&\quad +(0_{\Phi(k)},(G_\delta)_2 (u,\Phi(k)), \text{eval}(R_\delta\circ Di \circ \cdot), \Phi(k)\bullet dW)\\
&= (\ldots)dt + (0_{\Phi(k)}, (G_\delta)_2 (u,\Phi(k)),\underbrace{\text{eval}(\cdot, \Phi(k)) \circ (R_\delta\circ Di)_* }_{\equalscolon \theta^\delta (\bar{u}^\delta(\Phi(k)))}  \bullet dW),
\end{align*}
where $\theta^\delta \colon K\times \R^m \rightarrow \mathscr{L}(\VFsmu[r],\R^m)$ is given by
\begin{align*}
\theta^\delta (k,x)(w)=(\text{eval}(\cdot, k) \circ (R_\delta\circ D^2i)_*)(w) = \text{eval}(R_\delta\circ Di \circ w, k).
\end{align*}
Note that $\theta^\delta$ does not depend on the vector component (this is due to the simplicity of the additive noise we treat).

For technical reasons\footnote{The reason is that, in the SDE \eqref{eq:SDE_Lagr_approx} for $u^\delta$, the Stratonovich integral of $\bar{u}^\delta$ appears: having the SDE for $\bar{u}^\delta$, we can write this Stratonovich integral as an It\^o integral plus correction.}, it is convenient to work both with $\bar{u}^\delta$, where we did not embed $K$ yet, and $u^\delta \coloneq (i,I_{\R^m}) \circ \bar{u}^\delta \colon K \rightarrow \R^{2m}$, where $(i,I_{\R^m})\colon K\times \R^m\to \R^m \times \R^m =\R^{2m}$ is the embedding with $I_{\R^m}\colon \R^m \to \R^m$ the identity map. 
Apply now Lemma \ref{lem:SDE_embedded} 
to deduce that $u^\delta$ solves the following SDE on $[0,\tau)$ (since the mapping takes values in flat space, we express only the vector part of the SDE now, suppressing the base point): 
\begin{align}
\begin{split}\label{eq:SDE_Lagr_approx}
&du^\delta(k)\\
 =& \underbrace{Du^\delta \circ u(\Phi(k))\, dt -(0_{\R^m}, R_\delta\circ Di\circ [\Pi\nabla_u u] \circ \Phi(k))}_{B^\delta}\, dt
+(0_{\R^m}, \theta^\delta (\bar u^\delta (\Phi(k))) \bullet dW)\\
=:& B^\delta dt +\big(0_{\R^m},\theta^\delta (\bar{u}^\delta(\Phi(k))\, dW +\frac12\tr_{\rK}[D\theta^\delta \circ \underbrace{(0_{\Phi(k)}, (G_\delta)_2 (u,\Phi(k)),\theta^\delta (\bar{u}^\delta(\Phi(k)))}_{\equalscolon \bar \sigma^\delta (\bar{u}^\delta(\Phi(k)))}]dt\big)
\end{split}
\end{align}
Here $\bar \sigma^\delta \colon K \times \R^{m} \rightarrow \mathcal{L}(\VFsmu[r],TK\times R^{2m}), \bar{\Sigma}^\delta(k,x)(w) = (0_k,x,\theta^\delta(k,x)(w))$.
Next we consider the limit $\delta\to 0$. As $t\mapsto u (t)$ is a continuous map with values in $\VFsmu$, it is a continuous map with values in the $C^1$ sections on $K$, by the Sobolev embedding theorem, \cite[Corollary to Theorem 9.2]{MR0248880}. Therefore, $u^\delta(\Phi(k))$, resp. $\bar u^\delta(\Phi(k))$, converges to $Ti \circ u(\Phi(k))$, resp. $\bar u(\Phi(k)):= (\pi_K,Di)\circ u(\Phi(k))$, locally uniformly on $[0,\tau)$  a.s.. Since $Di\circ u$ has $H^s$ regularity, the map $Di\circ \Pi\nabla_u u$ is $H^{s-1}$. Now as $s>d/2+1$, by standard arguments the drift $B^\delta$ from \eqref{eq:SDE_Lagr_approx} converges, locally uniformly on $[0,\tau)$  a.s., to
\begin{align*}
[(DTi)\circ Tu \circ u -(0,D^2i\circ [\Pi\nabla_u u])]\circ \Phi(k) = DTi\circ \tilde B_k,
\end{align*}
where
\begin{align}\label{eq:Btilde}
\tilde B_k = (Tu \circ u\, dt -\text{vl}_{TK}(u, \Pi\nabla_u u))(\Phi(k)).
\end{align}
To deal with the diffusion coefficient, consider 
$\bar \sigma \colon K \times \R^m \rightarrow \mathscr{L}(\VFsmu[r],TK \times \R^{2m})$ defined by 
$\bar \sigma(k,x)(w) = (0_k,x,Di\circ w (k))$. In addition, we set $$\theta \colon K \times \R^m \rightarrow \mathscr{L}(\VFsmu[r], \R^m),\quad \theta(k,x)(w)\coloneq Di \circ w (k).$$
Again standard calculations establish convergence in the compact open $C^1$-topology of $\bar \sigma^\delta$ to $\bar \sigma$ in $C^1_{loc}(K\times \R^{m},\mathscr{L}(\VFsmu[r],TK\times \R^{2m}))$. More precisely, set $J=K\times \R^{m}$. and identify $TK$ via $Ti$ with a subset of $\R^{2m}$. Every tangent space $T_zJ$ inherits a canonical norm $\lVert \cdot \rVert_z$ induced by pulling back the norm of $\R^{2m} \times \R^{2m}$ to $TJ = TK \times \R^{2m}$. Then the following holds locally uniformly with respect to  $z\in J$ for $\delta \rightarrow 0$, 
\begin{align}
\begin{split}\label{eq:convergence_sigma_delta}
&\sup_{\substack{w\in \VFsmu[r], \\ \|w\|_{H^{r}}=1}} \|\bar \sigma^\delta(z)w -\bar \sigma(z)w\|_{z} \to 0,\\
&\sup_{\substack{v\in T_zJ,\\ \|v\|_{z}=1}}\sup_{\substack{w\in \VFsmu[r],\\ \|w\|_{H^{r}}=1}} \|[T_z\bar \sigma^\delta (v)]w-[T_z\bar \sigma (v)]w\|_{z} \to 0.
\end{split}
\end{align}
Recall from \cite[1.8]{MMS19} that the smooth map $(Ti,I_{\R^{2m}}) \colon TK \times \R^{2m} \rightarrow \R^{2m}\times \R^{2m}$ induces a smooth postcomposition operator on $\mathcal{L}(\VFsmu[r],TK\times \R^{2m})$. Hence we conclude that $(Ti,I_{\R^{2m}})\circ \bar \sigma^\delta(\bar u^\delta(\Phi(k))$ converges a.s., locally uniformly on $[0,\tau)$ with respect to the $C^1$-topology, to $(Ti,I_{\R^{2m}})\circ \bar \sigma(\bar u(\Phi(k)) \in \mathscr{L}(\mathfrak{X}^{r}_\mu(K),\R^{2m} \times \R^{2m})$ for $\delta \rightarrow 0$. Therefore, we can apply Lemma \ref{lem:convergence_Ito_int} with an announcing sequence $(\tau_n)_{n=1}^\infty$ for $\tau$, to get that for every $n$ as $\delta \rightarrow 0$,
\begin{align*}
\sup_{t\in [0,\tau_n]}\left|\int_0^t [\theta^\delta(\bar u^\delta(\Phi(k))) - \theta (u(\Phi(k)))]\, dW \right|\to 0 \text{ in  } \mathbb{P}.
\end{align*}
Moreover, \eqref{eq:convergence_sigma_delta} implies that with  $(e_k)_{k=1}^\infty$ being an orthonormal basis for the RKHS $\rK = \VFsmu[r]$ that the trace term from \eqref{eq:SDE_Lagr_approx}
\begin{align*}
\tr_{\rK}[D\theta^\delta\circ \bar \sigma^\delta(\bar u^\delta(\Phi(k))) ] = \sum_k [D\theta^\delta\circ \bar \sigma^\delta(u^\delta(\Phi(k)))e_k][e_k]
\end{align*}
converges, locally uniformly on $[0,\tau)$  a.s., for $\delta \rightarrow 0$ to
\begin{align*}
\tr_{\rK}[D\theta\circ \bar \sigma(\bar u(\Phi(k))) ] = \sum_k [D\theta\circ \bar \sigma(\bar u(\Phi(k)))e_k][e_k].
\end{align*}
Hence we can pass to the limit in probability in \eqref{eq:SDE_Lagr_approx}, as $\delta \to 0$: we obtain that the process $\bar{\eta}(k)=\bar{u}(\Phi(k))$, as $\R^{2m}$-valued process, satisfies on $[0,\tau)$
\begin{align}
d\bar{\eta}(k) = \tilde{B}_k dt +\big(0_{\R^m},\theta (\bar{u}(\Phi(k))\, dW +\frac12\tr_{\rK}[D\theta \circ \bar \sigma (\bar{u}(\Phi(k)))]dt\big).\label{eq:SDE_evaluated_2}
\end{align}
In order to show \eqref{eq:SDE_evaluated}, it remains to show that
\begin{align}
&\tilde B_k = DTi \circ B(\eta)(k),\label{eq:coeff_1}\\
&(0_{\R^m},\theta(\bar u(\Phi(k))) = DTi\circ \Sigma(\eta)(k),\label{eq:coeff_2}\\
&(0_{\R^m},\tr_{\rK}[D\theta\circ \bar\sigma(\bar u(\Phi(k)))]) = \tr_{\rK}[D(DTi\circ\Sigma)\circ \Sigma(\eta)](k).\label{eq:coeff_3}
\end{align}
The equality \eqref{eq:coeff_1} follows from comparing \eqref{eq:Btilde} with the evaluation at $k$ of the expression for $B$ given in \cite[Proposition 14.2]{EM70}. Namely, for $\eta = u\circ \Phi \in T\Diff^s_\mu(K)$,
\begin{align}
B(\eta) = Tu \circ \eta -\text{vl}_{T\Diff^s_\mu(K)}(\eta,\Pi[\nabla_u u]\circ \Phi).\label{eq:EM_drift}
\end{align}
For a fixed $k \in K$, a quick computation  using the definitions of $\Sigma$ and $\theta$, yields the following relation for $\eta \in T\Diff^s_\mu (K) \subseteq H^s(K,TK), w \in \VFsmu[r]$:
\begin{align*}
DTi \circ \left(\Sigma (\eta)w\right) (k)&=(0_{\R^{m}},Di\circ w(\pi_K(\eta (k)))) = (0_{\R^m}, \theta(\pi_K (\eta (k)), Di (\eta (k)))(w))
\end{align*}
Using this relation we can easily show the equalities \eqref{eq:coeff_2} and \eqref{eq:coeff_3}. Hence the coefficients of \eqref{eq:SDE_evaluated_2} are equal to the coefficients of \eqref{eq:SDE_evaluated}, thus showing \eqref{eq:SDE_evaluated} for $\bar{\eta}$. The proof is complete.
\end{proof}

We have now collected all necessary results to establish the second statement of Theorem \ref{thm:equivalence_flow_pde} showing that solutions to the Eulerian formulation give rise to solutions of the Lagrangian formulation.

\begin{proof}[Proof of Theorem \ref{thm:equivalence_flow_pde}, passage from Eulerian to Lagrangian solution]
By Proposition \ref{prop:Lagr_eval_emb}, $\bar{\eta}=Ti\circ \eta$ satisfies the Lagrangian equation \eqref{eq:stochEuler_Lagr} on $[0,\tau)$ as an equation on $H^s(K,\R^{2m})$. By Corollary \ref{cor:embedding_TDiff}, $(Ti)_*$ is an embedding, turning $T\Diff^s_\mu(K)$ into a split submanifold of $H^s(K,\R^{2m})$. Hence, by Lemma \ref{lem:SDE_embedded}, $\eta$ satisfies the Lagrangian equation \eqref{eq:stochEuler_Eulerian} on $[0,\tau)$ (now as equation on $T\Diff^s_\mu(K)$). The proof of the passage from Eulerian to Lagrangian solution is complete.
\end{proof}

\subsection*{Passage from Lagrangian to Eulerian solution}
\addcontentsline{toc}{subsection}{Passage from Lagrangian to Eulerian solution}
We prove now the first implication in Theorem \ref{thm:equivalence_flow_pde}, namely the passage from the Lagrangian solution to the Eulerian solution. 

\begin{setup}\label{setup:projected_eta}
For this we start with a solution $\eta$ on $[0,\tau)$ to the Lagrangian equation \eqref{eq:stochEuler_Lagr} on $T\Diff^s_\mu (K) \subseteq H^s(K,TK)$. Define 
$$\Phi \coloneq \pi_K \circ \eta \colon [0,\tau) \times \Omega \rightarrow  \Diff^s_\mu (K),$$ where $\pi_K \colon TK \rightarrow K$ is the bundle projection. Further, we set 
$$u_t (\omega) \coloneq u(t,\cdot,\omega) \coloneq \eta(t,\omega)\circ (\Phi(t,\omega))^{-1} \in T_{\id} \Diff^s_\mu (K)= \VFsmu[s],$$
where the inverse is taken in the group $\Diff^s_\mu(K)$. Note that by definition we have $u(t,\Phi(t,\omega),\omega) = \eta(t,\omega)$.
\end{setup}

\begin{lem}\label{lem:proj_eta}
 Fix $k \in K$, $\eta$ a solution to \eqref{eq:stochEuler_Lagr} and $\Phi = \pi_K \circ \eta$ and $u$ as in \ref{setup:projected_eta}. Write $u(k) \colon [0,\tau)\times \Omega \rightarrow TK, u(k)(t,\omega)=u_t (\omega)(k)$ and $\Phi^{-1}(k)\colon [0,\tau)\times\Omega\to K, (t,\omega) \mapsto (\pi_K \circ \eta (t,\omega))^{-1}(k)$. Then the $K$-valued process  satisfies the random partial differential equation
\begin{align}
\begin{split}\label{eq:TE}
d\Phi^{-1}(k) &= -T\Phi^{-1} (u(k)) dt,\\
\Phi^{-1}_0(k) &= k.
\end{split}
\end{align}
\end{lem}

\begin{proof}
For a.s. $\omega \in \Omega$, $\Phi(\omega)$ is the $H^s$ flow solving the random ODE
\begin{align}\label{Phi_eq}
d\Phi(\omega) = u(t,\Phi(t,\omega),\omega)\, dt.
\end{align}
 To see this, we observe that $d\Phi = T\pi_K (d\eta)$. Inserting the right hand side of \eqref{eq:stochEuler_Lagr} and rewrite the drift $B$ via \eqref{eq:EM_drift}. 
 Recall now that the vertical lift on $\Diff^s_\mu(K)$ in \eqref{eq:stochEuler_Lagr} is taken by the identification of $\Diff^s_\mu(K)$ as a subset of $H^s(K,TK)$ to the pushforward with the vertical lift on $TK$, cf.\ \eqref{Spray_form}. Similarly, the identification takes $T\pi \colon T^2\Diff^s_\mu (K) \rightarrow T\Diff_\mu^s (K)$ to the pushforward by $T\pi_K$. Now as $T\pi_K (v) =0$ holds for every $v \in T^2 K$ which is vertical, we deduce that $T\pi_K$ drops all parts of \eqref{eq:stochEuler_Lagr} which are vertical. After dropping all vertical terms, we see that $\Phi$ solves \eqref{Phi_eq}.
Now we derivate $\Phi^{-1}(k)$ using the formula for the inverse \eqref{inversion_nonsmooth} in $\Diff^s_\mu (K)$. Together with the defintions of $\Phi$ and $u$ then yields \eqref{eq:TE}.
\end{proof}

Note that the reason \eqref{eq:TE} needed to be formulated pointwise for every $k \in K$ is the loss of (spatial) derivatives on $\Phi^{-1}$. Thus while $\Phi^{-1}$ is a well-defined curve with values in $H^s$-diffeomorphisms, the equation does not make sense on $\Diff^s(K)$.

\begin{setup}\label{setup:uprocess}
Now we write somewhat shorter $ u_t = \eta(t)\circ \Phi^{-1}(t)$ for the flow $u$ from \ref{setup:projected_eta} with $\Phi^{-1}$ the flow from Lemma \ref{lem:proj_eta}. Then $ u$ is a $\VFsmu$-valued progressively measurable process. In particular, for $k \in K$, $u_t(k)$ is a $T_kK$-valued progressively measurable process. Using the embedding $i\colon K\to \R^m$ as defined before Proposition \ref{prop:embedding_Diff}, we obtain the $H^s (K,\R^m)$-valued process $\bar u_t \coloneq Di \circ \bar u_t$.
\end{setup}

\begin{prop}\label{prop:Eul_eval_emb}
The $H^{s-1}(K,\R^m)$-valued process $\bar u_t $ from \ref{setup:uprocess} satisfies on $[0,\tau)$ the stochastic differential equation 
\begin{align}
\begin{split}\label{eq:u_Hsm1}
d\bar u_t &= -\Pi[\nabla_{\bar u_t} \bar u_t] dt +dW_t,\\
\bar u_0 &= Di\circ \etazero.
\end{split}
\end{align}
Here, we abuse notation and write $\Pi[\nabla_{u} u]$ and $dW$ for $D^2i\circ \Pi[\nabla_{u} u]$ and $Di\circ dW$ respectively.
\end{prop}

\begin{proof}
Since $\bar u$ and all the integrals in \eqref{eq:u_Hsm1} are continuous processes with values in $H^{s-1}(K,\R^m)$, it is enough to show that on $[0,\tau)$, for every $k \in K$,
\begin{align}
\bar u_t(k) = \bar u_0(k) -\int_0^t \Pi[\nabla_{\bar u_s} \bar u_s](k)\, ds +W_t(k),\label{eq:SPDE_evaluated}
\end{align}
as equality on $\R^m$. Fix $k \in K$. As before in the proof of Proposition \ref{prop:Eul_eval_emb}, to establish \eqref{eq:SPDE_evaluated}, we would like to apply It\^o formula to $\text{eval}(Di \circ \eta,\Phi^{-1}(k))$, where now $\text{eval} \colon H^s (K,\R^m) \times K \rightarrow \R^m$ is the evaluation map. However, this is not directly possible, because $\eta$ is in $T\Diff^s_\mu(K)$ (mapped via $(Di)_\ast$ to $H^s(K,\R^{m})$, where $\text{eval}$ is not $C^2$. Therefore we use a regularization argument. For $\delta>0$, we consider the map
\begin{align*}
F_\delta:T\Diff^s_\mu(K) \times K \to \R^m,\quad F_\delta(\eta,k) = \text{eval}(R^\delta\circ Di\circ \eta,k),
\end{align*}
where $R_\delta\colon H^s(K,\R^m)\to H^{s+2}(K,\R^m)$ is the regularisation operator from Lemma \ref{lem-regularization}. By Lemma \ref{lem:eval}, since $s+2>d/2+2$, the map $F_\delta$ is of $C^2$-class  and we have
\begin{align*}
DF_\delta(V,v) = D(R_\delta\circ Di\circ \pi_{T\Diff^s_\mu}(V))(v) +R_\delta\circ Di\circ V\circ \pi_K(v),
\end{align*}
where $\pi_{T\Diff^s_\mu}:T^2\Diff^s_\mu(K) \to T\Diff^s_\mu(K)$ is the bundle projection. Therefore by formula \eqref{eq:SDE_man_def} from Definition \ref{def:SDE_man_sol} we infer together with \eqref{eq:TE} that  
the $\R^m$-valued process $\eta^\delta(\Phi^{-1}(k)):= (R_\delta\circ Di\circ \eta)(\Phi^{-1}(k))$ solves the following SDE on $[0,\tau)$,
\begin{align}
\begin{split}\label{eq:u_delta}
&d\eta^\delta(\Phi^{-1}(k))\\
&= D(R_\delta\circ Di\circ \eta)\circ (-T\Phi^{-1} (u(k)))dt +R_\delta\circ D^2i\circ B(\eta)\circ \Phi^{-1}(k)\, dt\\
&\quad +R_\delta\circ D^2i\circ \Sigma(\eta)\circ \Phi^{-1}(k) \bullet dW.
\end{split}
\end{align}
By the definition of $\Sigma$, namely $\Sigma(\eta)w=\text{vl}_{T\Diff^s_\mu(K)}(\eta,w\circ \pi_K(\eta))$, we have, for every $w \in \VFsmu[r]$ since $\Sigma$ takes values in the vertical bundle (cf.\ \eqref{eq:Sigma_vertical}),
\begin{align*}
[R_\delta\circ D^2i\circ \Sigma(\eta)\circ \Phi^{-1}(k)] w = R_\delta\circ Di\circ w\circ \Phi\circ \Phi^{-1}(k) = R_\delta\circ Di\circ w(k).
\end{align*}
Since the latter term is deterministic (it does not depend on $\eta$ nor on $\Phi$) and $\eta^\delta(\Phi^{-1}(k))$ lives on the Hilbert space $\R^m$, the It\^o-Stratonovich correction is zero and equation \eqref{eq:u_delta} reads in It\^o form as
\begin{align}
\begin{split}\label{eq:u_delta_2}
&d\eta^\delta(\Phi^{-1}(k))\\
&= -D(R_\delta\circ Di\circ \eta)\circ T\Phi_t^{-1} (u(k))\, dt +R_\delta\circ D^2i\circ B(\eta)\circ \Phi^{-1}(k)\, dt\\
&\quad +\text{eval}(R_\delta\circ Di\circ I, k)\, dW\\
&=: b^\delta dt +\text{eval}(R_\delta\circ Di\circ I, k)\, dW.
\end{split}
\end{align}
Now we let $\delta\to 0$. By similar arguments to those in the proof of Proposition \ref{prop:Lagr_eval_emb}, we get that $\eta^\delta(\Phi^{-1}(k))$ converges to $Di\circ u(k)=Di\circ \eta(\Phi^{-1}(k))$ locally uniformly on $[0,\tau)$,  a.s. and the drift $b^\delta$ converges to
\begin{align*}
b &:= -D(Di\circ \eta)\circ T\Phi_t^{-1} (u(k)) + D^2i\circ B(\eta)\circ \Phi^{-1}(k)
\end{align*}
locally uniformly on $[0,\tau)$  a.s.. As for the diffusion term, the stochastic integral reads simply as
\begin{align*}
\int_0^t \text{eval}(R_\delta\circ Di\circ I, k)\, dW = R_\delta \circ W_t (k)
\end{align*}
and so it converges to $W_t(k)$ locally uniformly in time,  a.s.. Hence we can pass to the  a.s. limit in \eqref{eq:u_delta_2} and obtain the following SDE for $u(k)$, as $\R^m$-valued process, on $[0,\tau)$:
\begin{align*}
u_t(k) = u_0(k) +\int_0^t b_s ds +W_t(k).
\end{align*}
It remains to show that $b = -Di\circ \Pi[\nabla_u u](k)$. For this, using the expression \eqref{eq:EM_drift} (cf. also \cite[Proposition 14.2]{EM70}), we have
\begin{align*}
b &= -D(Di\circ u\circ \Phi)\circ T\Phi^{-1} (u(k))+ D^2i\circ (Tu\circ u -\text{vl}_{T\Diff^s_\mu(K)}(\eta,\Pi[\nabla_u u]))(k)\\
&= -D^2i\circ Tu \circ T\Phi \circ T\Phi^{-1}\circ u(k) +D^2i\circ Tu\circ u(k) -Di\circ \Pi[\nabla_u u](k)\\
&= -Di\circ \Pi[\nabla_u u](k).
\end{align*}
 Hence $\bar u(k)$ satisfies \eqref{eq:SPDE_evaluated}. The proof is complete.
\end{proof}

\begin{proof}[Proof of Theorem \ref{thm:equivalence_flow_pde}, a passage from the Lagrangian to the Eulerian solutions] We only need to collect the results established. 
By Proposition \ref{prop:Eul_eval_emb}, $\bar u$ satisfies the Eulerian equation \eqref{eq:u_Hsm1} on $[0,\tau)$ as equation on $H^{s-1}(K,\R^{m})$. 
Now $u$ is an $\VFsmu[s-1]$-valued progressively measurable process such that $Ti \circ u = (i, Di \circ u) = (i, \bar u)$.
Recall that $Ti \colon TK \rightarrow \R^{2m}$ is an embedding. Its pushforward $(Ti)_\ast$ identifies $T\Diff^s_\mu (K)$ with a split submanifold of $H^s(K,\R^{2m})$, Corollary \ref{cor:embedding_TDiff}. Recall that $T_{\id} \Diff^s_{\mu} (K) = \VFsmu[s-1]$ is a closed linear subspace of $T\Diff^s_\mu (K)$. We write now $\text{pr}_m \colon \R^{2m} \rightarrow \R^m$ for the projection onto the first $m$-components of $\R^{2m}$, then $Ti$ identifies $\VFsmu[s-1]$ with the closed affine subspace $A = \{F \in H^{s-1}(K,\R^{2m}) : \text{pr}_{m} \circ F = i\}$ of $H^{s-1}(K,\R^{2m})$. Identifying the affine subspace canonically via $F \mapsto F-(i,0)$ with a closed subspace\footnote{Note that one obtains only a subspace, as $i(K)$ is an embedded submanifold and a vector field needs to take its values in the tangent space to $K$. This is however inessential for our above discussion and we thus get around choosing local representatives for these tangent spaces.} of $H^{s-1} (K,\R^m)$. The resulting identification satisfies $(Ti_*)u - (i,0) =\bar u$. We deduce from Lemma \ref{lem:SDE_embedded} that $u$ satisfies the Eulerian equation \eqref{eq:stochEuler_Eulerian} on $[0,\tau)$ (as an equation in $\VFsmu[s-1]$). The proof is complete.
\end{proof}

\begin{rem}[on forced versions of the Euler equations]
Of obvious interest is, in which way our results carry over to forced versions of the stochastic Euler equations. The equation in question reads in the Eulerian form on $K$ as
\begin{align*}
 \frac{\partial}{\partial t} u + \nabla_u u + \text{grad } p&= \dot{W} + f\\
 \text{div} (u)&=0
\end{align*}
where $f$ is a vector field modelling an external deterministic force.
As already established in \cite[Section 11]{EM70}, the force $f$ then lifts to an additive forcing term for the Lagrangian equation \eqref{eq:stochEuler_Lagr} on $T\Diff_\mu^s (K)$. There, the added force leads to a replacement of the drift term $B$ by a modified drift $B_f$. 
Checking the proofs of our main results, the proofs only needed to assume on the drift term $B$ that the drift is
\begin{enumerate}
\item  regular enough (at least $C^2$, giving the desired $C^{1,1}$ regularity), and 
\item equivariant under the canonical right action of  $\Diff_\mu^s (K)$ on its tagent bundle.
\end{enumerate}
It is a classical result (see again \cite{EM70}) that the new drift term $B_f$ is of class $C^k$ if the time dependent vector field $f \in C(\R, \VFsmu[s+k])$. Moreover, the modified drift term is also again right equivariant. Hence we conclude that all of our results such as the stochastic no-loss-no gain result, Theorem \ref{thm:noloss_nogain}, and results such as Theorem \ref{thm:wellposed_Lagr} hold also by the same proofs for the forced equation.
\end{rem}

\begin{appendix}
\section{Sobolev bundle sections and vector fields}\label{App:Sobsect}

In this appendix we compile some basic material on Sobolev type sections of vector bundles. These arise naturally as model spaces of manifolds of Sobolev type sections. We then will recollect some basic results on flows of Sobolev type vector fields on smooth manifolds (possibly with boundary).

\begin{defn}\label{defn:bundle-sections}
Let $\pi \colon E \to M$ be a smooth vector bundle of finite rank over a $d$-dimensional manifold $M$ (not necessarily compact, but possibly with smooth boundary). Assume that $s > d/2$, whence the $H^s$-Sobolev morphisms are at least continuous (cf.\ Definition \ref{defn:Sobolevmaps}).
We let $$H^s(E) = \{g \in H^s (M,E) \mid \pi \circ g = \id_M\}$$ be the space of $H^s$-sections of $E$ endowed with the subspace topology of $H^s(M,E)$. For a map $f \in H^s (K,M)$ we define
$$H^s_f(K,E) \coloneq \{X \in H^s (K,E) \mid \pi \circ X = f\}.$$ 
\end{defn}

In the literature, cf.\ e.g.\ \cite{zbMATH03808476,MR0198494}, vector valued Sobolev functions and Sobolev type bundle sections on Riemannian manifolds (possibly with boundary) are often defined as completions of spaces of smooth functions. However, this alternative definition coincides with the one from Definition \ref{defn:bundle-sections} as we shall explain now.

\begin{setup}\label{setup:top_struct_sect}
Let $\pi \colon E \to K$ be a vector bundle over $K$ with a fixed bundle metric $g_E$ on $E$. Using the bundle metric, the iterated covariant derivatives $\nabla^k$ and the volume form $\mu$ on $K$, one can define the $H^\ell$-norm for smooth sections and $\ell \in \mathbb{N}_0$ as follows:
$$\lVert X\rVert_{H^\ell} \coloneq \left(\sum_{j=0}^\ell \int_K g_E (\nabla^j X, \nabla^j X) \mathrm{d}\mu\right)^{1/2}, \quad X \in C^{\infty} (K,E).$$
This norm generalises \eqref{Hell_product} to the bundle section setting. As explained in \cite[Definition 2.3]{zbMATH03808476}, one can thus alternatively define the spaces $H^\ell (E)$ for all $\ell \in \mathbb{N}_0$ as the Hilbert space completion of the space of smooth sections with values in $E$.

For $\ell = s > d/2 +1$ the space $H^s(E)$ from Definition \ref{defn:bundle-sections} is isomorphic to the completion of the smooth sections of the bundle $E$ (see \cite[B.2, B.3 and Lemma B.4]{MMS19}). Note that the Hilbert space topology coincides with the subspace topology induced by $H^s(K,E)$ (to define  the ambient space topology, we needed $s$ to be large enough to define Sobolev mappings via charts).

The space $H^s_f(K,E)$ admits a unique Hilbert space structure. We shall always topologize the space $H^s_f(K,E)$ with this structure and note that the topology coincides again with the subspace topology induced by the inclusion 
$H^s_f(K,E) \subseteq H^s(K,E)$.
\end{setup}

\subsection*{Flows of Sobolev type vector fields}

We recall some known results on flows of time dependent Sobolev vector fields. 

For vector fields we use a specialised notation for the spaces of bundle sections discussed in the last section.
\begin{defn}
Let $M$ be a $d$-dimensional manifold (not necessarily compact, possibly with non-empty boundary). For the tangent bundle $TM$ we shall  write $\mathfrak{X}^s (M) \coloneq H^s(TM) = H^s_{\id} (M,TM)$ for the space of $H^s$-vector fields. Define the subspace $\mathfrak{X}^s_\mu (M)$ of divergence free vector fields (with respect to some chosen Riemannian volume form $\mu$).

If $M$ is a manifold with (smooth) boundary, we shall always assume that vector fields in $\mathfrak{X}^s(M)$ and $\VFsmu[s]$ are tangential to the boundary $\partial M$, i.e.\ $X(m)\in T_m (\partial M)$. 
The condition is equivalent to the boundary condition
\begin{align}\label{boundary_condition}
g(\vec{n}(k),X(m)\rangle =0, \quad \text{ for all } m\in \partial M,
\end{align}
where $\vec{n}$ denotes the (outward pointing) normal vector field to the boundary $\partial M$ (\cite[Proposition 15.33]{MR2954043}). Since $s>d/2$ the $H^s$ topology is finer than the compact open topology. In particular, point evaluations are continuous (cf.\ Lemma \ref{lem:eval}). Thus the tangential vector fields form a closed subspace of $\mathfrak{X}^s(M)$ and inherit the topological properties of the ambient space we care about. To keep the notation simple we will suppress the additional condition in the notation whenever $M$ has boundary.
\end{defn}

Let again $K$ be a $d$-dimensional compact Riemannian manifold (possibly with smooth boundary) and $I$ a compact interval containing $0$. Following \cite{MR3635359} we define the following.  

\begin{defn}\label{defn:flows}
Fix an integer $s > d/2 +1$ and consider a time dependent vector field $X \in L^1 (I, \mathfrak{X}^s(K))$. A map $\varphi \colon I \times K \to K$ is the \emph{pointwise flow} of $X$ if $\varphi(0,k)=k,\ \forall k \in K$ and for each pair $(t,k) \in I \times K$ there exists a chart $(U,\kappa)$ around $x$ and a chart $(V,\psi)$ around $\varphi (t,x)$ such that with $Y \coloneq T\psi \circ X \circ \psi^{-1}$ and $\eta \coloneq \psi \circ\varphi  \circ\kappa^{-1}$ the flow equation
$$\eta(s,\ell) = \eta(t,\kappa(k)) + \int_t^s Y(\tau , \eta(\tau, \ell)) \mathrm{d}\tau$$
holds for $(s,\ell)$ near $(t,\kappa(k))$. If in addition $\varphi \in C(I,\Diff^s(K))$, i.e., $\varphi$ is a continuous $\Diff^s(K)$-valued curve, we call $\varphi$ the $\Diff^s(K)$\emph{-valued flow} of $X$.
\end{defn}
Note that for (sufficiently)\, differentiable vector fields, the notion of flow in Definition \ref{defn:flows} coincides with the usual definition of a flow of a vector field.
Working in local charts one can establish the following.

\begin{prop}[{\cite[Theorem 5.8]{MR3635359}}]\label{prop:Hs-flows}
Let $s>d/2 + 1$ and $K$ be a manifold \textbf{without} boundary. Then $X \in L^1 (I,\mathfrak{X}^s(K))$ has a $\Diff^s(K)$-valued flow $\varphi_X$ and for each $t \in I$, one obtains a continuous  map 
$$\mathrm{Fl}_t \colon L^1 (I,\mathfrak{X}^s(K)) \to \Diff^s (K), \quad X \mapsto \varphi_X (t).$$
\end{prop}

As we wish to treat manifolds with smooth boundary, Proposition \ref{prop:Hs-flows} is not quite sufficient for our purposes.
However, the exact analog for $K$ with smooth boundary of Proposition \ref{prop:Hs-flows} does not seem to be available in the literature. 

\begin{rem}\label{rem:Hs-flow-reg}
In \cite{EM70}, a similar statement to Proposition \ref{prop:Hs-flows} is established for continuous in time vector fields on manifolds with boundary for $s > d/2 +2$. While continuous time dependence would be unproblematic for our arguments, we would prefer to press the regularity of the Sobolev function as low as possible. As announced in \cite{EaFaM72}, the result also holds for the lower regularity $s > d/2 +1$ and can indeed be found for the special case of $K$ being an embedded domain with smooth boundary in \cite[Appendix A]{BaB74}. However, we were not able to find the specific result (neither for $L^1$ in time, nor continuous in time vector fields) on a general manifold in the literature. Reviewing the arguments of \cite{MR3635359} it is easy to see that they can be adapted for the case of a manifold with boundary. We refrain from working this out in the present paper, as the technical arguments will not be different from the ones in \cite{MR3635359}. Details for the proof of Proposition \ref{prop:Hs-flows} for $K$ with smooth boundary will be given elsewhere. 
\end{rem}
\end{appendix}

\addcontentsline{toc}{section}{References}
\bibliography{EM_SPDE}
\end{document}